\documentclass[11pt]{article}
\usepackage{graphicx}
\usepackage{microtype}
\usepackage{hyperref}
\usepackage{amssymb}
\usepackage{amsmath}
\usepackage{amsthm}
\usepackage{nccmath}
\usepackage{dsfont}
\usepackage{algorithm}
\usepackage{algpseudocode}
\usepackage{subcaption}  
\usepackage{enumitem}
\usepackage{soul}
\usepackage{url}
\usepackage{natbib}
\usepackage{indentfirst}
\usepackage{multirow} %
\usepackage{makecell} %
\usepackage{longtable}
\usepackage{setspace} %
\usepackage{comment}
\usepackage{authblk}

\usepackage{xcolor}
\usepackage{threeparttable}
\DeclareMathOperator*{\argmin}{arg\,min}
\theoremstyle{plain}
\newtheorem{theorem}{Theorem}
\newtheorem{lemma}{Lemma}
\newtheorem{proposition}{Proposition}

\setlist[itemize,enumerate]{
  itemsep=0ex,      %
  parsep=0ex,       %
  topsep=3pt,       %
  partopsep=0pt,    %
  leftmargin=\parindent
}

\newtheoremstyle{myitalicremark}
  {3pt}      %
  {3pt}      %
  {\itshape} %
  {}         %
  {\bfseries}%
  {.}        %
  {5pt}      %
  {}         %
\theoremstyle{myitalicremark}
\newtheorem{remark}{Remark}
\newtheorem{example}{Example}

\newtheorem{assumption}{Assumption}

\allowdisplaybreaks

\newcommand{\beginsupplement}{
  \setcounter{section}{0}
  \renewcommand{\thesection}{\Alph{section}} %
  \renewcommand{\thesubsection}{\thesection\arabic{subsection}} %
  \renewcommand{\thesubsubsection}{\thesubsection\arabic{subsubsection}} 
  
  \setcounter{figure}{0}
  \setcounter{table}{0}
  \setcounter{equation}{0}
  
  \renewcommand{\thefigure}{\thesection\arabic{figure}} %
  \renewcommand{\thetable}{\thesection\arabic{table}}   %
  \renewcommand{\theequation}{\thesection\arabic{equation}} %
  
  \counterwithin{figure}{section}
  \counterwithin{table}{section}
  \counterwithin{equation}{section}
  
  \begin{center}
    \Large\bfseries Supplementary Material
  \end{center}
}

\newenvironment{keyword}[1][Keywords]{%
  \par\vspace{6pt}%
  \begin{list}{}{%
      \setlength\leftmargin{2.5em}
      \setlength\rightmargin{2.5em}
      \setlength\topsep{0pt}
      \setlength\parsep{0pt}
      \setlength\listparindent{0pt}%
    }%
    \item\relax
    \noindent\textbf{#1:}\space%
    \ignorespaces
}{%
  \end{list}%
  \par\vspace{6pt}%
}

\usepackage[margin=1in]{geometry}
\selectfont

\title{Harnessing Individual Motivation for Collective Efficiency: A
Mechanism-Driven Distributed Optimization Method}

\author[1]{Dongwei Xie}
\author[1]{Xuhao Wang}
\author[1]{Yujie Tang}
\author[1]{Jie Song}

\affil[1]{School of Advanced Manufacturing and Robotics, Peking University, No.5 Yiheyuan Road, Haidian District, 100871, Beijing, People’s Republic of China }

\date{}

\begin{document}

\maketitle

\begin{abstract}
In industrial scenarios involving multi-agent collective decision-making, centralized decision-making may not be admissible due to restrictive access to individual local information, while the conflicts between participants' self-interest and global performance may also impede collaborative distributed decision-making. This paper proposes a mechanism-driven distributed decision-making method, wherein incentives are employed and designed to motivate participants to collaborate in a distributed fashion even though each participant's decision is driven primarily by self-interest. Focusing on optimization problems with coupled objective functions and coupled constraints, we design a distributed optimization algorithm tailored for this class of problems and provide guarantees for its convergence. Furthermore, we design two incentive mechanisms, the shadow pricing mechanism and the Vickrey-Clarke-Groves mechanism, and demonstrate that participants are willing to engage in distributed collaboration under these mechanisms. The mechanism drives the execution of the distributed algorithm, and the optimal result of distributed computation guides the determination of incentives in the mechanism, both of which are interrelated to form a closed loop. Finally, numerical experiments illustrate the effectiveness of the proposed algorithm and mechanisms.
\end{abstract}


\section{Introduction}
With the advancement of industrial intelligence, optimal decision-making problems have grown in scale and complexity, accompanied by an increasing number of participating individuals or entities. Achieving group-level efficient decision-making typically requires access to individual-level information \citep{mann2018collective,hassidim2021limits}. However, in numerous industrial contexts, participating individuals tend to form temporary coalitions rather than long-term collectives. Consequently, participants often resist disclosing private information to a centralized administrator/aggregator to safeguard their interests, rendering traditional centralized decision-making frameworks untenable. For example, in supply chain management, multiple enterprises often collaborate through joint replenishment to reduce costs. However, the demand rates and inventory holding costs of each enterprise constitute private information, making centralized decision-making unachievable \citep{guler2017design}. Similarly, in distributed energy systems, the energy sector faces challenges characterized by fragmented ownership structures. Here, the system operator struggles to reconcile individual consumers' interests through centralized approaches due to limited or costly access to consumers' private data, ultimately resulting in suboptimal outcomes that fail to align with all parties' priorities \citep{sun2024break}.

Distributed optimization methods offer a viable solution to problems involving reluctant information-sharing participants. Ideally, by decomposing the global problem into local sub-problems, each participant simply solves its sub-problem independently while circumventing disclosure of private information, eliminating the need for an administrator to make a centralized decision \citep{yang2019survey}. However, such an ideal decomposition of global objectives into non-conflicting local ones is rare, as the scarcity of resources in real-world scenarios inevitably creates conflicting interests among participants \citep{homberger2017generic}. This results in inherently coupled individual decision-making relations that cannot be straightforwardly decomposed. For example, in resource transportation problems, when multiple vehicles carrying different resources utilize the same route concurrently, congestion will naturally occur which diminishes efficiency and escalates costs, thereby creating coupling among participants' decisions \citep{guo2024penalty}. In supply chain management, warehouses face storage capacity limitations, and late-arriving suppliers encounter progressively limited space availability, creating positional disadvantages that may force sub-optimal storage solutions \citep{chen2025primal}. In cloud computing platforms, the computational demand needs to be satisfied under capacity limitations, and low-priority tasks may have to be abandoned to maintain system stability when the demand exceeds a certain threshold \citep{liu2025efficient}. 
To handle such coupling in distributed decision-making, various distributed optimization algorithms have been proposed for different problem setups. However, existing works primarily address a single form of coupling---whether manifested in objective functions \citep{nedic2009distributed} or constraints \citep{falsone2017dual}---or transform one type of coupling into another, but typically at the expense of algorithm efficiency \citep{li2020distributed}; see Section~\ref{section:literature_review} for more detailed literature review. Few distributed frameworks are specifically designed to tackle optimization problems characterized by the coexistence of multiple coupled relations.

Additionally, while distributed optimization methods enable participants to collaboratively optimize group-level performance \citep{yang2019survey}, whether the final optimization outcome is acceptable to each participant is often overlooked. In reality, when participants have decision-making power, they are often driven by self-interest to maximize their own benefits rather than collective welfare  \citep{lang2016design}. \cite{xiao2017learning} posited that, within self-interested environments, noncooperative game dynamics emerge among individuals, necessitating the design of reasonable mechanisms to incentivize collaborative behavior. \cite{meng2023selfish} also pointed out that self-interested behavior by decision-makers may prevent the achievement of collective consensus. 
To address this challenge, researchers have proposed designing incentive mechanisms to counteract self-interested tendencies and promote consensus formation. However, existing relevant works on mechanism design predominantly address a single form of coupled relations. For instance, marginal pricing schemes in power grids manage coupled constraints arising from transmission line capacities \citep{alcantara2024optimal}, while game-theoretic mechanisms focus exclusively on couplings within objective functions \citep{tang2025biform}; see Section~\ref{section:literature_review} for more details. Mechanism designs capable of simultaneously addressing coexisting coupled relations remain underinvestigated.

To address these gaps, this paper aims to design a distributed optimization algorithm capable of efficiently handling problems with coupled objective functions and constraints, and develop incentive mechanisms that align participants' self-interest with group-level performance.  Specifically, the major contributions of this paper encompass the following:
\begin{itemize}[leftmargin=\parindent]
    \item A distributed optimization algorithm, called \emph{Consensus-Tracking-ADMM}, is designed to address problems with highly coupled objective functions and coupled constraints simultaneously. The coupled objective functions are decoupled through consensus on local copies of primal variables, while the coupled constraints are handled by distributed tracking of dual variables' gradients. Moreover, convergence guarantees are established to rigorously justify our algorithm's effectiveness.
    \item Two distinct mechanisms are designed: the shadow pricing mechanism (unit-price signal guidance) and the VCG mechanism (total-price signal guidance). Analysis of their properties demonstrates that, under the guidance of these mechanisms, self-interested participants are willing to engage in distributed coordination. This willingness arises as the optimal solution for participants is identical in both the noncooperative game scenario and the distributed coordination scenario.

    Note that in our proposed framework, the mechanism drives the execution of distributed algorithms, while the optimal solution produced by distributed computation guides the determination of incentives in the mechanism, thereby forming an interrelated closed loop.
    
    \item Numerical experiments demonstrate that our proposed algorithm enjoys faster convergence compared to several existing algorithms on test cases with variable scales. Meanwhile, we observe that the incentives received by the participants under each mechanism are task-dependent, suggesting that neither mechanism strictly dominates the other from the participants' perspective.
\end{itemize}

Compared to existing literature dealing with distributed coupled relationships [e.g., \citet{falsone2017dual,li2020distributed}], our algorithm can handle both coupled objective functions and coupled constraints. In contrast to traditional algorithm-centric distributed approaches \citep{yang2019survey}, our study highlights the mechanism-driven distributed approach, motivated by the observation that the absence of proper incentives for participants within transient coalitions may hinder the achievement of the desired outcomes of distributed coordination. Compared with the mechanisms in \cite{alcantara2024optimal,tang2025biform}, our mechanisms can cope with both coupled objective functions and coupled constraints, with theoretical guarantees that the proposed mechanisms comply with the desirable properties.

The structure of this paper is organized as follows: Section 2 presents the literature review of related studies. Section 3 describes the problem and the mathematical formulation. Section 4 introduces our proposed algorithm, \textit{Consensus-Tracking-ADMM}, along with its convergence guarantees. Section 5 proposes two incentive mechanisms to rationalize distributed approaches. Section 6 evaluates the performance of our algorithm through numerical experiments and compares the two incentive mechanisms. Finally, we conclude the paper and provide an outlook in Section 7.

\section{Literature review}
\label{section:literature_review}

Distributed optimization has emerged as a powerful technique for coordinating large-scale multi-agent systems, and has gained significant traction due to wide applicability in, e.g., resource allocation, formation control, sensor network estimation, and multi-agent reinforcement learning \citep{yi2016initialization, pu2020push}. 
In many practical situations, each agent not only makes decisions subject to its constraints, but also needs to be associated with other agents on limited resources to jointly satisfy certain constraints, which is called a constraint-coupled problem in \cite{falsone2020tracking}. Methods for distributed optimization with coupled constraints typically fall into three categories: dual-decomposition-based algorithms, consensus-based algorithms, and hybrid algorithms. Dual-decomposition-based algorithms are suitable for problems that are easily separable in the dual domain \citep{terelius2011decentralized}, enforcing coupled constraints through dual variables. Early algorithms were predominantly based on dual subgradient methods for solving separable problems with linear constraints \citep{necoara2015linear}, while modern algorithms often rely on the ADMM (Alternating Direction Method of Multipliers) framework \citep{mateos2010distributed}, which can handle more complex constraints and typically exhibit faster convergence \citep{chang2016proximal, zhang2020distributed}. For large-scale problems, \cite{zheng2022parallel} proposed dual block decomposition combined with gradient-based algorithms to reformulate the problem. Meanwhile, consensus-based approaches enable each agent to estimate global results by weighting the local decision variables of its neighbors in the update \citep{tsitsiklis2003distributed}. \cite{zhu2011distributed} proposed a classical distributed consensus-based primal-dual subgradient algorithm for handling coupled constraints. Based on this, \cite{su2021distributed} incorporated a perturbation approach to cope with more general constraints and non-separable global cost functions. Starting from parallel ADMM, \cite{falsone2020tracking} embedded a dynamic average consensus protocol in the proposed tracking-ADMM method, which utilizes tracking techniques to capture constraint violations. For large-scale linear programs, \cite{el2023consensus} employed the Dantzig-Wolfe decomposition and designed a consensus-based fully distributed algorithm to handle constraint matrices. Additionally, some algorithms leverage both dual decomposition and consensus techniques, which we categorize as hybrid algorithms. \cite{mateos2016distributed} proposed a distributed method that transforms coupled constraints into a saddle-point problem and incorporates consensus-based ideas. \cite{falsone2017dual} utilized dual decomposition combined with average consensus to update the primal variables. \cite{li2020distributed} further enhanced the algorithm's performance by simultaneously conducting consensus and performing approximate proximal updates on both primal and dual variables. In this paper, we leverage the ideas of consensus-based algorithms and develop the \emph{Consensus-Tracking-ADMM} algorithm tailored to handle both coupled objective functions and coupled constraints.

A mechanism refers to a framework where a designer strategically structures rules to guide participants toward actions aligned with the designer's objectives, by embedding the decision-making process within a game such that the outcome of the game converges to the designer's intended goal \citep{pavan2014dynamic}. Classic mechanism design classes encompass auctions, voting, matching, and pricing schemes. This study specifically focuses on the pricing mechanisms in guiding self-interested individuals to engage in decentralized coordination, thereby enabling efficient resource allocation. One common approach is \textit{dynamic pricing} mechanisms, which adjust prices in real-time based on actual conditions. For instance, \cite{egri2013distributed} utilized such mechanisms to address coordination challenges in supply chain networks comprising multiple retailers and a single supplier. \cite{hu2016pricing} employed such mechanisms to effectively coordinate charging behaviors among multiple electric vehicles. 
Another widely used mechanism is the \textit{Vickrey-Clarke-Groves} (VCG) mechanism, which determines the incentive to a participant by calculating the difference in marginal contributions created before and after each participant's involvement in the game. Although the VCG mechanism originates from an auction, it is frequently employed as a pricing instrument. For instance, \cite{qian2021automatically} leveraged this mechanism to enable distributed renewable resources and electric vehicles to trade power in local energy markets. Additionally, some mechanisms are designed for domain-specific applications. In electricity market contexts, \textit{locational marginal pricing} (LMP) constitutes the predominant pricing mechanism. By calculating nodal clearing prices based on transmission constraints between network nodes, LMP optimally guides distributed resource scheduling \citep{hamoud2004assessment}. This framework represents the most extensively implemented pricing mechanism in the U.S. and European electricity markets. Within carbon markets, commonly used mechanisms include \textit{cap-and-trade} policies where regulators coordinate participants through emission caps and allowance adjustments, thereby achieving ultimate objectives of carbon emission reduction \citep{feng2021altruistic,pun2025cap}. In this paper, we develop a shadow pricing mechanism guided by the concept of LMP and a tailored VCG mechanism adapted to the problem structure based on the traditional VCG.

\section{Problem formulation}\label{sec:Model}

In this paper, we consider the scenario where a group of participants needs to accomplish a common task. Each participant possesses private information regarding their operational capabilities, which remains non-transparent to others. To complete the task, some costs will be incurred on each participant; it is worth noting that the costs of participants may interact with each other, e.g., the congestion cost incurred when the same common resource is occupied by multiple individuals during the task.

We assume the existence of a coordinating administrator for the group of participants, whose primary objective is to achieve maximal group-level efficiency, trying to guide the individuals' decisions to minimize the total operational costs. On the contrary, the participants will strive to maximize their payoffs. This fundamental conflict arises from misaligned objectives between collective efficiency and individual interests. In the following, we provide a detailed mathematical modeling of the problem.

\subsection{The optimization problem}
\label{subsection:formulation_optimization}

Let $N$ be the number of participants. Each participant is associated with a local decision variable $\boldsymbol{x}_i \in \mathbb{R}^{n_i}$ for the task, and we let $\boldsymbol{x}={\rm col}(\boldsymbol{x}_1,\dots,\boldsymbol{x}_N) \in \mathbb{R}^{n_1+\dots+n_N}$ denote the decision of all participants. After implementing the decisions and completing the task, each participant incurs a local cost $f_i(\boldsymbol{x})=f_i(\boldsymbol{x}_1,\ldots,\boldsymbol{x}_N)$. Note that $f_i:\mathbb{R}^{n_1+\cdots+n_N}\rightarrow\mathbb{R}$ is a function of all participants' decisions, allowing coupling in the local objectives such as the congestion costs incurred by sharing common public resources. Additionally, the participants' decisions will be subject to (i) local inequality constraints of the form $g_i(\boldsymbol{x}_i)\leq \boldsymbol{0}$, where $g_i: \mathbb{R}^{n_i} \rightarrow \mathbb{R}^{q_i}$ is a vector-valued function and the inequality is to be interpreted componentwise; (ii) a coupled constraint of the form $\sum_{i=1}^n A_i\boldsymbol{x}_i=\boldsymbol{d}$, where $\boldsymbol{d}\in\mathbb{R}^{n_0}$ is a given vector and $A_i\in\mathbb{R}^{n_0\times n_i}$ are given matrices. Specifically, the coupled constraint mandates that all participants must coordinate their efforts to achieve task completion. The goal is to minimize the total cost given by $f(\boldsymbol{x})=\sum_{i=1}^N f_i(\boldsymbol{x})$, formulated as the following optimization problem:
\begin{equation}\label{eq:Distributed}
\tag{$\mathcal{P}_1$}
     \begin{aligned}
         \min_{\boldsymbol{x}} \quad & f(\boldsymbol{x})=\sum\nolimits_{i=1}^N f_i(\boldsymbol{x}) \\
        {\rm s.t.} \quad \  &\sum\nolimits_{i=1}^N A_i\boldsymbol{x}_{i} =\boldsymbol{d}, \quad \leftrightarrow \quad \lambda, \\
        & \ g_i(\boldsymbol{x}_i) \leq \boldsymbol{0}, \quad \leftrightarrow \quad \alpha_i, \quad \forall i=1,\ldots,N. 
    \end{aligned}
 \end{equation}
Here, $\lambda$ and $\alpha_1,\dots,\alpha_N$ denote the dual variables corresponding to the constraints. For notational clarity, we let $\boldsymbol{\alpha}={\rm col}(\alpha_1,\dots,\alpha_N)$, $\Omega_i = \{\boldsymbol{x}_i \ |\ g_i(\boldsymbol{x}_i) \leq \boldsymbol{0}\}$, and $\Omega=\Omega_1 \times \cdots \times \Omega_N$.
We assume that the functions $f_i$ and $g_i$ contain private information that the $i$th participant does not wish to disclose.

\begin{remark}
    In the optimization problem \ref{eq:Distributed}, the coupled constraint is formulated as an equality constraint. For inequality constraints, e.g., $\sum_{i} A_i\boldsymbol{x}_{i} \leq \boldsymbol{d}$, a slack variable $\boldsymbol{s}_i$ can be introduced to convert the inequality into an equality constraint, $\sum_{i} (A_i\boldsymbol{x}_{i} +\boldsymbol{s}_i) = \boldsymbol{d}$. This slack variable is subsequently incorporated into the original decision variable, forming a new composite variable $\boldsymbol{z}_i={\rm col}(\boldsymbol{x}_i,\boldsymbol{s}_i)$. Concurrently, the constraint matrices will be appropriately modified to accommodate this variable augmentation, e.g., $\sum_{i} A_i'\boldsymbol{z}_{i} = \boldsymbol{d}$.
\end{remark}

The following standard assumptions will be imposed on the optimization problem \ref{eq:Distributed}.

\begin{assumption}\label{ass:const_set}
    For all $i=1,\dots,N$:
    \begin{enumerate}[label=(\roman*),leftmargin=24pt]
        \item $f_i(\boldsymbol{x})$ is a convex function of $\boldsymbol{x}$, and each component of $g_i(\boldsymbol{x}_i)$ is a convex function of $\boldsymbol{x}_i$.
        \item The feasible region of~\ref{eq:Distributed} is nonempty, and each local constraint set $\Omega_i$ is compact.
    \end{enumerate}
\end{assumption}

\begin{assumption}\label{ass:saddle}
    For the Lagrangian function
\[\mathcal{L}(\boldsymbol{x},\lambda,\boldsymbol{\alpha}) = \sum\nolimits_i f_i(\boldsymbol{x}) - \lambda^{\rm T}\left(\sum\nolimits_i A_i\boldsymbol{x}_i - \boldsymbol{d}\right)+ \sum\nolimits_i\alpha_i^{\rm T}g_i(\boldsymbol{x}_i)
\]
associated with optimization problem \ref{eq:Distributed},
    there exists a saddle point $(\boldsymbol{x}^*, \lambda^*, \boldsymbol{\alpha}^*)$ satisfying the inequality
    \begin{equation*}
      \mathcal{L}(\boldsymbol{x}^*, \lambda, \boldsymbol{\alpha}) \leq  \mathcal{L}(\boldsymbol{x}^*, \lambda^*, \boldsymbol{\alpha}^*) \leq \mathcal{L}(\boldsymbol{x}, \lambda^*, \boldsymbol{\alpha}^*),
      \qquad\forall \boldsymbol{x}\in\mathbb{R}^{n_1+\cdots+n_N},\lambda\in\mathbb{R}^{n_0},\boldsymbol{\alpha}\geq\boldsymbol{0}.
    \end{equation*}
\end{assumption}

We demonstrate the applicability of the optimization problem \ref{eq:Distributed} through two examples abstracted from real-world applications.
\begin{example}\label{example:allo}
    The first example is a commodity allocation and transportation problem adapted from \cite{roughgarden2010algorithmic}. Suppose $N$ suppliers and $M$ demanders are located at certain nodes of a transportation network given by the directed graph $\mathcal{G}_{\mathrm{tra}}=(\mathcal{V},\mathcal{E})$, where each edge $e\in\mathcal{E}$ represents a one-way road connecting two nodes. There are in total $K$ types of infinitely divisible commodities, $m_{ik}$ denotes the quantity of type-$k$ commodity stored at supplier $i$, and $l_{ij}$ denotes the transportation capacity available from supplier $i$ to demander $j$. To fulfill demander $j$'s requirement of $d_{jk}$ units of type-$k$ commodity, suppliers need to allocate and transport the commodities. Assume that for each supplier-demander pair $(i,j)$, the set of available paths in the transportation network is $\mathcal{P}_{ij}$. Thus, supplier $i$ needs to determine the decision variable $x_{ijkr}$, representing the quantity of type-$k$ commodity to be transported from supplier $i$ to demander $j$ via path $r\in\mathcal{P}_{ij}$. We assume that a private transportation cost $c_{ie}$ will be incurred by supplier $i$ per unit flow transported along edge $e$. In addition, each edge $e\in\mathcal{E}$ may be shared by multiple paths, and congestion costs arise when multiple suppliers share the same road for transportation. We let $c_e(q_e(\boldsymbol{x}))$ be the congestion cost incurred for each unit of traffic along edge $e$, where $q_e(\boldsymbol{x})=\sum_{i,j,k}\sum_{r\in \mathcal{P}_{ij}:e\in r} x_{ijkr}$ represents the total traffic on edge $e$. We can formulate the following optimization problem to find the optimal allocations $x_{ijkr}$:
    \begin{equation}\label{eq:example-tran}
     \begin{aligned}
         \min_{\boldsymbol{x}=(x_{ijkr})} \quad & \sum\nolimits_{i} \left(\sum\nolimits_{j,k}\sum\nolimits_{r\in\mathcal{P}_{ij}}x_{ijkr}\cdot
         \sum\nolimits_{e\in r}c_e(q_e(\boldsymbol{x}))
         +
         \sum\nolimits_{j,k}\sum\nolimits_{r \in \mathcal{P}_{ij}}x_{ijkr} \cdot \sum\nolimits_{e \in r}c_{ie}\right) \\
        {\rm s.t.} \quad \  &\sum\nolimits_{i}\sum\nolimits_{r\in\mathcal{P}_{ij}} x_{ijkr} =d_{jk}, \qquad\forall j,k,\\
        & x_{ijkr}\geq 0,\ \ 
        \sum\nolimits_j\sum\nolimits_{r\in\mathcal{P}_{ij}} x_{ijkr}\leq m_{ik},\ \ \sum\nolimits_k\sum\nolimits_{r\in\mathcal{P}_{ij}} x_{ijkr}\leq l_{ij},\qquad\forall i,j,k\text{ and }r\in\mathcal{P}_{ij}.
    \end{aligned}
 \end{equation}
By letting $\mathbf{1}_{\{e \in r\}}$ denote the indicator function that takes the value $1$ if edge $e$ belongs to path $r$, and $0$ otherwise, and defining the local objective $f_i$ to be
\begin{equation*}
      f_i(\boldsymbol{x})= \sum\nolimits_{e \in \mathcal{E} }
      \kappa_{i,e}\cdot
       q_e(\boldsymbol{x}) \cdot c_e(q_e(\boldsymbol{x}))+\sum\nolimits_{j,k}\sum\nolimits_{r\in\mathcal{P}_{ij}}
         \sum\nolimits_{e\in r}c_{ie}\cdot x_{ijkr},
\end{equation*}
where
\begin{equation*}
         \kappa_{i,e}=\frac{\sum\nolimits_{j,k}\sum\nolimits_{r\in\mathcal{P}_{ij}} \mathbf{1}_{\{e \in r\}}}{\sum\nolimits_{i',j,k}\sum\nolimits_{r\in\mathcal{P}_{i'j}}\mathbf{1}_{\{e \in r\}}},
\end{equation*}
we see that problem \eqref{eq:example-tran} can be reformulated in the form of \ref{eq:Distributed}. The convexity of each $f_i$ will be guaranteed if $x\mapsto x\cdot c_e(x)$ is a convex function for each $e\in\mathcal{E}$ (see Supplementary Material \ref{sup_mat:Notes} for detailed justification).
\end{example}

\begin{example}\label{example:EV}
    The second example is a distributed electric vehicle (EV) aggregator optimization problem adapted from \cite{rivera2016distributed}. Suppose $N$ EVs coordinated by an aggregator seek to optimize their charging profiles. We let $\boldsymbol{x}_i={\rm col}(x_{i,1},\dots,x_{i,T})$ denote the charging profile of the $i$th EV. Each EV's charging incurs a local cost $r_i(\boldsymbol{x}_i)$, while the aggregated charging profile $\sum_i\boldsymbol{x}_i$ also incurs some cost $r_a(\sum_i \boldsymbol{x}_i)$ on the aggregator. We can then formulate the following optimization problem:
    \begin{equation*}\label{eq:example-ev}
     \begin{aligned}
         \min_{\boldsymbol{x}_1,\dots,\boldsymbol{x}_N} \quad & r_a\left( \sum\nolimits_i \boldsymbol{x}_i\right) + \sum\nolimits_i r_i(\boldsymbol{x}_i) \\
        {\rm s.t.} \quad \  &  \boldsymbol{l} \leq \sum\nolimits_{i} \boldsymbol{x}_{i} \leq \boldsymbol{m}, \\
        &  \boldsymbol{x}_i \in \mathbb{X}_i, \  \forall i,
    \end{aligned}
 \end{equation*}
where each $\boldsymbol{x}_i \in \mathbb{X}_i$ imposes constraints on individual charging, while $\boldsymbol{l} \leq \sum\nolimits_{i} \boldsymbol{x}_{i}\leq \boldsymbol{m}$ represents constraints on aggregated charging. By letting $f_i(\boldsymbol{x}) = \frac{1}{N}r_a\left( \sum\nolimits_i \boldsymbol{x}_i\right) + r_i(\boldsymbol{x}_i)$, we see that \eqref{eq:example-ev} is a special case of \ref{eq:Distributed}.
\end{example}

For the distributed optimization problem \ref{eq:Distributed}, there are two main challenges in solving it.
\begin{itemize}
    \item First, each local objective function $f_i$ not only depends on the decision variable $\boldsymbol{x}_i$ of participant $i$ but also on the decision variables of other participants, which prevents participant $i$ from making independent decisions. 
    \item Second, decentralized decisions need to be made by the participant under coupled constraints. As these constraints cannot be easily decomposed into independent sub-constraints, coordination of the participants’ decisions is required.
\end{itemize}

Addressing both challenges simultaneously is not straightforward, and existing methods mostly transform coupled objective functions into coupled consensus constraints, thereby reducing the two couplings to a single form \citep{li2020distributed}. However, the naive transformation approach fails to account for the intrinsic characteristics of different coupled constraints---particularly the decomposability of consensus constraints---resulting in efficiency loss. Meanwhile, numerical experiments demonstrate that the use of such an approach applied to existing algorithms exhibits inferior convergence performance. In section \ref{sec:Algorithm}, we will introduce our \emph{Consensus-Tracking-ADMM algorithm} (Algorithm \ref{alg:TCADMM}) designed to effectively address the above challenges. Distinct from mere transform approaches, our innovation lies in separately addressing the two types of coupling relationships. For coupling within the objective function, we decouple the relationships by constructing auxiliary variables based on the primal variables and integrating them into each participant's sub-optimization process, while performing consensus updates on the auxiliary variables. Regarding the coupled constraint, we adopt the tracking methodology \citep{falsone2020tracking} wherein each participant progressively estimates $\eta_i$ as a local tracker of the global average constraint violation, $(\sum\nolimits_i A_i\boldsymbol{x}_i -\boldsymbol{d})/N$, and updates this tracker through consensus among participants. Numerical experiments demonstrate that our algorithm significantly enhances convergence performance compared to existing methods. Additionally, we will provide a detailed discussion on establishing convergence guarantees for the optimization problem \ref{eq:Distributed}. 

\subsection{Individual interests vs collective efficiency}

Although the distributed optimization approach allows group-level efficient decision-making in the presence of individual information opacity, it introduces a \emph{new challenge}: Rational participants pursuing independent decision-making prioritize self-interest maximization over collective efficiency maximization. The discrepancy between collective and individual objectives causes participant decisions to deviate from outcomes that maximize collective efficiency. To achieve a collective efficiency solution under a distributed approach, we assume the existence of a virtual administrator responsible for providing incentives for the participants of the distributed optimization task. Each participant will receive a payment $\Pi_i$ as the incentive after completing the task, and the net cost of the $i$th participant is then given by
\[
u_i=f_i(\boldsymbol{x}_i,\boldsymbol{x}_{-i})-\Pi_i,
\]
where $\boldsymbol{x}_{-i}$ consists of all the entries in $\boldsymbol{x}$ excluding $\boldsymbol{x}_i$. The fundamental question is how to assign the values of $\Pi_i$, so that the group-level optimal solution will align with each participant's interest. In Section \ref{sec:Mechanisms}, we will present two $\Pi_i$ designs that align with our objectives: one induces participants' actual decision-making, while the other elicits truthful reporting of their cost functions.

We give further remarks on the virtual administrator who is responsible for disbursing the payments~$\Pi_i$. Here, we shall assume that participant-guiding payments remain negligible compared to task-completion revenues, and consequently minimizing the virtual administrator's payoff is not incorporated in our formulation. For instance, virtual power plants in electricity markets can coordinate consumption via pricing rules. By subsidizing energy storage systems, they increase green power consumption and reduce curtailment penalties, with savings significantly outweighing incentive costs \citep{wang2021evaluation}. Similarly, in supply chain management, managers can restructure profit allocation through contracts to benefit all parties. When manufacturers use buyback contracts as incentives to lower distributors' financing barriers, their increased procurement revenue will outweigh potential buyback costs \citep{zhu2025strategic}.

So far, we have introduced a framework to address the optimization problem under conditions of opaque individual information in a distributed fashion. In Section~\ref{sec:Algorithm}, we shall introduce an algorithm tailored to solve distributed optimization problems with coupled objectives and constraints. In Section~\ref{sec:Mechanisms}, we shall design the payments $\Pi_i$ that enable the distributed approach to operate cohesively while respecting individual rationality. We shall see that the determination of optimal payments $\Pi_i$ depends on the optimal solution of the underlying optimization problem, and consequently, the algorithm part and the mechanism part form an interconnected system.

\section{Algorithm}\label{sec:Algorithm}

In this section, we propose a distributed algorithm, \textit{Consensus-Tracking-ADMM}, for solving the optimization problem \ref{eq:Distributed}, and also establish its convergence guarantee. 

\subsection{Preliminaries for algorithm design}

We first specify the communication structure among participants. Assume that the $N$ participants communicate via a bidirectional communication network, whose topology is represented by an undirected graph $\mathcal{G}_{\mathrm{com}}=(\{1,\dots, N\},\mathcal{E}_{\mathrm{com}})$. Each edge $(i, j)\in\mathcal{E}_{\mathrm{com}}$ represents a communication link between participant $i$ and participant $j$, and the undirectedness of $\mathcal{G}_{\mathrm{com}}$ requires $(i,j) \in \mathcal{E}_{\mathrm{com}}$ if and only if $(j, i) \in \mathcal{E}_{\mathrm{com}}$. Each edge $(i, j) \in \mathcal{E}_{\mathrm{com}}$ is also associated with a tunable weight $w_{ij}$, which will be used in our algorithm to weigh the importance participant $i$ assigns to the information received from participant $j$. We also assign $w_{ij}=0$ if $i\neq j$ and $(i, j) \notin \mathcal{E}_{\mathrm{com}}$. Let $W \in \mathbb{R}_{+}^{ N\times N}$ be the weight matrix whose $(i, j)$-th entry is $w_{ij}$, and we make the following assumption.

\begin{assumption}\label{ass:com_network}
The graph $\mathcal{G}_{\mathrm{com}}$ is connected, and the weight matrix $W$ satisfy the following conditions:
    \begin{enumerate}[label=(\roman*),leftmargin=24pt]
        \item $w_{ii}>0$ for all $i$. For $i \neq j$, we have $w_{ij} > 0$ if and only if $(i, j) \in \mathcal{E}_{\mathrm{com}}$.
        \item $W$ is a symmetric doubly stochastic matrix, i.e., $W=W^{\rm T}$ and $W\mathbf{1}_N = \mathbf{1}_N$. Furthermore, $W$ is positive semidefinite.
    \end{enumerate}
\end{assumption}
\begin{remark}
The connectivity assumption is standard in distributed optimization, which ensures that each participant can obtain information from any other participant in a finite number of communication steps. The assumptions on the weight matrix $W$ are also standard. Particularly, we can construct $W$ as a lazy Metropolis weights matrix, given by
   \begin{equation*}
       w_{ij}=\left\{
       \begin{array}{ll}
       \mfrac{1}{2} + \mfrac{1}{2\max\{{\rm deg}(i),{\rm deg}(j)\}}, &  \text{\ if\ } i\neq j\text{ and }(i,j) \in \mathcal{E}_{\mathrm{com}}, \\
       0,  &  \text{\ if\ } i\neq j \text{\ and\ } (i,j) \notin \mathcal{E}_{\mathrm{com}}, \\
          1-\sum_{s\neq i} w_{is}, &    \text{\ if\ } i=j.
        \end{array}
       \right.
   \end{equation*} 
It can be verified that the lazy Metropolis weights matrix fulfills Assumption \ref{ass:com_network} \citep{olshevsky2017linear}.
\end{remark}

As mentioned in Section~\ref{subsection:formulation_optimization}, one of the challenges in designing distributed algorithms for solving \ref{eq:Distributed} is that the cost $f_i$ of each participant depends on the decisions of others, which prevents the objective function from being decomposed into mutually independent components directly. To handle this issue, we augment the decision variables for the participants and propose an equivalent reformulation of~\ref{eq:Distributed}. We let each participant maintain local copies of the decisions of others when making their own decisions. Specifically, we let $\boldsymbol{y}_i = {\rm col}(\boldsymbol{x}_{i,1},\boldsymbol{x}_{i,2},\dots,\boldsymbol{x}_{i}\dots, \boldsymbol{x}_{i,N})$ be the augmented decision variable for participant $i$, where $\boldsymbol{x}_{i,s}$ denotes participant $i$'s estimate of participant $s$'s original decision variable $\boldsymbol{x}_s$ ($s \neq i$). Then, we introduce the consensus constraints $\boldsymbol{y}_{i} = \boldsymbol{y}_{j}$ ($\forall\,i,j$), so that ultimately each participant's local copy $\boldsymbol{x}_{i,s}$ will be the same as the actual decision $\boldsymbol{x}_s$ made by the associated participant. As a result, we obtain the following reformulation of the original optimization problem~\ref{eq:Distributed}:

\begin{equation}\label{eq:Distri_Conv_copy}
\tag{$\mathcal{P}_2$}
     \begin{aligned}
        \min_{\boldsymbol{y}_1,\dots, \boldsymbol{y}_N} \quad & \sum\nolimits_{i=1}^Nf_i(\boldsymbol{y}_i) \\
         {\rm s.t.} \quad \ & \sum\nolimits_{i}  \tilde{A}_i\boldsymbol{y}_{i}  =\boldsymbol{d}, \\
         & \ \boldsymbol{y}_{i} = \boldsymbol{y}_{j}, \quad \forall i,j,\\
         & \  \boldsymbol{y}_{i}\in \tilde{\Omega}_i, \quad \forall i.
    \end{aligned}
 \end{equation}
Here the matrices $\tilde{A}_i$ and the sets $\tilde{\Omega}_i$ are given by
\[
\begin{aligned}
& \tilde{A}_i=\begin{pmatrix}
0_{n_0\times n_1} & \cdots & 0_{n_0\times n_{i-1}} & A_i & 0_{n_0\times n_{i+1}} & \cdots & 0_{n_0\times n_N}
\end{pmatrix}, \\ 
& \tilde{\Omega}_i=\mathbb{R}^{n_1}\times\cdots\times\mathbb{R}^{n_{i-1}}\times\Omega_i\times\mathbb{R}^{n_{i+1}}\times\cdots\times\mathbb{R}^{n_N}.
\end{aligned}
\]

\subsection{Algorithm description}

Our proposed Consensus-Tracking-ADMM algorithm is detailed in Algorithm~\ref{alg:TCADMM}.

\begin{algorithm}[ht]
    \caption{Consensus-Tracking-ADMM} \label{alg:TCADMM}
    \begin{algorithmic}[1]
        \Statex \textbf{Initialization:} 
        \State Choose parameters $\sigma$, $\rho$, and set $k=0$.  \label{line1_1}
        \For {$i=1,\dots,N$ (in parallel)}   \label{line1_2}
        \State Choose the initial value $\boldsymbol{y}_{i}(0)  \in \tilde{\Omega}_i$    \label{line1_3}
        \State Set $\lambda_i(0)=0$, \ $\eta_i(0)=\tilde{A}_i\boldsymbol{y}_i(0)-\boldsymbol{d}/N$ \label{line1_4}
         \State Set $v_i(0)=\text{deg}(i)^{-1} \sum\nolimits_{\{i,j\} \in \mathcal{E}} \frac{1}{2} (\boldsymbol{y}_{i}(0)+\boldsymbol{y}_{j}(0))$ \label{line1_5}
        \EndFor  \label{line1_6}       
        \Statex \textbf{Iterations:}
        \While {convergence is not reached}   \label{line1_7}
        \For {$i=1,\dots,N$ (in parallel)}  \label{line1_8}
        \State $\gamma_i(k)=\sum_{s\in \mathcal{N}_i} w_{is}\eta_s(k)$   \label{line1_9}
        \State $l_i(k)=\sum_{s\in \mathcal{N}_i} w_{is}\lambda_s(k)$  \label{line1_10}
        \State $\boldsymbol{y}_i(k+1)=\mathop{\mathrm{argmin}}\limits_{\boldsymbol{y}_i \in \tilde{\Omega}_i} \left\{f_i(\boldsymbol{y}_i)+\dfrac{\rho}{2} \text{deg}(i) \| \boldsymbol{y}_i-v_i(k)\|^2 \right.$ 
        \Statex \hspace{59mm} $ +\,l_i(k)^{\rm T}\tilde{A}_i\boldsymbol{y}_i+ \left. \dfrac{\sigma}{2} \| \tilde{A}_i\boldsymbol{y}_i-\tilde{A}_i\boldsymbol{y}_i(k)+\gamma_i(k)\|^2\right\}$  \quad \quad \quad $(\mathcal{S}_{i,k})$   \label{line1_11}
        \State $\eta_i(k+1)=\gamma_i(k)+\tilde{A}_i\boldsymbol{y}_i(k+1)-\tilde{A}_i\boldsymbol{y}_i(k)$  \label{line1_12}
        \State $\lambda_i(k+1)=l_i(k)+\sigma \eta_i(k+1)$  \label{line1_13}
        \State $\delta_i(k+1)= \boldsymbol{y}_i(k+1)-\frac{1}{2}\boldsymbol{y}_i(k)$ \label{line1_14}
        \State $v_i(k+1) = v_i(k) + \text{deg}(i)^{-1}\sum_{s\in \mathcal{N}_i}\delta_s(k+1)- \frac{1}{2}\boldsymbol{y}_i(k)$  \label{line1_15}
        \EndFor  \label{line1_16}
        \State  Set $k= k+1$  \label{line1_17}
        \EndWhile  \label{line1_18}
    \end{algorithmic}
\end{algorithm}

Algorithm \ref{alg:TCADMM} consists of two parts: the initialization part and the iteration part. 
The initialization part (line \ref{line1_1}-\ref{line1_6}) sets initial values for the variables that will be used in the algorithm. Parameters $\sigma$, $\rho$ are common knowledge shared by all participants when they participate in the task.

The iteration part consists of four steps: communicating tracking-related information (lines \ref{line1_9}--\ref{line1_10}), solving the sub-optimization problem (line \ref{line1_11}), updating quantities related to dual variables and constraint violations (lines \ref{line1_12}--\ref{line1_13}), and communicating consensus-related information (line \ref{line1_14}--\ref{line1_15}). The roles and rationales of these four steps are explained as follows:
\begin{itemize}
\item Lines \ref{line1_9}--\ref{line1_10} represent the communication process between participant $i$ and its neighbors for aggregating information on dual variables and constraint violations. Each communication transmits two indicators: $\eta_i(k)$ and $\lambda_i(k)$. These received variables are subsequently weighted and processed to generate $\gamma_i(k)$ and $l_i(k)$ as substitutes for the dual variable $\lambda$ and the average constraint violations.

\item In line \ref{line1_11}, the optimization problem $\mathcal{S}_{i,k}$ is solved to update the primal decision variable $\boldsymbol{y}_i$ for each participant, with the solution constrained to the convex set $\tilde{\Omega}_i$. This subproblem is indicative of the ADMM framework we employ for algorithm design. Note that the objective function of $\mathcal{S}_{i,k}$ contains quadratic terms $\rho/2\cdot\operatorname{deg}(i)\|\boldsymbol{y}_i-\nu_i(k)\|^2$ and $\sigma/2\cdot \|\tilde{A}_i\boldsymbol{y}_i-\tilde{A}_i\boldsymbol{y}_i(k)+\gamma_i(k)\|^2$, which, loosely speaking, result from the quadratic penalty terms in the augmented Lagrangian of~\ref{eq:Distri_Conv_copy}.
\item Lines \ref{line1_12}-\ref{line1_13} are updates to the tracking and dual variables, respectively. Under a centralized architecture, these variables would be publicly available information, with the update given by
\begin{subequations}
    \begin{align}
    & \eta(k+1) = \frac{1}{N}\left(\sum\nolimits_i \tilde{A}_i \boldsymbol{y}_i(k+1)-\boldsymbol{d}\right), \label{eq:track_signal}\\
     & \lambda(k+1)=\lambda(k) + \sigma \eta(k+1).
\end{align}
\end{subequations}
The variable $\eta(k)$ would quantify the average degree of constraint violation in the task. However, in the distributed framework, the values of $\eta(k)$ and $\lambda(k)$ cannot be passed as public information to each participant, so each participant must calculate local estimates of $\eta(k)$ and $\lambda(k)$, generating copies denoted as $\eta_i(k)$ and $\lambda_i(k)$ whose updates are given by
\begin{subequations}
    \begin{align}
    & \eta_i(k+1) =\sum\nolimits_{s\in \mathcal{N}_i} w_{is}\eta_i(k)+(\tilde{A}_i\boldsymbol{y}_i(k+1)-\boldsymbol{d}/N)-(\tilde{A}_i\boldsymbol{y}_i(k)-\boldsymbol{d}/N), \\
     & \lambda_i(k+1)=\sum\nolimits_{s\in \mathcal{N}_i} w_{is}\lambda_s(k) + \sigma \eta_i(k+1).
\end{align}
\end{subequations}
The variable $\eta_i(k)$ serves as a local tracking signal for equation \eqref{eq:track_signal}, while the variable $\lambda_i(k+1)$ is subsequently updated based on the value of $\eta_i(k+1)$. Through sufficient iterations of communication, the local copies of these participants will eventually converge to a consensus.
\item In line \ref{line1_14}, participant $i$ computes an intermediate variable $\delta_i(k)$ related to the solution $\boldsymbol{y}_i(k)$ of problem $\mathcal{S}_{i,k}$ for subsequent network information transmission. In line \ref{line1_15}, the update for auxiliary variable $v_i(k)$  differs slightly from variables $\eta_i(k)$ and $\lambda_i(k)$, as it utilizes neighbors' information with equal weighting (rather than weighted averaging). The auxiliary variable $v_i(k)$ can actually be seen as guidance for optimizing variable $\boldsymbol{y}_i(k)$, where consensus among variables $v_i(k)$ induces the corresponding consensus in variables $\boldsymbol{y}_i(k)$, ultimately causing the optimization variables $\boldsymbol{y}_i$ to satisfy the consensus constraints in optimization problem \ref{eq:Distri_Conv_copy}.
\end{itemize}

Each participant's subproblem in Algorithm \ref{alg:TCADMM} exhibits computational complexity that scales up with the problem dimension. To accelerate computation, we develop dimension-reduction techniques for the subproblem under a specialized structure (affine local constraints). Due to space limitations, implementation details are provided in Supplementary Material  \ref{sup_mat:Alg_impro}.

\subsection{Convergence analysis} \label{subsec:con_ana}
In this subsection, we analyze the convergence properties of Algorithm \ref{alg:TCADMM}. Due to space limitations, we only present an outline of our analysis with critical lemmas and propositions; the detailed proofs are given in Supplementary Material~\ref{sec:supplementary_proofs}.

To begin our analysis, we first decouple the consensus constraints $\boldsymbol{y}_{i}=\boldsymbol{y}_{j},\forall i,j$ in the problem \ref{eq:Distri_Conv_copy} by introducing the auxiliary variables $\boldsymbol{\theta}_{e}$, $e \in \mathcal{E}_{\mathrm{com}}$ and reformulating them as $\boldsymbol{y}_{i} = \boldsymbol{\theta}_e$ and $\boldsymbol{y}_{j} = \boldsymbol{\theta}_e$, $\forall e=(i,j) \in \mathcal{E}_{\mathrm{com}}$. This operation essentially introduces variables based on the network's topology to decouple the consensus constraints, enabling the resulting constraints to be allocated to participants' computations.

Next, we denote $\tilde{\boldsymbol{y}}={\rm col}(\boldsymbol{y}_1,\boldsymbol{y}_2,\dots,\boldsymbol{y}_N)$, $\tilde{\boldsymbol{\theta}}=(\boldsymbol{\theta}_{e}: e \in \mathcal{E})$, $\hat{A} = 
    \begin{pmatrix}
        \tilde{A}_1 & \tilde{A}_2 & \dots & \tilde{A}_N
    \end{pmatrix}$, and $\tilde{\Omega}=\tilde{\Omega}_1 \times \cdots \times \tilde{\Omega}_N$, so that the optimization problem \ref{eq:Distri_Conv_copy} can be expressed in a more compact form:
\begin{equation}\label{eq:Distri_ADMM}
\tag{$\mathcal{A}_0$}
    \begin{aligned}
         \min_{\tilde{\boldsymbol{y}} \in \tilde{\Omega},\, \tilde{\boldsymbol{\theta}}}  \quad & H(\tilde{\boldsymbol{y}}) \\
         {\rm s.t.} \quad \ \ & \  \hat{A}\tilde{\boldsymbol{y}} =\boldsymbol{d}, \quad \leftrightarrow \quad \lambda, \\
         & \ M_1 \tilde{\boldsymbol{y}}+ M_2 \tilde{\boldsymbol{\theta}} = 0, \quad \leftrightarrow \quad \mu , 
    \end{aligned}
\end{equation}
where $H(\tilde{\boldsymbol{y}})=\sum_i f_i(\boldsymbol{y}_i)$.

\begin{lemma}\label{lem:full_rank}
    The matrices $M_1$ and $M_2$ are of full column rank.
\end{lemma}

The full column rank properties of matrices $M_1$ and $M_2$ will be leveraged in the subsequent convergence proof. Then, we elaborate on how the iterative process of the optimal problem \ref{eq:Distri_ADMM} leads to the steps specified in Algorithm \ref{alg:TCADMM}. Given notations $\boldsymbol{\eta}(k)= {\rm col}(\eta_1(k),\eta_2(k),\dots,\eta_N(k))$, $\boldsymbol{\lambda}(k)= {\rm col}(\lambda_1(k),\lambda_2(k),\dots,\lambda_N(k))$, $\hat{A}_b = {\rm blkdiag}(\hat{A}_1,\hat{A}_2, \dots, \hat{A}_N)$, and $\mathcal{W}= W \otimes I_{n_0}$, we have the following proposition.

\begin{proposition}\label{pro:step_ADMM}
    Under Assumptions \ref{ass:const_set}--\ref{ass:com_network}, the iterations of Algorithm \ref{alg:TCADMM} can be equivalently reformulated as
    \begin{subequations}\label{eq:step_ADMM}
    \begin{align}
        \begin{split}\label{eq:step_ADMM_y}
            & \tilde{\boldsymbol{y}}(k+1) =  \argmin\limits_{\tilde{\boldsymbol{y}} \in \tilde{\Omega}} \left\{H(\tilde{\boldsymbol{y}})+\mu(k)^{\rm T}\left(M_1 \tilde{\boldsymbol{y}}+ M_2 \tilde{\boldsymbol{\theta}}(k)\right) + \frac{\rho}{2}\left\|M_1 \tilde{\boldsymbol{y}}+ M_2 \tilde{\boldsymbol{\theta}}(k)\right\|^2 \right. \\
        & \hspace{22mm} + \left(\mathcal{W}\boldsymbol{\lambda}(k)\right)^{\rm T}\hat{A}_b\tilde{\boldsymbol{y}}+ \dfrac{\sigma}{2}\left. \left\| \hat{A}_b\tilde{\boldsymbol{y}}- \hat{A}_b\tilde{\boldsymbol{y}}(k)+\mathcal{W}\boldsymbol{\eta}(k)\right\|^2\right\},
        \end{split} \\
        & \tilde{\boldsymbol{\theta}}(k+1) =  \argmin\limits_{\tilde{\boldsymbol{\theta}}} \left\{\mu(k)^{\rm T}\left(M_1 \tilde{\boldsymbol{y}}(k+1)+ M_2 \tilde{\boldsymbol{\theta}}\right)+\frac{\rho}{2}\left\|M_1 \tilde{\boldsymbol{y}}(k+1)+ M_2 \tilde{\boldsymbol{\theta}}\right\|^2\right\}, \label{eq:step_ADMM_theta} \\
        & \mu(k+1)  = \mu(k) + \rho \left(M_1 \tilde{\boldsymbol{y}}(k+1)+ M_2 \tilde{\boldsymbol{\theta}}(k+1)\right),  \label{eq:step_ADMM_mu} \\
        & \boldsymbol{\eta}(k+1) = \mathcal{W}\boldsymbol{\eta}(k) + \hat{A}_b  \left(\tilde{\boldsymbol{y}}(k+1)-\tilde{\boldsymbol{y}}(k)\right), \label{eq:step_ADMM_eta} \\
        & \boldsymbol{\lambda}(k+1) = \mathcal{W}\boldsymbol{\lambda}(k) + \sigma \boldsymbol{\eta}(k+1).  \label{eq:step_ADMM_lambda}
    \end{align}
\end{subequations}
\end{proposition}

By Proposition \ref{pro:step_ADMM}, our subsequent convergence analysis for Algorithm \ref{alg:TCADMM} will be conducted based on \eqref{eq:step_ADMM}. We introduce some intermediate variables that will be utilized in proving the convergence:
\begin{equation*}
    \begin{aligned}
        & \bar{\eta}(k) =\frac{1}{N}\sum\nolimits_i \eta_i(k), \quad \bar{\boldsymbol{\eta}}(k) = \mathbf{1}_N \otimes \bar{\eta}(k), \quad \boldsymbol{\epsilon}_1(k) = \boldsymbol{\eta}(k) -\bar{\boldsymbol{\eta}}(k), \\
        &  \bar{\lambda}(k) =\frac{1}{N}\sum\nolimits_i \lambda_i(k), \quad \bar{\boldsymbol{\lambda}}(k) = \mathbf{1}_N \otimes \bar{\lambda}(k), \quad \boldsymbol{\epsilon}_2(k) = \boldsymbol{\lambda}(k) -\bar{\boldsymbol{\lambda}}(k).
    \end{aligned}
\end{equation*}
Here, $\bar{\eta}(k)$ and $\bar{\lambda}(k)$ represent the averages of copied values of $\eta_i(k)$ and $\lambda_i(k)$ for each node;  $\boldsymbol{\epsilon}_1(k)$ and $\boldsymbol{\epsilon}_2(k)$ are the consensus errors of $\eta_i(k)$ and $\lambda_i(k)$, i.e., the differences between the copied value at each node and the corresponding average after each iteration.

The following lemma highlights key features of the consensus error terms $\boldsymbol{\epsilon}_1(k)$ and $\boldsymbol{\epsilon}_1(k)$.

\begin{lemma}\label{lem:bounded}
    Under the Assumptions \ref{ass:const_set} and \ref{ass:com_network}, by defining the auxiliary sequence $\boldsymbol{\xi}(k) = \hat{A}_b \tilde{\boldsymbol{y}}(k)- \bar{\boldsymbol{\eta}}(k)$ and compression matrix $\mathcal{W}_0=\mathcal{W}-(\frac{1}{N}\mathbf{1}_N\mathbf{1}_N^{\rm T}\otimes I_{n_0})$, we have that
    \begin{enumerate}[label=(\roman*),leftmargin=24pt]
        \item $\boldsymbol{\epsilon}_1(k+1) = \mathcal{W}_0 \boldsymbol{\epsilon}_1(k) + \boldsymbol{\xi}(k+1) - \boldsymbol{\xi}(k)$, \ $\boldsymbol{\epsilon}_2(k+1) = \mathcal{W}_0 \boldsymbol{\epsilon}_2(k) + \sigma \boldsymbol{\epsilon}_1(k+1)$,
        \item The sequences $\{\boldsymbol{\xi}(k)\}_{k \geq 0}$, $\{\boldsymbol{\epsilon}_1(k)\}_{k \geq 0}$, and $\{\boldsymbol{\epsilon}_2(k)\}_{k \geq 0}$ are bounded.
    \end{enumerate}
\end{lemma}

\begin{remark}
As a result of the Perron–Frobenius theorem, all eigenvalues of the matrix $\mathcal{W}_0$ are confined within the unit circle. Consequently, it can be shown that if the term $\boldsymbol{\xi}(k+1)-\boldsymbol{\xi}(k)$ converges to $0$ as $k\rightarrow\infty$, then the consensus error terms $\boldsymbol{\epsilon}_1(k)$ and $\boldsymbol{\epsilon}_2(k)$ will also tend towards $0$. This observation will be made rigorous and utilized in our subsequent analysis.
\end{remark}

Finally, based on the boundedness of the above key terms, we aim to identify terms related to the difference between the iterative solution $\tilde{\boldsymbol{y}}(k)$ and the optimal solution $\tilde{\boldsymbol{y}}^*$ for optimization problem \ref{eq:Distri_ADMM}. Subsequently, by constructing the corresponding Lyapunov function, we demonstrate that these difference terms converge to zero, thereby proving that the iterative solution $\tilde{\boldsymbol{y}}(k)$ in the Algorithm \ref{alg:TCADMM} can converge to the optimal solution $\tilde{\boldsymbol{y}}^*$. 

Before proceeding with the detailed elaboration, we define $\tilde{\boldsymbol{\theta}}^*$, and $\lambda^*$, $\mu^*$ as the primal and dual optimal solutions, respectively, for optimization problem \ref{eq:Distri_ADMM}. Meanwhile, we define vectors $\boldsymbol{\zeta}(k) = \sigma (\boldsymbol{\xi}(k)-\hat{A}_b \tilde{\boldsymbol{y}}^*)$ and $\tilde{\boldsymbol{\epsilon}}(k) = {\rm col}(\boldsymbol{\epsilon}_1(k),\sigma \boldsymbol{\epsilon}_2(k))$, and matrix $U_1= \begin{pmatrix}
        I_{Nn_0} & I_{Nn_0}
    \end{pmatrix}^{\rm T}$ and introduce the Lyapunov function
\begin{equation}
    V(k)= \frac{\sigma}{\rho} \| \mu(k)-\mu^*\|^2+\|\bar{\boldsymbol{\lambda}}(k)-\boldsymbol{\lambda}^*\|^2 + \sigma\rho \|M_2(\tilde{\boldsymbol{\theta}}(k)-\tilde{\boldsymbol{\theta}}^*)\|^2+\|U_1\boldsymbol{\zeta}(k)-\tilde{\boldsymbol{\epsilon}}(k)\|_{\mathcal{Q}}^2,
\end{equation}
where $\boldsymbol{\lambda}^* = \mathbf{1}_N \otimes \lambda^*$, and $\mathcal{Q}$ is a symmetric positive-definite matrix, whose specific form is given in Lemma \ref{lem:pos_def_mat} in \ref{pro:Lyapunov}. It is easy to see that the first two terms in the Lyapunov function represent the differences between the iterative values of the dual variables $\lambda(k)$, $\mu(k)$, and their optimal values, respectively. The third term describes the difference between the iterative values of the primal variable $\tilde{\boldsymbol{\theta}}(k)$ and its optimal value. Since Lemma \ref{lem:full_rank} states that matrix $M_2$ is of full column rank, it follows that $M_2^{\rm T}M_2 \succ 0$. Consequently, the boundedness and convergence properties of $\|\tilde{\boldsymbol{\theta}}(k)-\tilde{\boldsymbol{\theta}}^*\|$ are consistent with those of $\|M_2(\tilde{\boldsymbol{\theta}}(k)-\tilde{\boldsymbol{\theta}}^*)\|$. As for the last term, we do the following calculation
\begin{equation*}
\begin{aligned}
    \|U_1\boldsymbol{\zeta}(k)-\tilde{\boldsymbol{\epsilon}}(k)\|_{\mathcal{Q}}^2 & = (U_1\boldsymbol{\zeta}(k)-\tilde{\boldsymbol{\epsilon}}(k))^{\rm T}\mathcal{Q} (U_1\boldsymbol{\zeta}(k)-\tilde{\boldsymbol{\epsilon}}(k)) \\
    & = \|\boldsymbol{\zeta}(k)\|_{\mathcal{R}}^2 - 2\boldsymbol{\zeta}(k)^{\rm T}U_1^{\rm T}\mathcal{Q}\tilde{\boldsymbol{\epsilon}}(k)+\|\tilde{\boldsymbol{\epsilon}}(k)\|_{\mathcal{Q}}^2 \\
    & = \sigma \|\hat{A}_b(\tilde{\boldsymbol{y}}(k)-\tilde{\boldsymbol{y}}^*)-\bar{\boldsymbol{\eta}}(k)\|_{\mathcal{R}}^2 - 2\boldsymbol{\zeta}(k)^{\rm T}U_1^{\rm T}\mathcal{Q}\tilde{\boldsymbol{\epsilon}}(k)+\|\tilde{\boldsymbol{\epsilon}}(k)\|_{\mathcal{Q}}^2,
\end{aligned}
\end{equation*}
where the matrix $\mathcal{R}=(I_{Nn_0}-\mathcal{W}_0)^{-2}$ is a positive definite matrix. Thus, this term to some extent characterizes the difference between the iterative values of the primal variable $\tilde{\boldsymbol{y}}(k)$ and its optimal value, as well as the magnitude of the consensus error.
\begin{proposition}\label{pro:Lyapunov}
    Under Assumptions \ref{ass:const_set}--\ref{ass:com_network}, the sequence $\{V(k)+2\sigma \boldsymbol{\zeta}(k)^{\rm T}\boldsymbol{\epsilon}_2(k)\}_{k \geq 0}$ is non-increasing.
\end{proposition}

\begin{remark}
 Since $\boldsymbol{\xi}(k)$ is bounded, we can infer that $\boldsymbol{\zeta}(k)$ is bounded as well by definition. Additionally, given that $\boldsymbol{\epsilon}_2(k)$ is bounded, it follows that the term $2\sigma \boldsymbol{\zeta}(k)^{\rm T}\boldsymbol{\epsilon}_2(k)$ is also bounded. Therefore, after taking the telescoping sum of the Lyapunov function $V(k)$, the term  $2\sigma \boldsymbol{\zeta}(k)^{\rm T}\boldsymbol{\epsilon}_2(k)$ does not affect the convergence of the corresponding series.
\end{remark}

Thus far, we can conclude that the non-increasing sequence in Proposition \ref{pro:Lyapunov} is bounded below, and hence it converges. Consequently, each term in the Lyapunov function is bounded, implying the existence of convergent subsequences for each term. By identifying a common convergent subsequence from among the bounded sequences, constructing its limit points, and proving that these limit points coincide with the optimal point, we can demonstrate the convergence of the Algorithm \ref{alg:TCADMM}. 

\begin{theorem}\label{thm:convergence}
    Under Assumptions \ref{ass:const_set}--\ref{ass:com_network}, the sequence generated by Consensus-Tracking-ADMM satisfy:
    \begin{enumerate}[label=(\roman*),leftmargin=24pt]
        \item $\lim_{k\rightarrow 0}  H\left(\tilde{\boldsymbol{y}}(k+1)\right) = H\left(\tilde{\boldsymbol{y}}^*\right)$,
        \item The primal sequence $\{\tilde{\boldsymbol{y}}(k)\}_{k\geq 0}$ converges to the optimal solution $\tilde{\boldsymbol{y}}^*$,
        \item Each dual sequence $\{\lambda_i(k)\}_{k\geq 0}$, $i = 1,\dots,N$, converges to the same optimal solution $\lambda^*$.
    \end{enumerate}
\end{theorem}

Theorem \ref{thm:convergence} establishes that the function values of the objective function converge, the sequence $\{\tilde{\boldsymbol{y}}(k)\}_{k \geq 0}$ converges to the optimal solution, and the estimates of the dual variables by all participants eventually converge to the same optimal value.

\section{Design of incentive mechanisms}\label{sec:Mechanisms}

In this section, we present two pricing mechanisms designed to incentivize individual participation while ensuring that decisions prioritizing self-interest align with collective efficiency, thereby demonstrating the rationality of employing distributed approaches for solutions under the proposed mechanisms.

To enhance clarity, we formalize the goals of mechanism design into three key criteria: Individual Rationality (IR), Social Efficiency (SE), and Incentive Compatibility (IC) --- a standard framework in mechanism design theory \citep{borgers2015introduction}. The definitions of these criteria are as follows.

\begin{enumerate}[label=\arabic*)]
    \item Individual rationality ensures that participants do not incur losses by engaging in the mechanism. Formally, a mechanism satisfies IR if the utility for each participant is non-negative. 
    \item Social Efficiency (Welfare maximization) refers to the equilibrium solution of participant decisions under a specified mechanism, which also minimizes the total social cost.    
    \item Incentive compatibility ensures that truthful revelation of private information is optimal for all rational participants. 
\end{enumerate}
In the following text, we will present two pricing mechanisms, the shadow pricing mechanism and the VCG mechanism, and elucidate whether these two mechanisms conform to the above three criteria.

\subsection{Shadow pricing mechanism}

Recall that $\Pi_i$ represents the total payment received by participant $i$ for participating in the task. Our shadow pricing mechanism designs the total payment as the product of the unit price and the contribution of the task, i.e., $\Pi_i= \boldsymbol{\pi}_i^{\rm T} \boldsymbol{x}_i$. The value of each $\boldsymbol{\pi}_i$, acting as a price signal, affects the decision of each participant driven by self-interest. Specifically, the decision-making process of the group of participants can be viewed as a game represented by optimization problems of the form:
\begin{equation}\label{eq:ind_supply}
\tag{$\mathcal{P}^{s}$}
\begin{aligned}
      \forall i \in \{1,\dots,N\}   \quad \min_{\boldsymbol{x}_i} \ \ & u_i(\boldsymbol{x}_i,\boldsymbol{x}_{-i})=f_i(\boldsymbol{x}_i,\boldsymbol{x}_{-i})-\boldsymbol{\pi}_i^{\rm T}\boldsymbol{x}_i \\
         {\rm s.t.} \ \ & (\boldsymbol{x}_i,\boldsymbol{x}_{-i}) \in \mathcal{X}:=\Omega \cap \left\{\boldsymbol{x}\ \Big\vert \sum\nolimits_{i=1}^NA_i\boldsymbol{x}_i =\boldsymbol{d}\right\}.
\end{aligned}
\end{equation}
Let $(\boldsymbol{x}_1^E,\ldots,\boldsymbol{x}_N^E)$ denote any \emph{equilibrium} of the problem \ref{eq:ind_supply} satisfying $u_i(\boldsymbol{x}_i^E, \boldsymbol{x}_{-i}^E) \leq u_i(\boldsymbol{x}_i,\boldsymbol{x}_{-i}^E)$ for all $\boldsymbol{x}_i$ such that $(\boldsymbol{x}_i,\boldsymbol{x}_{-i}^E) \in \mathcal{X}$ for all $i$ \citep{yang2024distributed}. Our goal is to derive an optimal price signal that ensures the equilibrium coincides with the solution $(\boldsymbol{x}_1^\ast,\ldots,\boldsymbol{x}_N^\ast)$ to the optimization problem \ref{eq:Distributed}, while satisfying the desired mechanism design criteria.

Before proceeding, we first introduce some technical assumptions.

\begin{assumption}\label{ass:str_conv}
    The summation $\sum_i f_i(\boldsymbol{x})$ is strictly convex with respect to $\boldsymbol{x}$. Furthermore, for $i=1,\dots,N$,
    \begin{enumerate}[label=(\roman*),leftmargin=24pt]
        \item Each $f_i(\boldsymbol{x})$ is a twice continuously differentiable convex function with respect to $\boldsymbol{x}_i$, and each component of $g_i(\boldsymbol{x}_i)$ is a continuously differentiable convex function with respect to $\boldsymbol{x}_i$.
        \item The function $f_i$ satisfies $f_i(\boldsymbol{x})=0$ when $\boldsymbol{x}_i = 0$.
        \item Each constraint set $\Omega_i$ is nonempty and compact, with the property that $0 \in \Omega_i$.
    \end{enumerate}
\end{assumption}

Assumption \ref{ass:str_conv} imposes further conditions compared to Assumption \ref{ass:const_set},
and introduces a ``neutral point'' $\boldsymbol{x}_{i}=0$ at which the local cost will be zero. This neutral point can be seen as participant $i$ not included in or contributing to the task. In the scenarios where the origin $\boldsymbol{x}_i=0$ does not represent such a neutral point, we only need to make a change of variable by shifting the neutral point to the origin accordingly.

\begin{assumption}\label{ass:LICQ}
    In optimization problem \ref{eq:Distributed}, the linear independence constraint qualification (LICQ) holds at the point $\boldsymbol{x}^*$. 
\end{assumption}
Under Assumptions \ref{ass:saddle}, \ref{ass:str_conv} and \ref{ass:LICQ}, both the primal optimal solution and the dual optimal solution to the optimization problem \ref{eq:Distributed} exist and are uniquely determined. 
\begin{assumption}\label{ass:NE_uni}
    The mapping $F:\mathcal{X} \rightarrow \mathbb{R}^{n_1+\dots+n_N}$ defined by $F(\boldsymbol{x})=\left(\nabla_{\boldsymbol{x}_1}f_1(\boldsymbol{x}),\dots, \nabla_{\boldsymbol{x}_N}f_N(\boldsymbol{x})\right)$ is strictly monotone. That is, for all distinct $\boldsymbol{x} \neq \boldsymbol{x}' \in \mathcal{X}$, $(F(\boldsymbol{x})-F(\boldsymbol{x}'))^{\rm T}(\boldsymbol{x}-\boldsymbol{x}') > 0$. 
\end{assumption}

Assumption \ref{ass:NE_uni} guarantees the existence of at most one equilibrium solution in the game \ref{eq:ind_supply}. This is a standard assumption commonly adopted in game-theoretic frameworks \citep{le2020peer}.

Now, we define the price signal $\boldsymbol{\pi}_i^*$ of our shadow pricing mechanism as
\begin{equation}\label{eq:LMP_price} \boldsymbol{\pi}_i^*:=A_i^{\rm T}\lambda^* - \sum\nolimits_{s \neq i} \nabla_{\boldsymbol{x}_i}f_s(\boldsymbol{x}^*),
\end{equation}
It can be seen that the calculation of $\boldsymbol{\pi}_i^*$ involves the optimal primal-dual pair $(\boldsymbol{x}^*,\lambda^*)$ of the optimization problem \ref{eq:Distributed}, which need to be obtained by the algorithm in Section \ref{sec:Algorithm}. We refer to the above price signals as the shadow pricing mechanism due to the use of the dual variable within it. 

We now illustrate the mechanism criteria satisfied by the shadow pricing mechanism~\eqref{eq:LMP_price}.

\begin{proposition}\label{pro:SDP_IR}
    Under Assumptions \ref{ass:saddle}, \ref{ass:str_conv} and \ref{ass:LICQ}, with the shadow pricing mechanism~\eqref{eq:LMP_price}, we have $u_i^\ast=f_i(\boldsymbol{x}^\ast)-{\boldsymbol{\pi}_i^\ast}^\mathrm{T}\boldsymbol{x}_i^\ast\leq 0$ for any participant $i$.
\end{proposition}

Proposition~\ref{pro:SDP_IR} indicates that teach participant's net cost is guaranteed to be non-positive (or its profit is non-negative), 
which conforms to the individual rationality criterion. 

\begin{proposition}\label{pro:NE_eq_OP}
    Under Assumptions \ref{ass:saddle}, \ref{ass:str_conv}--\ref{ass:NE_uni}, with the shadow pricing mechanism~\eqref{eq:LMP_price}, the equilibrium $\boldsymbol{x}_i^E$ of the game \ref{eq:ind_supply} exists and is unique, and coincides with the optimal solution $\boldsymbol{x}_i^{*}$.
\end{proposition}

Proposition~\ref{pro:NE_eq_OP} demonstrates that, under the price signal $\boldsymbol{\pi}_i^*$, the equilibrium arising from participant interactions still coincides with the optimum for optimization problem \ref{eq:Distributed}, thereby confirming the criterion of social efficiency.


Note that the price signal $\boldsymbol{\pi}_i^*$ is dependent on the optimal solution $\boldsymbol{x}^*$, and so far we have implicitly assumed that all individuals will not use fake local cost functions or constraints when solving~\ref{eq:Distributed} to obtain $\boldsymbol{x}^*$. However, as the following example shows, by strategically using a fake local cost function $f_i$, an individual may obtain higher net benefits, indicating that the shadow pricing mechanism~\eqref{eq:LMP_price} lacks the property of incentive compatibility.

\begin{example}\label{ex:simple_case}
We consider a special case of Example~\ref{example:allo}, where the transportation network is shown in Figure~\ref{fig:sim_cas}, consisting of three suppliers $(N=3)$ and one demander $(M=1)$. Each supplier has only one path $(R=1)$ to the demander, and only one type of material $(K=1)$ is transported. The suppliers have no constraints on inventory capacity and transportation capacity, and the demander has the fixed demand of $\boldsymbol{d}=5$. There are 4 edges, $e_1$, $e_2$, $e_3$, $e_4$, in the network. Each supplier $i$ has a private cost $c_{i,e_s}$ per edge, forming its cost vector $\boldsymbol{c}_i = (c_{i,e_1}, c_{i,e_2}, c_{i,e_3}, c_{i,e_4})^{\rm T}$. The actual cost vectors are $\boldsymbol{c}_1=(1,0,0,1)^{\rm T}$, $\boldsymbol{c}_2=(0,2,0,1)^{\rm T}$, $\boldsymbol{c}_3=(0,0,3,1)^{\rm T}$. Additionally, we specify the congestion cost function on each edge $e$ as $c_e(q_e(\boldsymbol{x}))=c_0q_e(\boldsymbol{x})$, 
with $c_0=1$.
    \begin{figure}[!ht]
        \centering
        \includegraphics[scale=0.65]{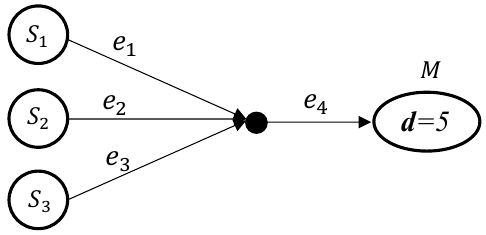}
        \caption{Transportation network in Example \ref{ex:simple_case}}\label{fig:sim_cas}
    \end{figure}

If all suppliers use true local cost parameters, the optimal solution is $x_1^*=13/6$, $x_2^*=5/3$, $x_3^*=7/6$, and their respective benefits are $|u_1^*|=9.38$, $|u_2^*|=5.56$, $|u_3^*|=2.72$. However, if supplier 1 deliberately conceals its private cost parameter and instead uses a fake parameter $\boldsymbol{c}_1=(0.5,0,0,0.5)^{\rm T}$ for distributed optimization, while the other suppliers continue to use true parameters, the solution changes to $x_1'=2.5$, $x_2'=1.5$, $x_3'=1$. As a result, suppler 1's benefit increases to $|u_1'|=12.5$, while suppliers 2 and 3's benefits decrease to $|u_2'|=4.5$, $|u_3'|=2$ (see Supplementary Material \ref{sup_mat:Notes} for full derivations).
\end{example}

The above example demonstrates that in the shadow pricing mechanism, some participants can benefit by using fake local cost functions in the distributed optimization phase, which violates the incentive compatibility criterion. In Supplementary Material \ref{app:typ_exa}, we illustrate an example to show that although participants have incentives to use fake local cost function, participants struggle to gain more profits when all engage in misreporting. In Section~\ref{sec:numerical}, we conduct simulations of misreporting in more complex network topologies, illustrating that a single participant cannot guarantee that using fake local cost function will yield higher benefits. This deters participants from employing fake local cost function, thereby partially addressing the limitation of the shadow pricing mechanism.

\subsection{Vickrey-Clarke-Groves (VCG) mechanism}

Compared to the design in the shadow pricing mechanism, which breaks down the total payment into unit prices, the VCG mechanism directly provides the total payment for each participant \citep{pu2020online}. This mechanism posits that a participant's value is manifested in their substitution effect on the remaining participants. In the context of our problem, the total payment $\Pi_i$ received by participant $i$ should be the difference in total social cost of remaining nodes before and after participant $i$ participates in the task. 

Specifically, let function $f_{-i}(\boldsymbol{x}_{-i})= \sum\nolimits_{s\neq i} \left.f_s(\boldsymbol{x})\right|_{\boldsymbol{x}_i=\boldsymbol{0}}$ denote the total social cost required for the remaining participants to complete the task after excluding participant $i$. 
The optimization problem when participant $i$ is not involved in the task is given by 
\begin{equation}\label{eq:sub_optimal_i}
\tag{$\mathcal{SP}_i$}
    \begin{aligned}
         \min_{\boldsymbol{x}_{-i}} \quad & f_{-i}(\boldsymbol{x}_{-i})= \sum\nolimits_{s\neq i} \left.f_s(\boldsymbol{x})\right|_{\boldsymbol{x}_i=\boldsymbol{0}} \\
        {\rm s.t.} \quad \  &\sum\nolimits_{s \neq i} A_s\boldsymbol{x}_s =\boldsymbol{d}, \\
        & \ g_s(\boldsymbol{x}_s) \leq \boldsymbol{0},  \quad \forall s \neq i.
    \end{aligned}
\end{equation}

The following assumption implies that if any one participant is removed, the task can still be accomplished by the remaining ones.
\begin{assumption}\label{ass:mon}
     The summation $\sum_i f_i(\boldsymbol{x})$ is strictly convex with respect to $\boldsymbol{x}$. Furthermore, for $i=1,\dots,N$,
    \begin{enumerate}[label=(\roman*),leftmargin=24pt]
        \item The function $f_i$ satisfies $f_i(\boldsymbol{x})=0$ when $\boldsymbol{x}_i = 0$.
        \item Each constraint set $\Omega_i$ is compact, with the property that $\boldsymbol{0} \in \Omega_i$.
        \item The optimization problem \ref{eq:sub_optimal_i} has a feasible solution.
    \end{enumerate}
\end{assumption}

To formally define the incentive payments under the VCG mechanism, we introduce the notation $\hat{f}_i$ to denote the cost function reported by the participant $i$ in solving~\ref{eq:Distributed} and each~$\mathcal{SP}_i$, which may be different from the true cost function $f_i$. We also let $\hat{\mathbf{f}}_{-i}$ denote the functions used by the remaining participants other than participant $i$, and denote $\hat{f}_{-i}(\boldsymbol{x}_{-i})= \sum\nolimits_{s\neq i} \left.\hat{f}_s(\boldsymbol{x})\right|_{\boldsymbol{x}_i=\boldsymbol{0}}$. We use $\boldsymbol{x}^*(\hat{ f}_1,\dots,\hat{f}_{N})$ to denote the solution to~\ref{eq:Distributed} and use $\boldsymbol{x}_{-i}^{\circ}(\hat{\mathbf{f}}_{-i})$ to denote the optimal solution to~\ref{eq:sub_optimal_i}, when the participants employ the functions $\hat{f}_1,\dots,\hat{f}_N$. We abbreviate $\boldsymbol{x}^*(\hat{ f}_1,\dots,\hat{f}_{N})$ for participant $i$ as $\boldsymbol{x}^*(\hat{f}_i,\hat{\mathbf{f}}_{-i})$.
Then the total payment $\Pi_i$ received by participant $i$ under the VCG mechanism is defined by  
\begin{equation}\label{eq:VCG_price}
\begin{aligned}
    \Pi_i(\hat{f}_{i},\hat{\mathbf{f}}_{-i}) ={} & \hat{f}_{-i}(\boldsymbol{x}_{-i}^{\circ}(\hat{\mathbf{f}}_{-i})) - \left(\sum\nolimits_{s=1}^N\hat{f}_s (\boldsymbol{x}^*(\hat{f}_{i},\hat{\mathbf{f}}_{-i}))-\hat{f}_i(\boldsymbol{x}^*(\hat{f}_i,\hat{\mathbf{f}}_{-i})) \right) \\
={} & \hat{f}_{-i}(\boldsymbol{x}_{-i}^{\circ}(\hat{\mathbf{f}}_{-i})) - \sum\nolimits_{s\neq i}\hat{f}_s (\boldsymbol{x}^*(\hat{f}_{i},\hat{\mathbf{f}}_{-i})).
\end{aligned}
\end{equation}
Determining the payments using the VCG mechanism requires solving the $N+1$ optimization problems \ref{eq:Distributed}, $\mathcal{SP}_1$, $\dots$, $\mathcal{SP}_N$. Each individual has to participate in $N$ of these problems, and it is worth emphasizing that each participant reports the same local function (whether true or fake) across the $N$ problems.

\begin{remark}
\label{remark:VCG_implementation}
    In contrast to the ex-ante price signals of the shadow pricing mechanism, the VCG mechanism emphasizes ex-post payments. Participants must first declare the functions $\hat{f}_i$, which become irrevocable upon announcement. Incentive payments are then calculated based on the declared functions, and each participant receives the final payment only if it actually delivers the required optimal decision $\boldsymbol{x}^\ast_i(\hat{f}_i,\hat{\mathbf{f}}_{-i})$.
\end{remark}

Now, we illustrate the mechanism design criteria satisfied by the VCG mechanism. Since the true local function is $f_i$, it follows that $\boldsymbol{x}^*({f}_{1},\dots, {f}_N)= \boldsymbol{x}^*$. Assuming that all participants report true local functions, the payment under the VCG mechanism for each participant is then
$
    \Pi_i^* = f_{-i}(\boldsymbol{x}_{-i}^{\circ}(\mathbf{f}_{-i})) - \left(f(\boldsymbol{x}^*)-f_i(\boldsymbol{x}^*) \right)
$.
The following proposition shows that the VCG mechanism fulfills the individual rationality and social efficiency criteria when all participants engage truthfully.
\begin{proposition}\label{pro:VCG_IR}
    Under Assumptions \ref{ass:mon}, with the ex-post incentive payment~$\Pi_i^*$, we have 
    \begin{enumerate}[label=(\roman*),leftmargin=24pt]
        \item Each participant's net cost satisfies $u_i\leq 0$.
        \item Each participant's decision under the VCG mechanism coincides with the optimal solution $\boldsymbol{x}^*$.
    \end{enumerate}
\end{proposition}

\proof For the first statement, since \ref{eq:sub_optimal_i} can be obtained by adding a constraint to~\ref{eq:Distributed}, it is straightforward to see that $f(\boldsymbol{x}^*)$, the optimal value of~\ref{eq:Distributed}, is always less than or equal to $f_{-i}(\boldsymbol{x}_{-i}^{\circ}(\mathbf{f}_{-i}))$, the optimal value of~\ref{eq:sub_optimal_i}. Thus
$
u_i  = f_i(\boldsymbol{x}^*)- \Pi_i^* 
         = f_i(\boldsymbol{x}^*)- f_{-i}(\boldsymbol{x}_{-i}^{\circ}(\mathbf{f}_{-i}))+(f(\boldsymbol{x}^*)-f_i(\boldsymbol{x}^*)) 
         = f(\boldsymbol{x}^*) - f_{-i}(\boldsymbol{x}_{-i}^{\circ}(\mathbf{f}_{-i}))
         \leq 0
$.
The second statement holds by the design of the VCG mechanism when each participant uses true cost information; see Remark~\ref{remark:VCG_implementation}.
\endproof

Next, we illustrate that employing fake local cost function cannot result in a reduction of local net costs for each participant under the VCG mechanism.

\begin{proposition}\label{pro:VCG_IC}
    Regardless of whether other participants report their local functions truthfully, it remains optimal for participant $i$ to use the true function $f_i$ in the distributed optimization phase.
\end{proposition}

\proof By the VCG mechanism~\eqref{eq:VCG_price}, we have
\begin{equation*}
    \begin{aligned}
        u_i(\boldsymbol{x}^*(f_{i},\hat{\mathbf{f}}_{-i})) 
         ={} & f_i(\boldsymbol{x}^*(f_{i},\hat{\mathbf{f}}_{-i}))+\sum\nolimits_{j \neq i}\hat{f}_j(\boldsymbol{x}^*(f_{i},\hat{\mathbf{f}}_{-i}))- \hat{f}_{-i}(\boldsymbol{x}_{-i}^{\circ}(\hat{\mathbf{f}}_{-i})) \\
         \leq{} & f_i(\boldsymbol{x}^*(\hat{f}_{i},\hat{\mathbf{f}}_{-i})) +\sum\nolimits_{j \neq i}\hat{f}_j(\boldsymbol{x}^*(\hat{f}_{i},\hat{\mathbf{f}}_{-i}))- \hat{f}_{-i}(\boldsymbol{x}_{-i}^{\circ}(\hat{\mathbf{f}}_{-i}))  
        =   u_i(\boldsymbol{x}^*(\hat{f}_{i},\hat{\mathbf{f}}_{-i})),
    \end{aligned}
\end{equation*}
where the inequality holds because $\boldsymbol{x}^\ast(f_i,\hat{\mathbf{f}}_{-i})$ minimizes $f_i(\boldsymbol{x}) +\sum\nolimits_{j \neq i}\hat{f}_j(\boldsymbol{x})$, and the last equality holds because participant $i$ must use its true function $f_i$ when calculating the net cost.\endproof

Therefore, the VCG mechanism satisfies the incentive compatibility criterion, implying that all participants in the task will choose to truthfully use their private local cost function.

We now compare the two proposed mechanisms. The shadow pricing mechanism only needs to solve one optimization problem, but it does not satisfy the incentive compatibility criterion. Nevertheless, in complex tasks, it is difficult for any participant to generate positive profits by using fake local cost function, which can partially make up for the lack of the incentive compatibility property. The VCG mechanism satisfies all three aforementioned criteria but has a relatively large solving time overhead. In Section~\ref{subsection:mechanism_compare_numerical}, we shall further compare the two mechanisms numerically by experiments.

\section{Numerical study}\label{sec:numerical}

We demonstrate the effectiveness of the algorithm and discuss the performance of the two mechanisms. We implement the following experiments using the Gurobi optimizer (v11.0.3) on the Python (v3.11.7) platform, with the relevant code made available at \url{https://github.com/CodeDW771/EJOR25_code}.

\subsection{Comparison of algorithms}\label{subsec:Com_Alg}

We conduct algorithm comparisons using the optimization problem \eqref{eq:example-tran} setup presented in Example~\ref{example:allo}, and specify the congestion function on each edge as $c_e(q_e(\boldsymbol{x}))=c_0q_e(\boldsymbol{x})$, indicating identical sensitivity to congestion across all edges. The problem scale is defined by parameters $(N, M, K, R)$, specifically, encompassing three configurations: small-scale $(4, 2, 3, 2)$, medium-scale $(10, 5, 10, 2)$, and large-scale $(20, 8, 10, 3)$. Figure \ref{fig:trans_graph} shows the three transportation networks, with suppliers in red and demanders in blue. We compare our proposed algorithm, \textit{Consensus-Tracking-ADMM}, against several representative recent methods, including DPMM \citep{gong2023decentralized}, IPLUX \citep{wu2023distributed}, Tracking-ADMM \citep{falsone2020tracking}, and DC-ADMM \citep{chang2016proximal}. All sub-problems were addressed using the acceleration technique presented in Supplementary Material~\ref{sup_mat:Alg_impro}. Figure \ref{fig:Opt_gap} illustrates the changes in the relative error of the function value during the solution process, defined as $|\sum\nolimits_i f_i(\boldsymbol{y}_i(k))-f(\boldsymbol{x}^*)|/f(\boldsymbol{x}^*)$. Figure \ref{fig:Con_vio} shows the evolution of the constraint violation, given by $\|\sum\nolimits_i \tilde{A}_i \boldsymbol{y}_i(k)-\boldsymbol{d}\|+\sum_{i,j}\|\boldsymbol{y}_i(k)-\boldsymbol{y}_j(k)\|$.
\begin{figure}[!ht]
    \centering  
    \begin{minipage}{.33\textwidth}  
        \centering  
        \includegraphics[height=30mm,width=35mm]{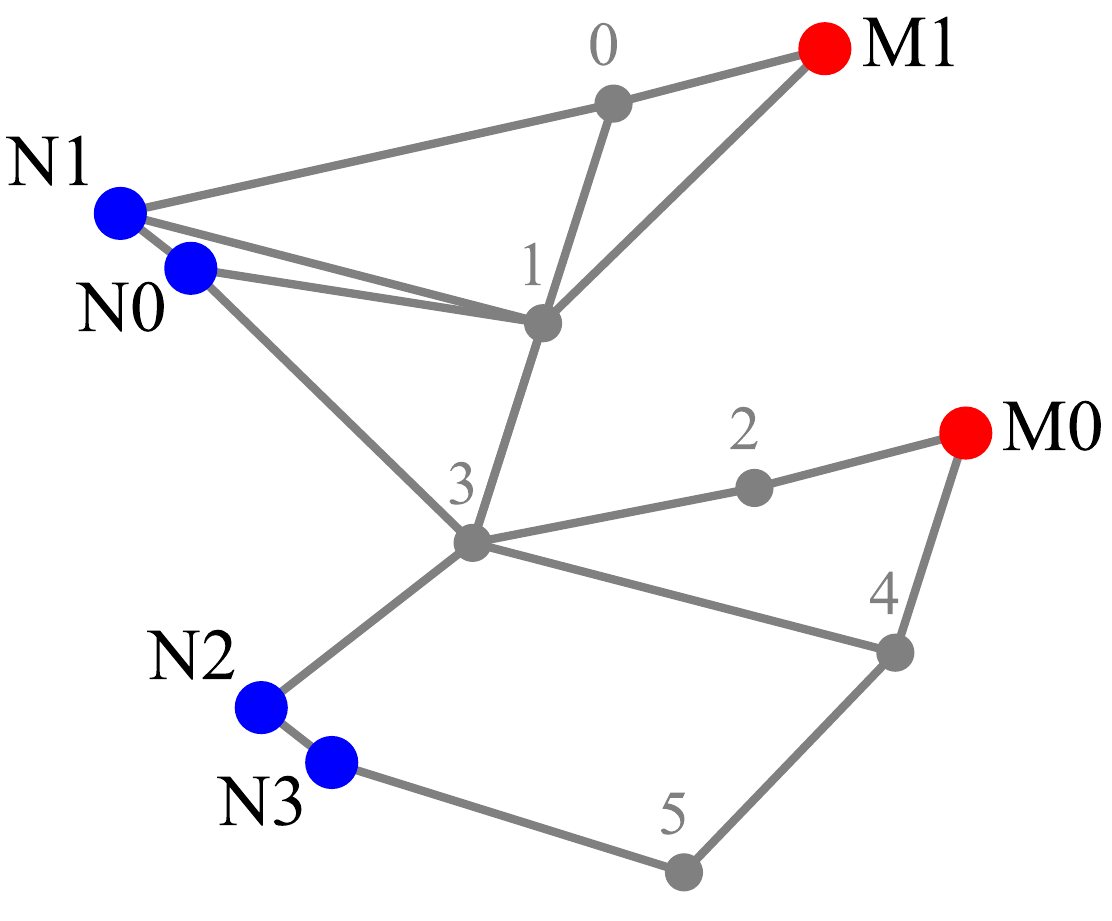}  
        \subcaption{Small-scale} \label{subfig:small_graph} 
    \end{minipage}%
    \hfill
    \begin{minipage}{.33\textwidth}  
        \centering  
        \includegraphics[height=30mm,width=35mm]{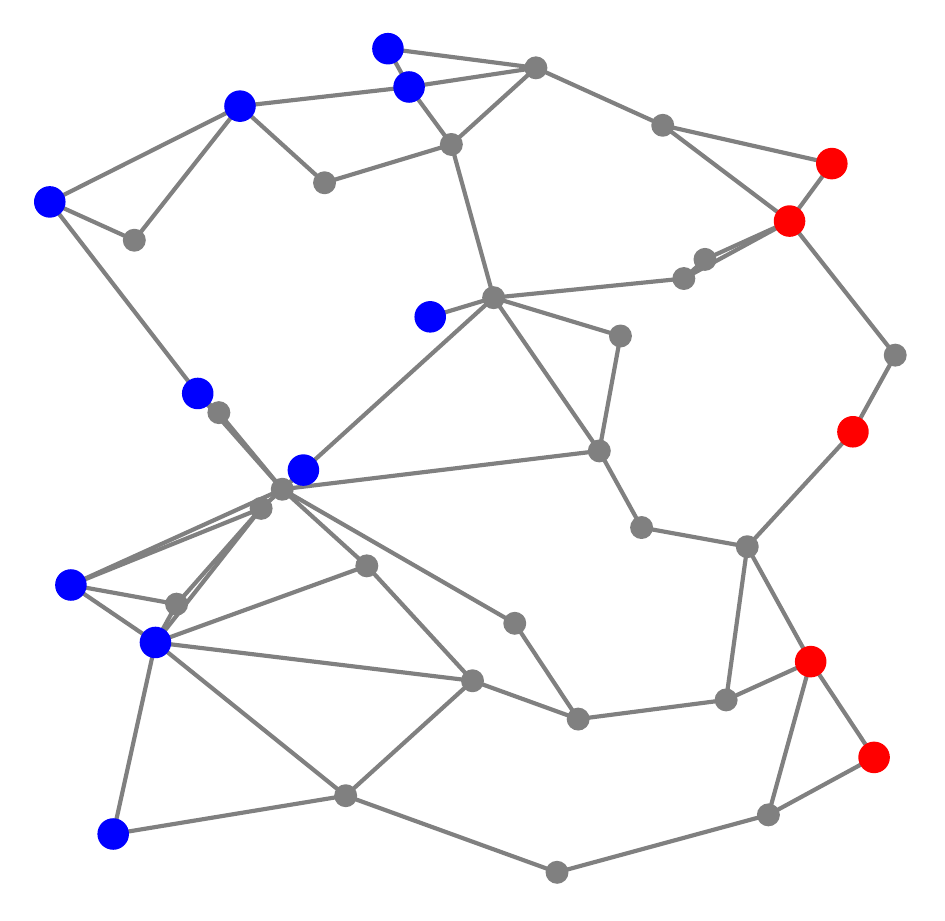}  
        \subcaption{Medium-scale}  
    \end{minipage}%
    \hfill
    \begin{minipage}{.33\textwidth}  
        \centering  
        \includegraphics[height=30mm,width=35mm]{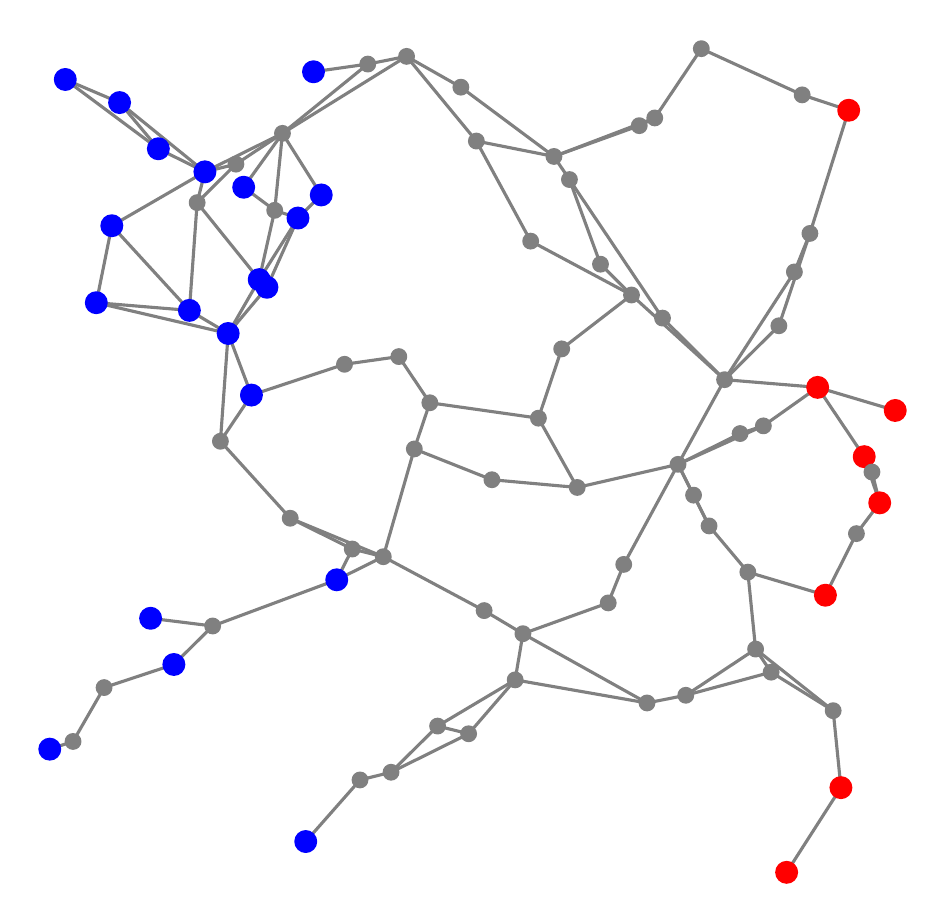}  
        \subcaption{Large-scale}  
    \end{minipage}%
    \caption{Transportation network in different scale}  
    \label{fig:trans_graph}  
\end{figure} 
\begin{figure}[!ht]  
    \centering  
    \begin{minipage}{.33\textwidth}  
        \centering  
        \includegraphics[width=\linewidth]{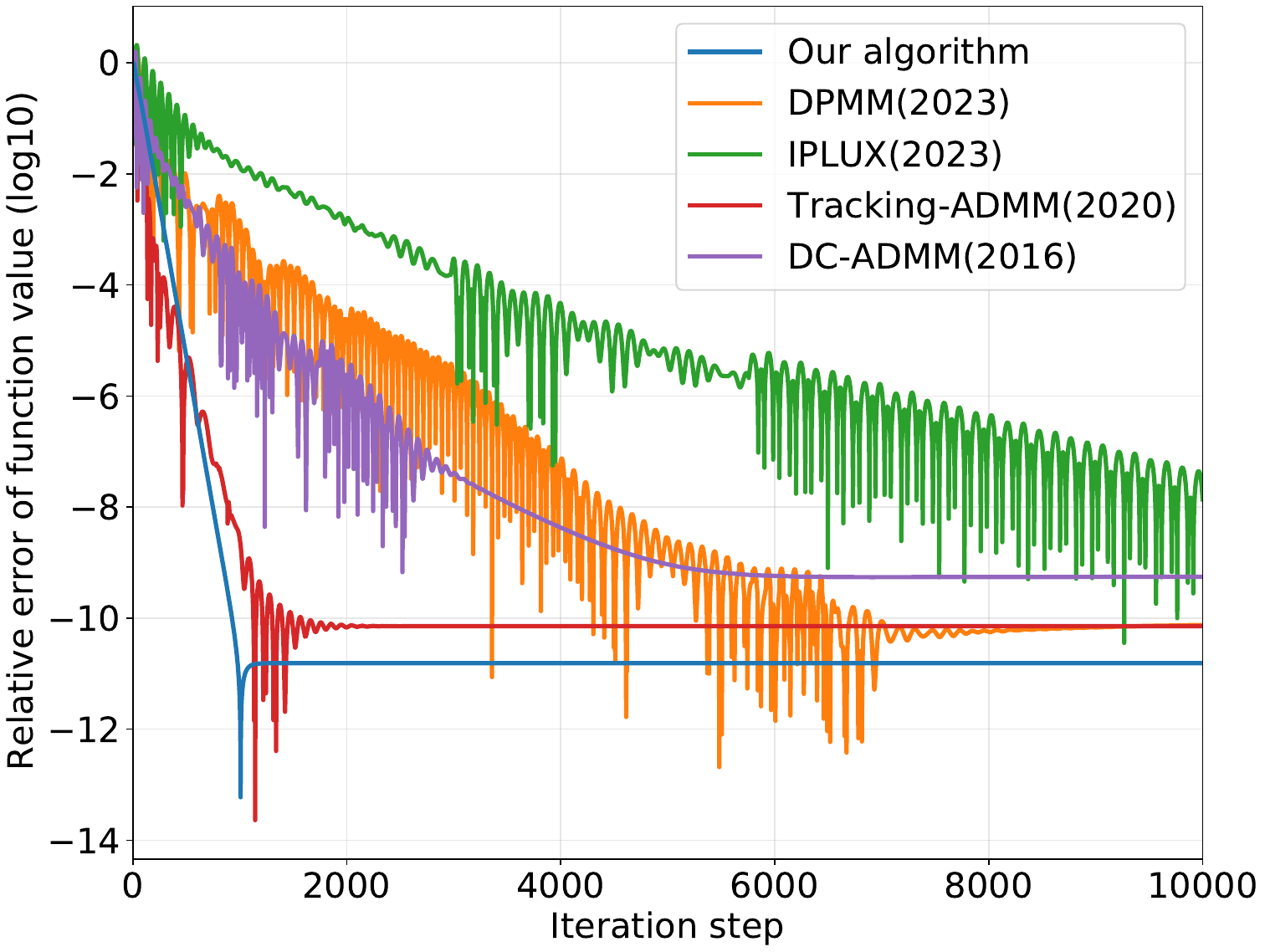}  
        \subcaption{Small-scale}  
    \end{minipage}%
    \hfill 
    \begin{minipage}{.33\textwidth}  
        \centering  
        \includegraphics[width=\linewidth]{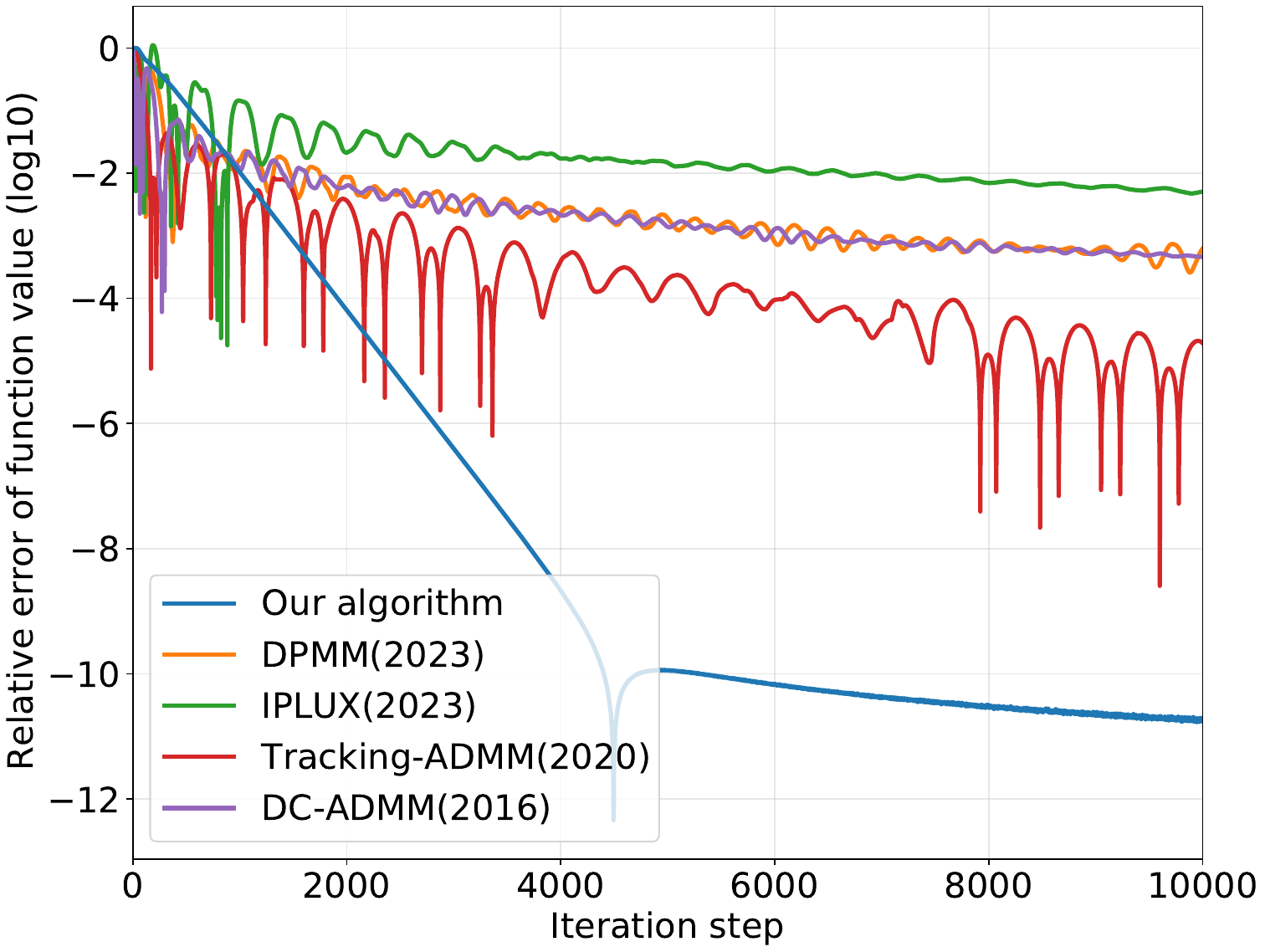}  
        \subcaption{Medium-scale}  
    \end{minipage}%
    \hfill 
    \begin{minipage}{.33\textwidth}  
        \centering  
        \includegraphics[width=\linewidth]{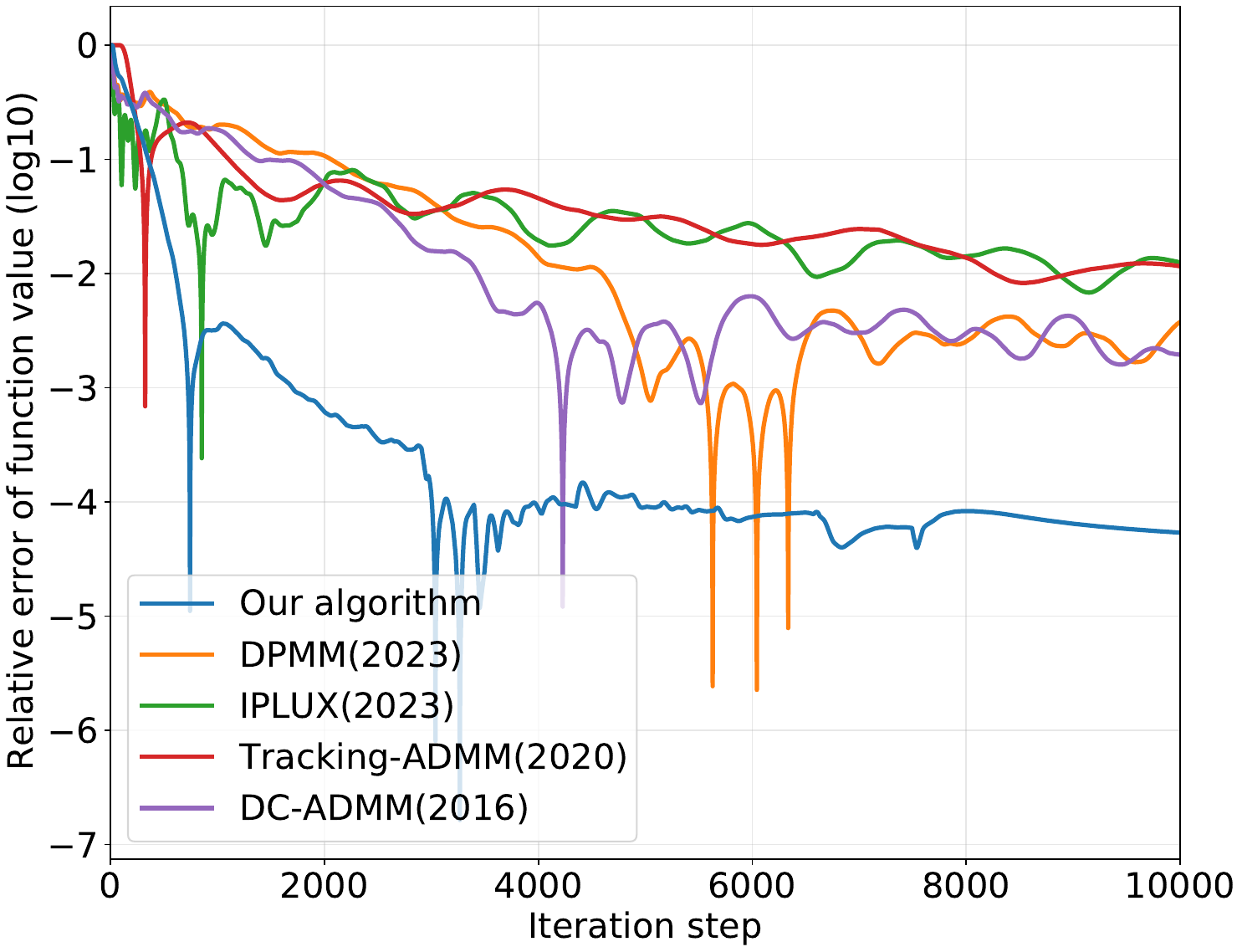}  
        \subcaption{Large-scale}  
    \end{minipage}%
    \caption{Comparison of optimality gap performance}  
    \label{fig:Opt_gap}  
\end{figure} 

\begin{figure}[!ht]  
    \centering  
    \begin{minipage}{.33\textwidth}  
        \centering  
        \includegraphics[width=\linewidth]{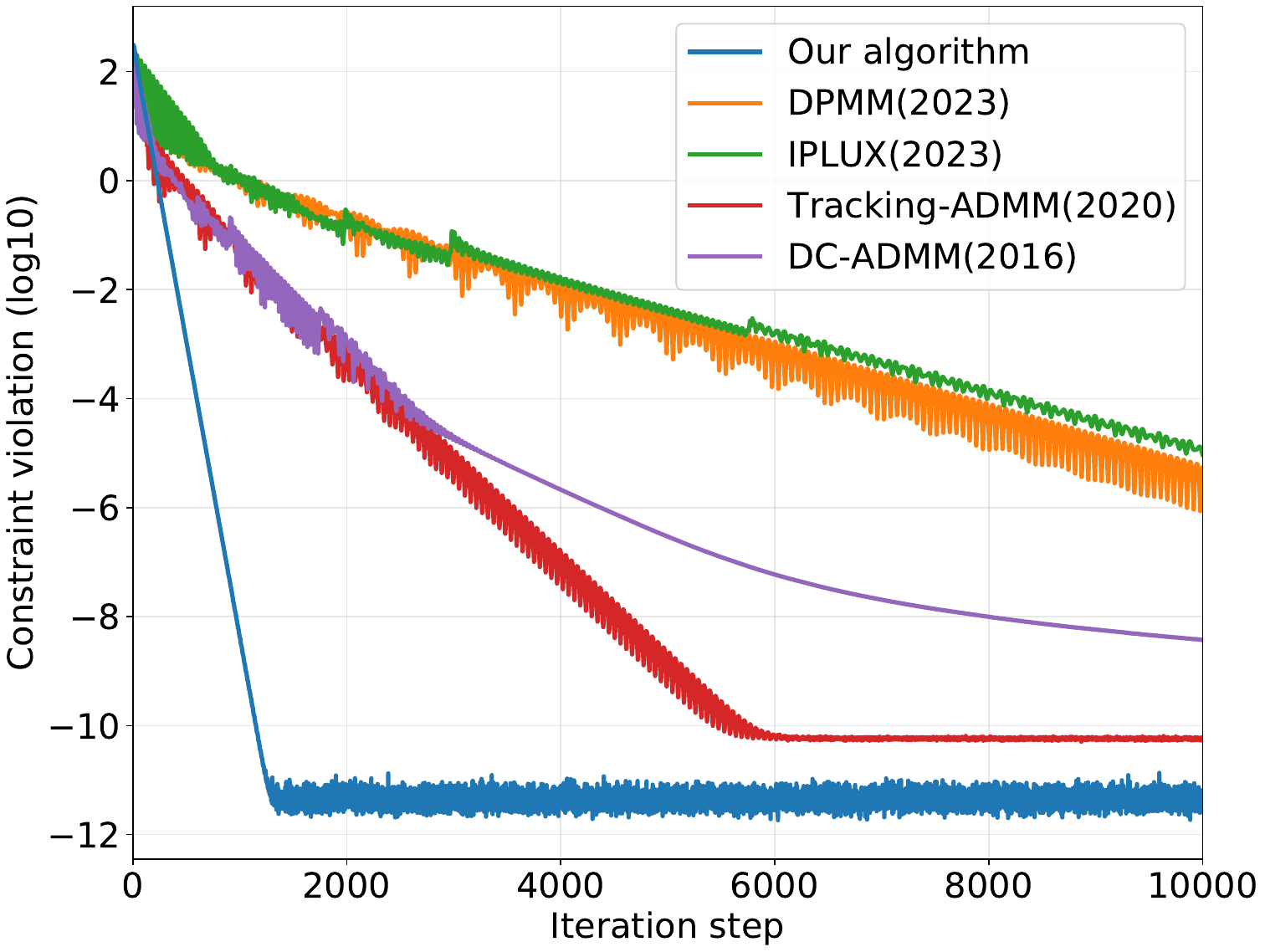}  
        \subcaption{Small-scale}  
    \end{minipage}%
    \hfill 
    \begin{minipage}{.33\textwidth}  
        \centering  
        \includegraphics[width=\linewidth]{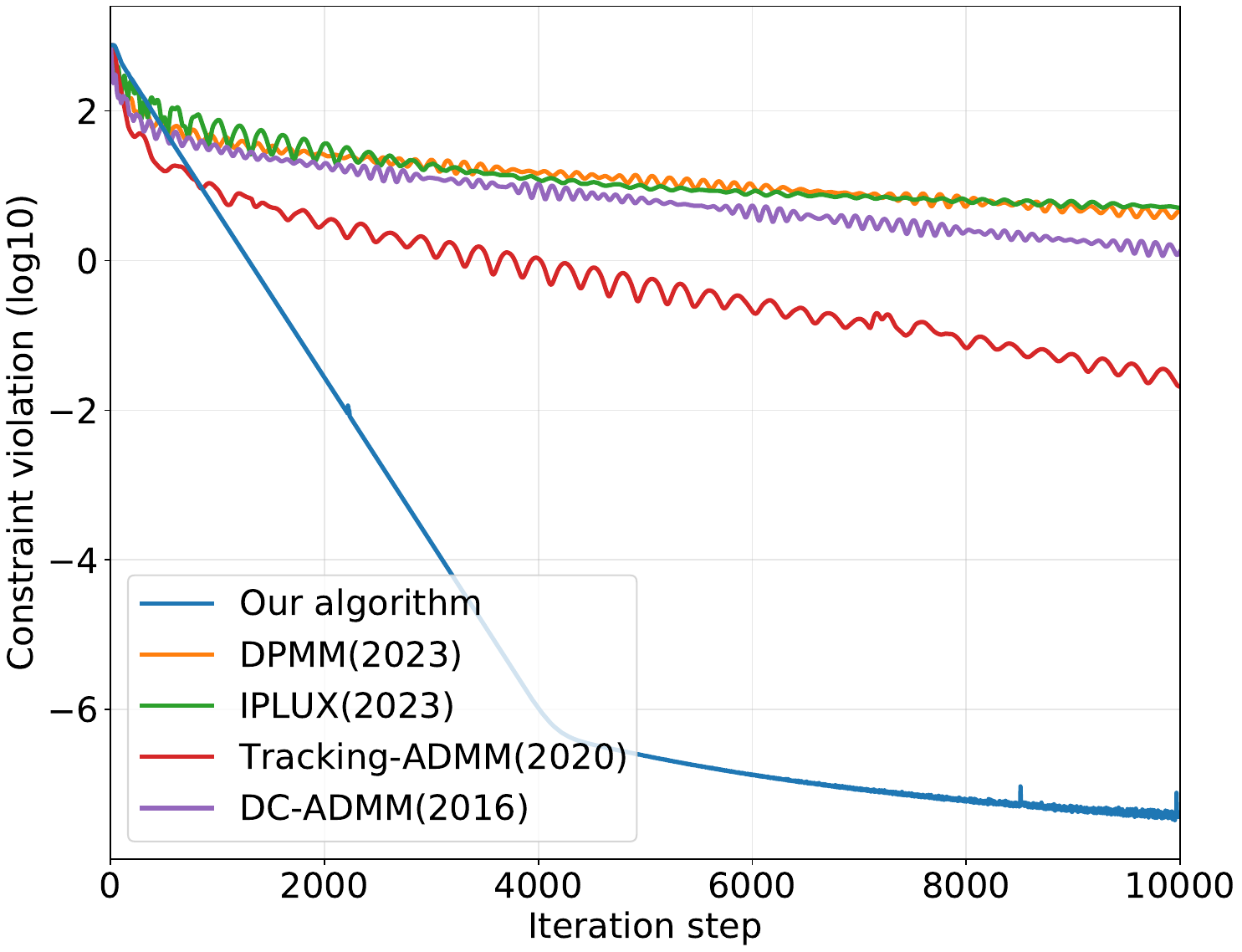}  
        \subcaption{Medium-scale}  
    \end{minipage}%
    \hfill
    \begin{minipage}{.33\textwidth}  
        \centering  
        \includegraphics[width=\linewidth]{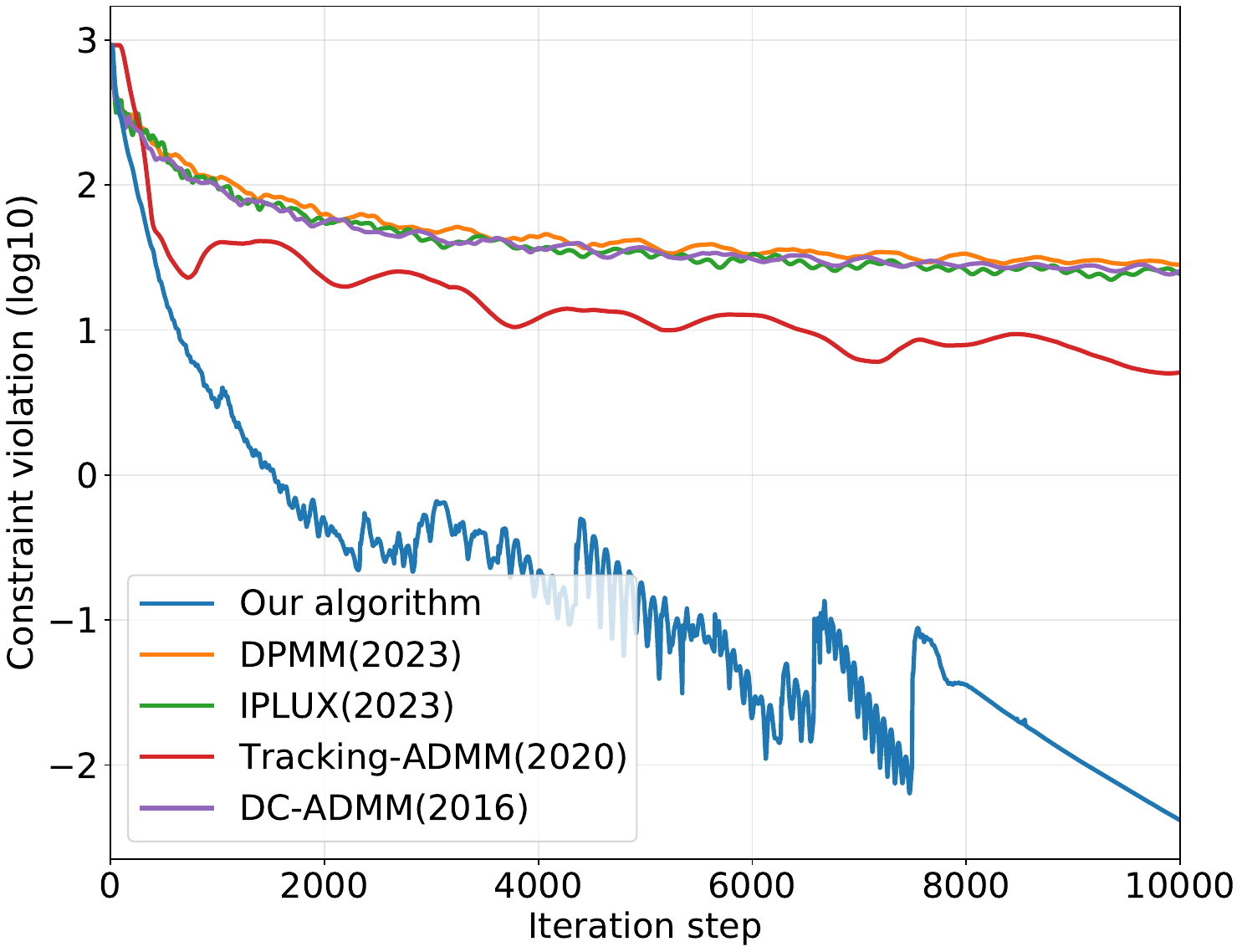}  
        \subcaption{Large-scale}  
    \end{minipage}%
    \caption{Comparison of constraint violation performance}  
    \label{fig:Con_vio}  
\end{figure} 

As can be seen from the figures, under different configurations, our proposed algorithm approaches the optimum significantly faster than other methods handling distributed coupled constraints, while achieving lower constraint violation at the same iteration steps.

\subsection{Comparison of the two mechanisms}
\label{subsection:mechanism_compare_numerical}

We present the comparative analysis of the performance of two mechanisms proposed in Section \ref{sec:Mechanisms}, employing a small-scale parameter tuple $(N, M, K, R)$ = $(4, 2, 3, 2)$. Note that the network structure will influence the mechanism's performance. To demonstrate the core conclusions, we employed the balanced network structure (no participant has a path length advantage to any demand node) shown in Figure \ref{fig:trans_graph}(\subref{subfig:small_graph}), while the comparisons using other network configurations are relegated to Supplementary Material \ref{app:supp_exp}.

First, we demonstrate the variation in total incentive payments resulting from perturbing specific parameters, thereby showing that neither mechanism is absolutely superior to the other. The perturbed parameters include the task target $\boldsymbol{d}$, the individual constraint bound $\boldsymbol{m}$ ($\boldsymbol{m}_i$ of all individual changes proportionally at the same time), and the congestion cost coefficient $c_0$. The total payment difference is defined as SP incentive payments minus VCG incentive payments. The results indicate that under balanced network structures, parameter changes exert a negligible influence on the total payments of both mechanisms, and the differences in these payments do not vary linearly with changes in the same parameter. This suggests that neither mechanism demonstrates significant dominance over the other in balanced networks. Additionally, Supplementary Material \ref{app:supp_exp} will further demonstrate the differences in total payments between the mechanisms under various other network structures.

\begin{table}[!ht]
  \centering
  \caption{Differences in incentive payments across different settings of parameters in a well-balanced network}
  \setstretch{1.1}
    \begin{tabular}{ccccccccc}
    \hline
    \makecell[c]{Task\\Demand} & SP-VCG & Pct. & \makecell[c]{Individual \\Constraints} & SP-VCG & Pct. & \makecell[c]{Congestion \\ Parameter} & SP-VCG & Pct. \\
    \hline
    $0.4\boldsymbol{d}$  & -353.96   & 4\%  & 0.8$\boldsymbol{m}$ & -210.25   & 1\%  & 0.6$c_0$ & -926.62  & 3\%\\
    $0.6\boldsymbol{d}$  & -173.13   & 1\% & 0.9$\boldsymbol{m}$ & 44.43   &   0\% & 0.8$c_0$ & -209.09 & 1\% \\
    0.8$\boldsymbol{d}$    &  -373.33 & 1\%  & $\boldsymbol{m}$     & 245.57   & 1\% & $c_0$  & 245.57 & 1\% \\
    $\boldsymbol{d}$  & 245.57  & 1\%  & 1.1$\boldsymbol{m}$  & 366.76     & 1\% & 1.2$c_0$ & 661.64 &  2\% \\
    1.2$\boldsymbol{d}$  & 574.31  & 1\%  & 1.2$\boldsymbol{m}$  & 217.47  &  1\% & 1.4$c_0$ & 780.4  & 2\% \\
    \hline
    \end{tabular}%
  \label{tab:incentive}%
  \begin{tablenotes}
  \item \footnotesize Pct. (Percentage)= $|SP- VCG|/\max(SP, VCG)$ quantifies the relative difference in total payments between the shadow pricing mechanism (SP) and the VCG mechanism (VCG).
  \end{tablenotes}
\end{table}%
Next, we conduct simulations regarding the lack of incentive compatibility in the SP mechanism, showing that it is difficult for participants to obtain additional gains through misreporting within a complex network structure. We define the private transportation cost for each participant $i$ on edge $e$ as $c_{i,e}$. Then, the transportation cost vector for participant $i$ across all network edges is denoted $\boldsymbol{c}_i={\rm col}(c_{i,e_1},\dots, c_{i,e_E})$, where $E$ is the number of edges in transportation network $\mathcal{G}_{\mathrm{tra}}$. Figure~\ref{fig:misreporting_single} presents the net payoffs $|u_i|$ of all participants as one participant varies its parameter $\boldsymbol{c}_i$ in local function $f_i$, while the others truthfully report their local function. The $x$-label represents the magnitude of variation $\Delta_i$, where at each magnitude the parameter $\boldsymbol{c}_i$ becomes $\boldsymbol{c}_i+\Delta_i \cdot \boldsymbol{1}$. When $\Delta_i$ equals 0, it indicates that all of the participants use the local function truthfully, i.e., at this point, the value obtained is $u_i^*$. The overall pattern revealed that when one participant increases its misreporting parameter, the profit of the other participants rises. This is intuitive because when a participant reports higher costs, other participants progressively gain a relative cost advantage, thus securing a larger task allocation and consequently increasing their profits. In the Figure \ref{fig:misreporting_single}(\subref{subfig:single0}), participant N0 can increase its profit by slightly raising its cost parameter, but excessive increases are counterproductive. Closer inspection reveals that this is because N0 possesses a comparative cost advantage on a path to the demand node M1. A modest cost increase preserves N0’s relative advantage over alternative nodes, while the SP mechanism provides higher incentives for this path, boosting N0’s profit. On the other hand, when the cost increase becomes excessive, N0’s cost advantage diminishes, leading to declining profit. If N0 lowers its cost parameter, it maintains its advantage, but the reduced incentives provided by the SP mechanism result in lower profit. For the remaining three participants, neither upward nor downward misreporting of their own parameters is likely to yield substantial profit improvements. We see that for this test case, when participants lack knowledge of others' parameters to assess their cost advantages, adopting misreporting strategies becomes highly challenging.

\begin{figure}[!ht]  
    \centering  
    \begin{minipage}{.24\textwidth}  
        \centering  
        \includegraphics[width=\linewidth]{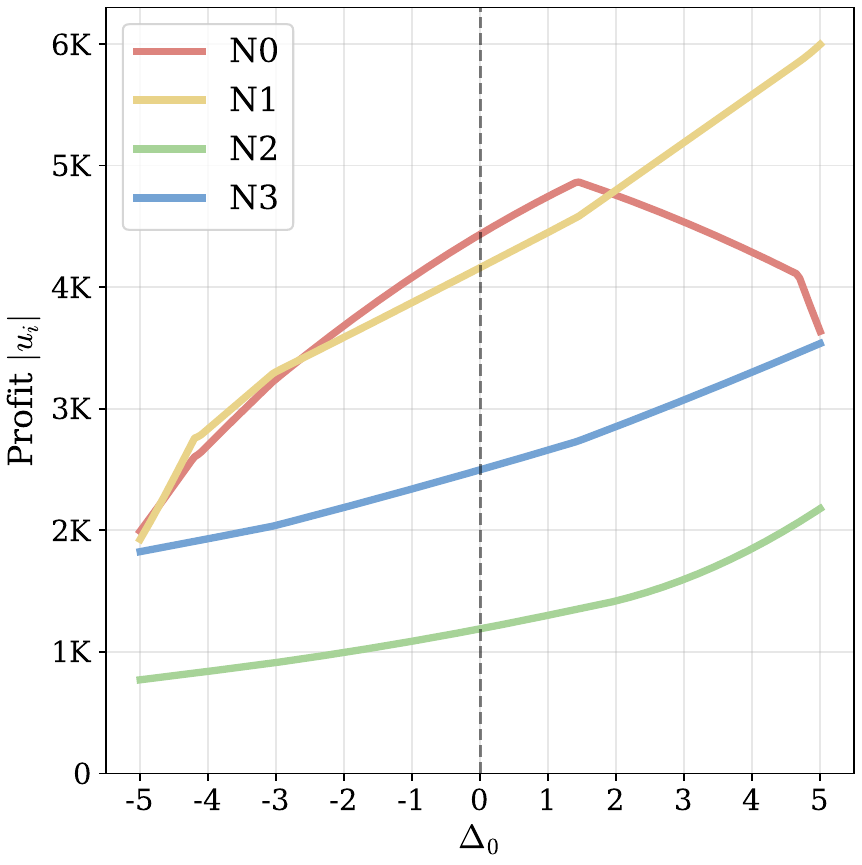}  
        \subcaption{Participant N0}\label{subfig:single0}  
    \end{minipage}%
    \hfill 
    \begin{minipage}{.24\textwidth}  
        \centering  
        \includegraphics[width=\linewidth]{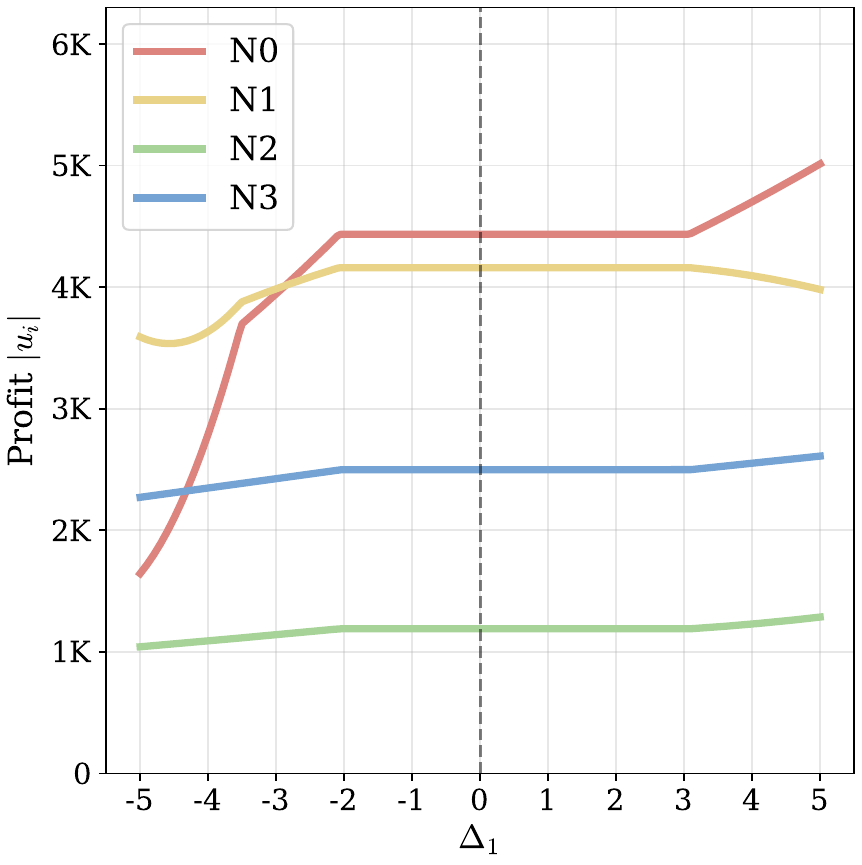}  
        \subcaption{Participant N1}\label{subfig:single1}  
    \end{minipage}%
    \hfill 
    \begin{minipage}{.24\textwidth}  
        \centering  
        \includegraphics[width=\linewidth]{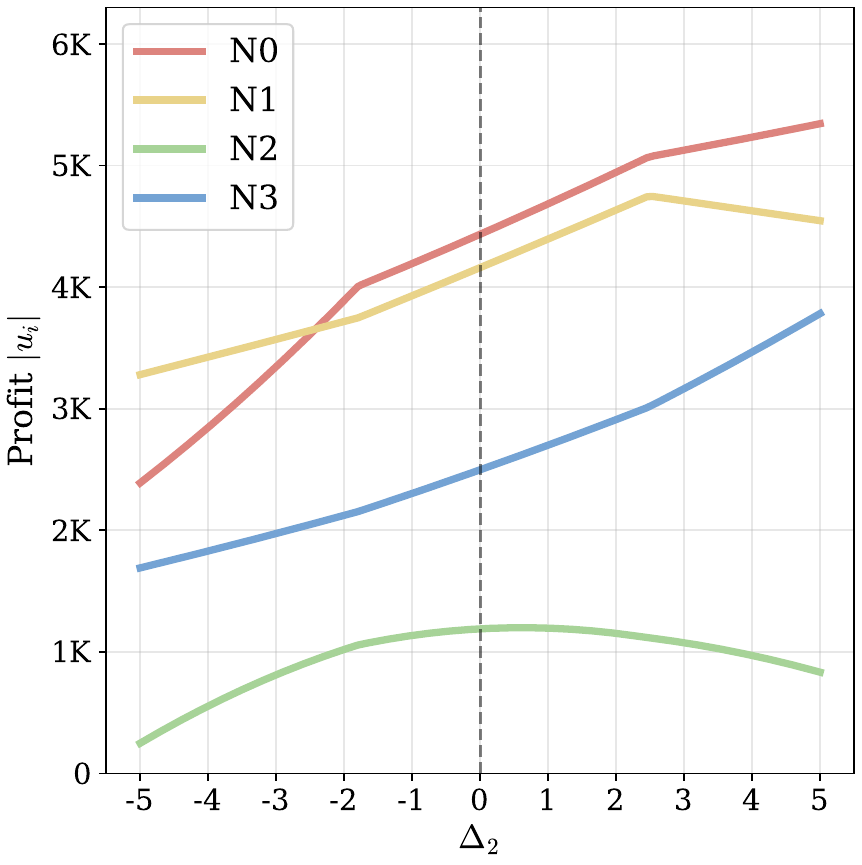}  
        \subcaption{Participant N2}\label{subfig:single2}  
    \end{minipage}%
    \hfill 
    \begin{minipage}{.24\textwidth}  
        \centering  
        \includegraphics[width=\linewidth]{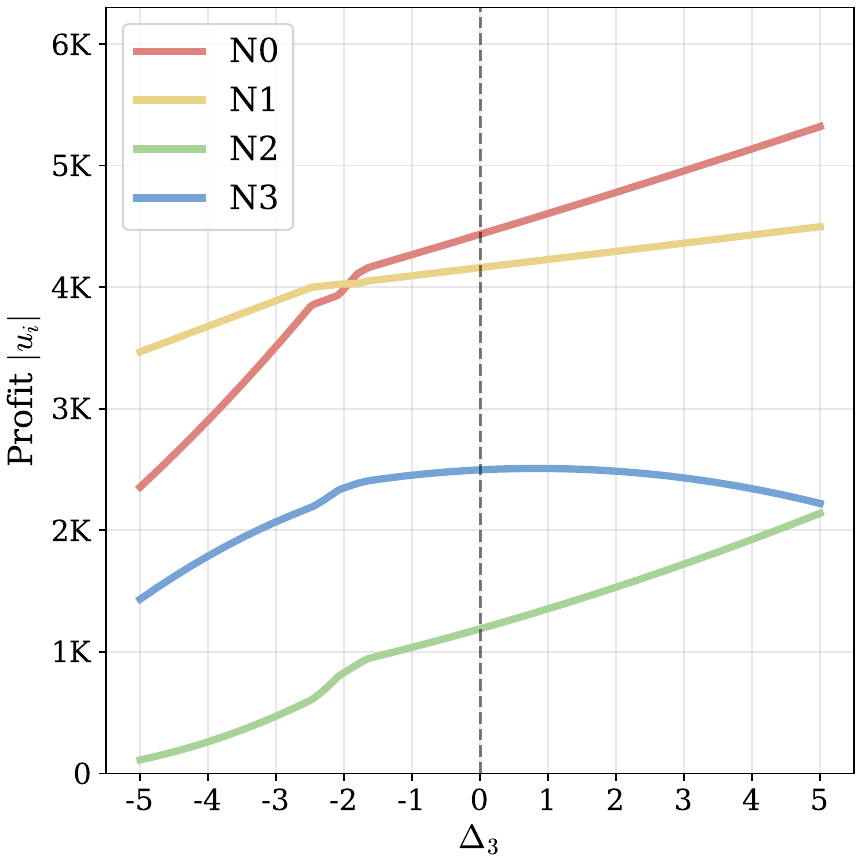}  
        \subcaption{Participant N3}\label{subfig:single3}  
    \end{minipage}%
    \caption{The impact of individual parameters misreporting on profit}  
    \label{fig:misreporting_single}
\end{figure}

Finally, Figure \ref{fig:misreporting_portfolio} further illustrates the profit variations of participants across 30 distinct experimental cases, each simulating coordinated strategic misreporting by all participants. In each case, participants attempted to maximize their profits following this procedure: (1) A random perturbation magnitude was generated for each participant (specific to each participant $i$); (2) A subvector of the original parameter vector $\boldsymbol{c}_i$ was randomly selected; (3) The perturbation was applied to this subvector to produce a manipulated parameter vector $\boldsymbol{c}_i'$. The horizontal dotted line represents the optimal profit values when all participants use their true parameters, while the dashed line illustrates the variations in participants' payoff values across different cases. As illustrated in the figure, when participants lack knowledge of others' cost parameters, random implementation of slight misreporting by all parties cannot ensure guaranteed profit improvement--it may even result in reduced profits. Consequently, in the absence of a strictly dominant misreporting strategy, adopting true information emerges as a prudent choice for all participants.
\begin{figure}[!ht]  
    \centering
    \includegraphics[height=45mm, width=160mm]{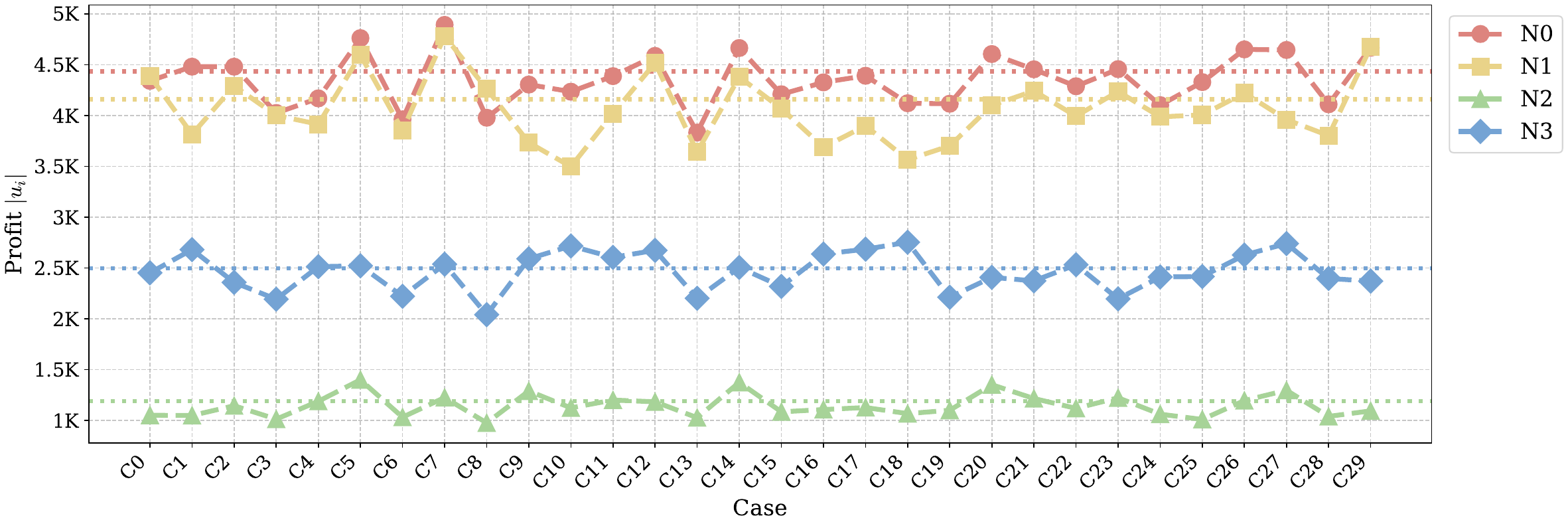}
    \caption{The impact of parameter misreporting portfolios on profit}  
    \label{fig:misreporting_portfolio}
\end{figure}

\section{Conclusion}

This paper focuses on addressing the dilemmas of information barriers and the lack of cooperation among participating parties in collective decision-making, particularly when employing distributed methods may incur conflicts between individual and collective decisions. We propose a mechanism-driven distributed coordination paradigm, consisting of an optimization phase and an incentive mechanism that aligns participant decisions with collective outcomes. Specifically, we develop a distributed optimization algorithm tailored to problems with coupled objectives and coupled constraints, which are increasingly prevalent in practice.
We establish rigorous convergence guarantees for the proposed algorithm. To address large-scale problems, we further 
propose a dimension reduction technique under affine constraint structures, leading to enhanced efficiency and reduced computational costs.
As demonstrated in Section \ref{sec:numerical}, existing algorithms exhibit inferior performance compared to our proposed method. Moreover, to incentivize rational participants to engage in tasks, we develop two incentive mechanisms: the shadow pricing mechanism using unit-price signals, and the VCG mechanism based on participation impact differentials. As revealed by Section \ref{sec:numerical}, neither mechanism dominates the other in terms of total incentive expenditure. 

Future research can proceed in the following directions. Firstly, despite empirical evidence of linear convergence for our proposed algorithm, a formal theoretical proof is lacking. Secondly, it would be interesting to extend our modeling to incorporate two or more types of participants (e.g., supply-side and demand-side participants). Finally, it is worth investigating how to improve the incentive mechanisms in terms of both theoretical properties and computational efficiency.

\bibliographystyle{apalike}
\bibliography{references}  

\newpage
\beginsupplement
\setcounter{page}{1}
\section{Proofs of Lemmas, Propositions, and the Main Theorem}
\label{sec:supplementary_proofs}

\subsection{Proof of Lemma \ref{lem:full_rank}} \label{proof:lem_fu}
In the homogeneous linear system 
\begin{equation}\label{eq:hom_lin_sys}
   M_1 \tilde{\boldsymbol{y}}+M_2 \tilde{\boldsymbol{\theta}} =0,
\end{equation}
if we set $\tilde{\boldsymbol{\theta}}=0$, then the residual equation  $M_1 \tilde{\boldsymbol{y}} =0$ admits a nontrivial solution $\tilde{\boldsymbol{y}}'$ when matrix $M_1$ is not of full column rank. However, since system \eqref{eq:hom_lin_sys} is an equivalent representation of $\boldsymbol{y}_{i} = \boldsymbol{\theta}_{e}, \ \boldsymbol{y}_{j}=\boldsymbol{\theta}_{e}(\forall \ {e}=\{i,j\} \in \mathcal{E}_{\mathrm{com}})$, the pair $\tilde{\boldsymbol{\theta}}=0, \tilde{\boldsymbol{y}} = \tilde{\boldsymbol{y}}' \neq 0$ can not be a solution to the system \eqref{eq:hom_lin_sys}. This contradiction implies that matrix $M_1$ must be of full column rank. Similarly, it can be proven that matrix $M_2$ is also of full column rank.

\subsection{Proof of Proposition \ref{pro:step_ADMM}} \label{proof:pro_st}

Recall the optimization problem \ref{eq:Distri_ADMM}, 
\begin{equation*}
    \begin{aligned}
         \min_{\tilde{\boldsymbol{y}}\in \tilde{\Omega},\tilde{\boldsymbol{\theta}}} \quad & H(\tilde{\boldsymbol{y}})+G(\tilde{\boldsymbol{\theta}}) \\
         {\rm s.t.} \quad \ & \  \hat{A}\tilde{\boldsymbol{y}} =\boldsymbol{d}, \quad \leftrightarrow \quad \lambda, \\
         & \ M_1 \tilde{\boldsymbol{y}}+ M_2 \tilde{\boldsymbol{\theta}} = 0, \quad \leftrightarrow \quad \mu.
    \end{aligned}
\end{equation*}
In the context of distributed computing methods, the constraint $\hat{A}\tilde{\boldsymbol{y}}=\boldsymbol{d}$ is challenging to directly decompose and compute independently on each participant. Therefore, we introduce a new variable $\boldsymbol{\nu}_i$ to decompose this constraint, thereby transforming the optimization problem as follows,
\begin{equation*}
    \begin{aligned}
         \min_{\tilde{\boldsymbol{y}} \in \tilde{\Omega},\tilde{\boldsymbol{\theta}}} \quad & H(\tilde{\boldsymbol{y}})+G(\tilde{\boldsymbol{\theta}}) \\
         {\rm s.t.} \quad \ & \  \hat{A}_i\boldsymbol{y}_i =\boldsymbol{\nu}_i, \quad \leftrightarrow \quad \varrho_i \quad \forall i,\\
         & \ M_1 \tilde{\boldsymbol{y}}+ M_2 \tilde{\boldsymbol{\theta}} = 0, \quad \leftrightarrow \quad \mu , \\
         & \sum\nolimits_i \boldsymbol{\nu}_i = \boldsymbol{d}.
    \end{aligned}
\end{equation*}
The augmented Lagrangian function is expressed as
\begin{equation*}
    \begin{aligned}
         \mathcal{L}_{\rho,\sigma} = & H(\tilde{\boldsymbol{y}})+\mu^{\rm T}\left(M_1 \tilde{\boldsymbol{y}}+ M_2 \tilde{\boldsymbol{\theta}}\right) + \frac{\rho}{2}\left\|M_1 \tilde{\boldsymbol{y}}+ M_2 \tilde{\boldsymbol{\theta}}\right\|^2 + \sum\nolimits_i \left(\varrho_i^{\rm T}\left(\hat{A}_i \boldsymbol{y}_i-\boldsymbol{\nu}_i\right)+ \frac{\sigma}{2} \left\| \hat{A}_i\boldsymbol{y}_i-\boldsymbol{\nu}_i\right\|^2\right) \\
         = & \sum\nolimits_i \left( f_i(\boldsymbol{y}_i) + \varrho_i^{\rm T} \left(\hat{A}_i \boldsymbol{y}_i-\boldsymbol{\nu}_i \right) + \frac{\sigma}{2} \left\| \hat{A}_i\boldsymbol{y}_i-\boldsymbol{\nu}_i\right\|^2 \right) \\ 
         & + \sum\nolimits_{e=\{i,j\} \in \mathcal{E}}\left[\mu_{e,i} ^{\rm T}( \boldsymbol{y}_{i}-\boldsymbol{\theta}_{e})+ \mu_{e,j}^{\rm T}(\boldsymbol{y}_{j}-\boldsymbol{\theta}_{e}) + \frac{\rho}{2} (\|\boldsymbol{y}_{i}-\boldsymbol{\theta}_{e}\|^2 + \|\boldsymbol{y}_{j}-\boldsymbol{\theta}_{e}\|^2)\right] 
    \end{aligned}
\end{equation*}
subject to $\sum_i \boldsymbol{\nu}_i =\boldsymbol{d}$ and $\boldsymbol{y}_i \in \tilde{\Omega}_i$ for any $i$. Thus, the iteration with respect to the variables $\boldsymbol{y}_i$, $\boldsymbol{\theta}_e$, and $\boldsymbol{\nu}_i$ of the classic ADMM is given by
\begin{subequations}
    \begin{align}
    \begin{split} \label{eq:group_y}
         & \boldsymbol{y}_i(k+1) =  \argmin\limits_{\boldsymbol{y}_i \in \tilde{\Omega}_i} \left\{f_i(\boldsymbol{y}_i) + \varrho_i(k)^{\rm T} \left(\hat{A}_i \boldsymbol{y}_i-\boldsymbol{\nu}_i(k) \right) + \frac{\sigma}{2} \left\| \hat{A}_i\boldsymbol{y}_i-\boldsymbol{\nu}_i(k)\right\|^2 \right. \\
        & \hspace{28mm} +  \left.\sum\nolimits_{e\in \mathcal{E}: e\ni i} \boldsymbol{y}_{i}^{\rm T}\left(\mu_{e,i}(k)-\rho\boldsymbol{\theta}_{e}(k)\right) + \frac{\rho}{2}{\rm deg}(i)\|\boldsymbol{y}_{i}\|^2 \right\}, \quad \forall i 
    \end{split} \\
        & \boldsymbol{\theta}_{e}(k+1) =  \argmin\limits_{ e=\{i,j\}, \boldsymbol{\theta}_{e}} \left\{ -\boldsymbol{\theta}_{e}^{\rm T}\left(\mu_{e,i}(k) + \mu_{e,j}(k) + \rho \boldsymbol{y}_{i}(k+1)+\rho \boldsymbol{y}_{j}(k+1)\right)+\rho\left\|\boldsymbol{\theta}_{e}\right\|^2\right\}, \quad \forall e \label{eq:theta_e_k} \\
        & \left\{\boldsymbol{\nu}_i(k+1)| \ i \in [N] \right\} = \argmin\limits_{\sum\nolimits_i \boldsymbol{\nu}_i =\boldsymbol{d}} \left\{ \sum\nolimits_i \left( - \varrho_i(k)^{\rm T}\boldsymbol{\nu}_i + \frac{\sigma}{2} \left\| \hat{A}_i\boldsymbol{y}_i(k+1)-\boldsymbol{\nu}_i\right\|^2 \right)  \right\}  \label{eq:group_nu}\\
        & \varrho_i(k+1)  = \varrho_i(k) + \sigma \left(\hat{A}_i \boldsymbol{y}_i(k+1) -\boldsymbol{\nu}_i(k+1)\right), \quad \forall i  \label{eq:varrho_k_1} \\
        & \mu_{e,s}(k+1)  = \mu_{e,s}(k) + \rho (\boldsymbol{y}_s(k+1) - \boldsymbol{\theta}_{e}(k+1)), \quad \forall e \in \mathcal{E}_{\rm com}, s \in e \label{eq:mu_k_1}
\end{align}
\end{subequations}

Note that the process of minimizing the equation \eqref{eq:group_nu} with respect to the variable group $\boldsymbol{\nu}_i$ involves dealing with a separable quadratic function along with a single equality constraint, which allows for an analytical solution. Specifically, the minimum value is achieved when
\begin{equation*}
    \boldsymbol{\nu}_i^{*} = \hat{A}_i \boldsymbol{y}_i + \frac{\varrho_i(k)-\lambda}{\sigma}, \quad \forall i,
\end{equation*}
where $\lambda$ is a scalar value acting as a Lagrange multiplier(Here, $\lambda$ essentially serves as the dual variable corresponding to the original constraint $\sum\nolimits_i \hat{A}_i \boldsymbol{y}_i=\boldsymbol{d}$), selected to ensure that the condition $\sum_i \boldsymbol{\nu}_i^* = \boldsymbol{d}$ holds. Thus, under the optimal values $\boldsymbol{y}_i(k+1)$, the optimal values $\lambda(k+1)$ and $\boldsymbol{\nu}_i(k+1)$ at $k$-th iteration are given by
\begin{align}
   & \lambda(k+1) = \frac{1}{N}\sum\nolimits_i \varrho_i(k) + \frac{\sigma}{N}\left(\sum\nolimits_i \hat{A}_i \boldsymbol{y}_i(k+1) - \boldsymbol{d}\right), \\
    & \boldsymbol{\nu}_i(k+1) = \hat{A}_i \boldsymbol{y}_i(k+1) + \frac{\varrho_i(k)-\lambda(k+1)}{\sigma}, \quad \forall i. \label{eq:nu_k_1}
\end{align}
By comparing equation \eqref{eq:nu_k_1} and equation \eqref{eq:varrho_k_1}, it can be obtained that
\begin{equation*}
    \varrho_i(k+1) = \lambda(k+1), \quad \forall i.
\end{equation*}
Therefore, we can compute the expression for $\boldsymbol{\nu}_i$ in equation \eqref{eq:group_y},
\begin{equation*}
    \begin{aligned}
        \boldsymbol{\nu}_i(k) & = \hat{A}_i \boldsymbol{y}_i(k) + \frac{\varrho_i(k-1)-\lambda(k)}{\sigma} \\
        & = \hat{A}_i \boldsymbol{y}_i(k) + \frac{\lambda(k-1)-\lambda(k)}{\sigma} \\
        & = \hat{A}_i \boldsymbol{y}_i(k) -\frac{1}{N}\left(\sum\nolimits_i \hat{A}_i \boldsymbol{y}_i(k) - \boldsymbol{d}\right) \\
        & = \hat{A}_i \boldsymbol{y}_i(k) - \eta(k).
    \end{aligned}
\end{equation*}
Then, the equation \eqref{eq:group_y}-\eqref{eq:mu_k_1} becomes
\begin{align*}
    \begin{split}
         & \boldsymbol{y}_i(k+1) =  \argmin\limits_{\boldsymbol{y}_i \in \tilde{\Omega}_i} \left\{f_i(\boldsymbol{y}_i) + \lambda(k)^{\rm T} \hat{A}_i \boldsymbol{y}_i + \frac{\sigma}{2} \left\| \hat{A}_i\boldsymbol{y}_i- \hat{A}_i \boldsymbol{y}_i(k) +\eta(k) \right\|^2 \right. \\
        & \hspace{28mm} +  \left.\sum\nolimits_{e\in \mathcal{E}: e\ni i} \boldsymbol{y}_{i}^{\rm T}\left(\mu_{e,i}(k)-\rho\boldsymbol{\theta}_{e}(k)\right) + \frac{\rho}{2}{\rm deg}(i)\|\boldsymbol{y}_{i}\|^2 \right\}, \quad \forall i 
    \end{split}  \\
        & \boldsymbol{\theta}_{e}(k+1) =  \argmin\limits_{ e=\{i,j\},  \boldsymbol{\theta}_{e} } \left\{ -\boldsymbol{\theta}_{e}^{\rm T}\left(\mu_{e,i}(k) + \mu_{e,j}(k) + \rho \boldsymbol{y}_{i}(k+1)+\rho \boldsymbol{y}_{j}(k+1)\right)+\rho\left\|\boldsymbol{\theta}_{e}\right\|^2\right\}, \quad \forall e \\
        & \lambda(k+1)  = \lambda(k) + \sigma \eta(k+1)   \\
        & \mu_{e,s}(k+1)  = \mu_{e,s}(k) + \rho (\boldsymbol{y}_s(k+1) - \boldsymbol{\theta}_{e}(k+1)), \quad \forall e \in \mathcal{E}_{\rm com}, s \in e
\end{align*}

In the framework of distributed optimization, the existing forms of $\lambda(k)$ and $\eta(k)$ are not accessible to each participant, as mentioned in the main text. Therefore, we utilize $\lambda_i(k)$ and $\eta_i(k)$ as substitutes for $\lambda(k)$ and $\eta(k)$ at each agent. According to the distributed average consensus mechanism \citep{falsone2020tracking,kia2019tutorial}, we propose:
\begin{equation*}
    \begin{aligned}
        & \lambda_i(k+1) = \sum_{s \in \mathcal{N}_i} w_{is} \lambda_s(k) +\sigma \eta_i(k) = l_i(k) + \sigma \eta_i(k), \\
        & \eta_i(k+1) = \sum_{s \in \mathcal{N}_i} w_{is} \eta_s(k) + \hat{A}_i\boldsymbol{y}_i(k+1) - \hat{A}_i\boldsymbol{y}_i(k) = \gamma_i(k) + \hat{A}_i\boldsymbol{y}_i(k+1) - \hat{A}_i\boldsymbol{y}_i(k).
    \end{aligned}
\end{equation*}
This represents the consensus process described in lines \ref{line1_9}-\ref{line1_10} and lines \ref{line1_12}-\ref{line1_13} of our Algorithm \ref{alg:TCADMM}. After making local copies of the public variables and introducing consensus, the equation \eqref{eq:group_y}-\eqref{eq:mu_k_1} transforms into
\begin{align}
    \begin{split}\label{eq:group_y_new}
         & \boldsymbol{y}_i(k+1) =  \argmin\limits_{\boldsymbol{y}_i \in \tilde{\Omega}_i} \left\{f_i(\boldsymbol{y}_i) + l_i(k)^{\rm T} \hat{A}_i \boldsymbol{y}_i + \frac{\sigma}{2} \left\| \hat{A}_i\boldsymbol{y}_i- \hat{A}_i \boldsymbol{y}_i(k) +\gamma_i(k) \right\|^2 \right. \\
        & \hspace{28mm} +  \left.\sum\nolimits_{e\in \mathcal{E}: e\ni i} \boldsymbol{y}_{i}^{\rm T}\left(\mu_{e,i}(k)-\rho\boldsymbol{\theta}_{e}(k)\right) + \frac{\rho}{2}{\rm deg}(i)\|\boldsymbol{y}_{i}\|^2 \right\}, \quad \forall i 
    \end{split}  \\
        & \boldsymbol{\theta}_{e}(k+1) =  \argmin\limits_{ e=\{i,j\},  \boldsymbol{\theta}_{e} } \left\{ -\boldsymbol{\theta}_{e}^{\rm T}\left(\mu_{e,i}(k) + \mu_{e,j}(k) + \rho \boldsymbol{y}_{i}(k+1)+\rho \boldsymbol{y}_{j}(k+1)\right)+\rho\left\|\boldsymbol{\theta}_{e}\right\|^2\right\}, \quad \forall e \notag \\
        & \mu_{e,s}(k+1)  = \mu_{e,s}(k) + \rho (\boldsymbol{y}_s(k+1) - \boldsymbol{\theta}_{e}(k+1)), \quad \forall e \in \mathcal{E}_{\rm com}, s \in e \notag \\ 
        & \eta_i(k+1) = \sum_{s \in \mathcal{N}_i} w_{is} \eta_s(k) + \hat{A}_i\boldsymbol{y}_i(k+1) - \hat{A}_i\boldsymbol{y}_i(k) , \quad \forall i  \notag \\
        & \lambda_i(k+1) = \sum_{s \in \mathcal{N}_i} w_{is} \lambda_s(k) +\sigma \eta_i(k), \quad \forall i  \notag
\end{align}
If we convert the above iterative process into a compact form using the notation in section \ref{subsec:con_ana}, we obtain the iterative process \eqref{eq:step_ADMM} described in Proposition \ref{pro:step_ADMM}.

However, in the iterative process, the solution for variable $\boldsymbol{\theta}_{e}$ and the update of $\mu_{e,s}$ depend on the structure of the edges $e$ in the communication network $\mathcal{E}_{\rm com}$, making them similarly infeasible to perform at individual agents. We did not employ local copies for these two parts at each agent because they can be simplified by the following method. 

Considering equation \eqref{eq:theta_e_k}, as it is an unconstrained quadratic optimization problem, it has an explicit solution,
\begin{equation}\label{eq:theta_express_k}
    \boldsymbol{\theta}_{e}(k+1) = \frac{\rho^{-1}(\mu_{e,i}(k)+\mu_{e,j}(k))+\boldsymbol{y}_{i}(k+1)+\boldsymbol{y}_{j}(k+1)}{2}, \quad e=\{i,j\}.
\end{equation}
Then we simplify the above iterations by eliminating variables $\boldsymbol{\theta}_{e}(k+1)$. First, note that by plugging the equation \eqref{eq:theta_express_k} into the equation \eqref{eq:mu_k_1}, we get
\begin{equation}\label{eq:mu_k_new}
\begin{aligned}
    \mu_{e,i}(k+1) & = \mu_{e,i}(k) +\rho \boldsymbol{y}_{i}(k+1) - \frac{1}{2}\left(\mu_{e,i}(k)+\mu_{e,j}(k) + \rho \boldsymbol{y}_{i}(k+1)+\rho\boldsymbol{y}_{j}(k+1) \right) \\
    & = \frac{1}{2}( \mu_{e,i}(k) + \mu_{e,j}(k) +\rho\boldsymbol{y}_{i}(k+1)- \rho \boldsymbol{y}_{j}(k+1)), \quad e \in \{i,j\}.
\end{aligned}
\end{equation}
Based on the symmetry of nodes on edges in the communication network, by exchanging $i$ and $j$ in the above equation and summing the two resulting equations, we can obtain for all $k \geq 0$,
\begin{equation}\label{eq:sum_eq_zero}
    \mu_{e,i}(k+1) + \mu_{e,j}(k+1) =  0.
\end{equation}
Without loss of generality we may adopt the initialization $\mu_{e,i}(0) = \mu_{e,j}(0) = 0$ for all $e = \{i, j\} \in \mathcal{E}_{\rm com}$, and then we get
\begin{equation*}
    \boldsymbol{\theta}_{e}(k+1)=\frac{\boldsymbol{y}_{i}(k+1)+\boldsymbol{y}_{j}(k+1)}{2}.
\end{equation*}
Then, we eliminate $\boldsymbol{\theta}_{e}$ in the equation \eqref{eq:group_y_new} and obtain
\begin{equation}\label{eq:group_y_new2}
    \begin{split}
         & \boldsymbol{y}_i(k+1) =  \argmin\limits_{\boldsymbol{y}_i \in \tilde{\Omega}_i} \left\{f_i(\boldsymbol{y}_i) + l_i(k)^{\rm T} \hat{A}_i \boldsymbol{y}_i + \frac{\sigma}{2} \left\| \hat{A}_i\boldsymbol{y}_i- \hat{A}_i \boldsymbol{y}_i(k) +\gamma_i(k) \right\|^2 \right. \\
        & \hspace{28mm} +  \left.\sum\nolimits_{e\in \mathcal{E}: e=\{i,j\}} \boldsymbol{y}_{i}^{\rm T}\left(\mu_{e,i}(k)-\frac{\rho}{2}(\boldsymbol{y}_{i}(k)+\boldsymbol{y}_{j}(k))\right) + \frac{\rho}{2}{\rm deg}(i)\|\boldsymbol{y}_{i}\|^2 \right\}, \quad \forall i. 
    \end{split}
\end{equation}
To further simplify the iterations, we introduce the auxiliary variable
\begin{equation*}
    v_{i}(k) = \frac{1}{{\rm deg}(i)} \sum_{e \in \mathcal{E}: e=\{i,j\}} \left[ -\rho^{-1} \mu_{e,i}(k) +\frac{1}{2}(\boldsymbol{y}_{i}(k)+\boldsymbol{y}_{j}(k))\right].
\end{equation*}
Then the equation \eqref{eq:group_y_new2} can be written as
\begin{equation*}
\begin{split}
    \boldsymbol{y}_i(k+1) & =  \argmin\limits_{\boldsymbol{y}_i \in \tilde{\Omega}_i} \left\{f_i(\boldsymbol{y}_i) + l_i(k)^{\rm T} \hat{A}_i \boldsymbol{y}_i + \frac{\sigma}{2} \left\| \hat{A}_i\boldsymbol{y}_i- \hat{A}_i \boldsymbol{y}_i(k) +\gamma_i(k) \right\|^2  \right. \\
    & \hspace{28mm} \left. -\rho {\rm deg}(i) \left(v_i(k)^{\rm T}\boldsymbol{y}_i -\frac{1}{2} \|\boldsymbol{y}_{i}\|^2 \right) \right\} \\
    & =  \argmin\limits_{\boldsymbol{y}_i \in \tilde{\Omega}_i} \left\{f_i(\boldsymbol{y}_i) + l_i(k)^{\rm T} \hat{A}_i \boldsymbol{y}_i + \frac{\sigma}{2} \left\| \hat{A}_i\boldsymbol{y}_i- \hat{A}_i \boldsymbol{y}_i(k) +\gamma_i(k) \right\|^2 + \frac{\rho {\rm deg}(i)}{2}\|\boldsymbol{y}_{i}-v_i(k)\|^2 \right\}\!.
\end{split} 
\end{equation*}
This is the sub-optimization problem $\mathcal{S}_{i,k}$ in line \ref{line1_11} of Algorithm \ref{alg:TCADMM}. Moreover, for the auxiliary variable $v_i$, we have 
\begin{equation*}
    \begin{aligned}
        & v_{i}(k+1)-v_{i}(k) \\
        = & \frac{1}{{\rm deg}(i)} \sum_{e\in \mathcal{E}: e\in \{i,j\}} \left[ - \rho^{-1} (\mu_{e,i}(k+1)-\mu_{e,j}(k))+ \frac{1}{2}(\boldsymbol{y}_{i}(k+1)+\boldsymbol{y}_{j}(k+1)- \boldsymbol{y}_{i}(k) - \boldsymbol{y}_{j}(k))\right] \\
        = & \frac{1}{{\rm deg}(i)} \sum_{e\in \mathcal{E}: e\in \{i,j\}} \left[ - \frac{1}{2\rho} (\mu_{e,i}(k)+\mu_{e,j}(k))+\boldsymbol{y}_{j}(k+1)- \frac{1}{2}( \boldsymbol{y}_{i}(k) + \boldsymbol{y}_{j}(k))\right] \\
        = & \frac{1}{{\rm deg}(i)} \sum_{j \in \mathcal{N}_{i}}\left(\boldsymbol{y}_j(k+1)-\frac{1}{2}\boldsymbol{y}_j(k)\right) -\frac{1}{2}\boldsymbol{y}_i(k),
    \end{aligned}
\end{equation*}
where the second equality uses the update formula \eqref{eq:mu_k_new}, and the third equality employs equation \eqref{eq:sum_eq_zero}. Thus, the update formula for auxiliary variable $v_i$ is
\begin{equation*}
\begin{aligned}
    v_{i}(k+1) & = v_{i}(k) + \frac{1}{{\rm deg}(i)} \sum_{j \in \mathcal{N}_{i}}\left(\boldsymbol{y}_j(k+1)-\frac{1}{2}\boldsymbol{y}_j(k)\right) -\frac{1}{2}\boldsymbol{y}_i(k) \\
    & =v_{i}(k) + \frac{1}{{\rm deg}(i)} \sum_{j \in \mathcal{N}_{i}}\delta_j(k+1) -\frac{1}{2}\boldsymbol{y}_i(k).
\end{aligned}
\end{equation*}
This is the update step in line \ref{line1_15} of Algorithm \ref{alg:TCADMM}. So far, we have shown that the steps in Proposition \ref{pro:step_ADMM} are consistent with the steps in our Algorithm \ref{alg:TCADMM}.

\subsection{Proof of Lemma \ref{lem:bounded}} \label{proof:lem_bo}

$(\romannumeral 1)$ We start by considering
\begin{equation*}
        \mathcal{W}\bar{\boldsymbol{\eta}}(k)  = (W \otimes I_{n_0})(\textbf{1}_{N} \otimes \bar{\eta}(k))  = (W \textbf{1}_{N} \otimes I_{n_0} \bar{\eta}(k))  =  \bar{\boldsymbol{\eta}}(k),
\end{equation*}
where the last equality relies on the row stochasticity stated in Assumption \ref{ass:com_network}. Then, based on the definition of the error term $\boldsymbol{\epsilon}_1(k+1)$ and the update formula for $\boldsymbol{\eta}(k+1)$ in equation \eqref{eq:step_ADMM_eta}, we have
\begin{equation*}
    \begin{aligned}
        \boldsymbol{\epsilon}_1(k+1) & = \boldsymbol{\eta}(k+1)-\bar{\boldsymbol{\eta}}(k+1) \\
        & = \mathcal{W}\boldsymbol{\eta}(k) + \hat{A}_b(\tilde{\boldsymbol{y}}(k+1)-\tilde{\boldsymbol{y}}(k)) - \bar{\boldsymbol{\eta}}(k+1) +\bar{\boldsymbol{\eta}}(k) -\bar{\boldsymbol{\eta}}(k) \\
        & = \mathcal{W}\boldsymbol{\eta}(k)  -\bar{\boldsymbol{\eta}}(k)  + (\hat{A}_b\tilde{\boldsymbol{y}}(k+1) - \bar{\boldsymbol{\eta}}(k+1)) -(\hat{A}_b\tilde{\boldsymbol{y}}(k) -\bar{\boldsymbol{\eta}}(k)) \\
        & = \mathcal{W}(\boldsymbol{\eta}(k)  -\bar{\boldsymbol{\eta}}(k)) + \boldsymbol{\xi}(k+1) -\boldsymbol{\xi}(k) \\
        & = \mathcal{W}\boldsymbol{\epsilon}_1(k)  + \boldsymbol{\xi}(k+1) -\boldsymbol{\xi}(k).
    \end{aligned}
\end{equation*}
Moreover, since $\|\mathcal{W}\|=\|W \otimes I_{n_0}\|=\|W\|\|I_{n_0}\|=1$, the above iterative relationship does not exhibit a contraction property when term $\boldsymbol{\xi}(k+1) -\boldsymbol{\xi}(k)$ can be controlled. We further consider the compression matrix $\mathcal{W}_0$, whose norm satisfies $\|\mathcal{W}_0\|=\|W-\frac{1}{N}\textbf{1}_N\textbf{1}_N^{\rm T}\|\|I_{n_0}\| < 1$, and additionally fulfills the following equation
\begin{equation*}
\begin{aligned}
    \mathcal{W}_0 \boldsymbol{\epsilon}_1(k) & = \mathcal{W} \boldsymbol{\epsilon}_1(k) - \left(\frac{1}{N}\textbf{1}_N\textbf{1}_N^{\rm T} \otimes I_{n_0}\right)\boldsymbol{\epsilon}_1(k) \\
    & = \mathcal{W} \boldsymbol{\epsilon}_1(k) - \left(\frac{1}{N}\textbf{1}_N\textbf{1}_N^{\rm T} \otimes I_{n_0}\right)(\boldsymbol{\eta}(k)-\bar{\boldsymbol{\eta}}(k)) \\
    & = \mathcal{W} \boldsymbol{\epsilon}_1(k) - \left(\frac{1}{N}\textbf{1}_N\textbf{1}_N^{\rm T} \otimes I_{n_0}\right)\boldsymbol{\eta}(k)-\textbf{1}_N \otimes \bar{\eta}(k) \\
    & = \mathcal{W} \boldsymbol{\epsilon}_1(k) - \textbf{1}_N \otimes \left(\frac{1}{N}\sum\nolimits_i\eta_i(k)\right)-\textbf{1}_N \otimes \bar{\eta}(k) \\
    & = \mathcal{W} \boldsymbol{\epsilon}_1(k).
\end{aligned}
\end{equation*}
Therefore, we have
\begin{equation}\label{eq:epsilon_1}
    \boldsymbol{\epsilon}_1(k+1)= \mathcal{W}_0\boldsymbol{\epsilon}_1(k)  + \boldsymbol{\xi}(k+1) -\boldsymbol{\xi}(k).
\end{equation}
Thus, the evolution of $\boldsymbol{\epsilon}_1$ is asymptotically stable.

$(\romannumeral 2)$ Similarly, following the same process outlined above, we obtain
\begin{equation*}
    \begin{aligned}
        &\mathcal{W}\bar{\boldsymbol{\lambda}}(k)  = (W \otimes I_{MK})(\textbf{1}_{N} \otimes \bar{\lambda}(k)) =  \bar{\boldsymbol{\lambda}}(k), \\
        &\mathcal{W}_0 \boldsymbol{\epsilon}_2(k) = \mathcal{W} \boldsymbol{\epsilon}_2(k).
    \end{aligned}
\end{equation*}
Subsequently, we consider the iteration process of the error term $\boldsymbol{\epsilon}_2(k+1)$,
\begin{equation}\label{eq:epsilon_2}
    \begin{aligned}
        \boldsymbol{\epsilon}_2(k+1)  & = \boldsymbol{\lambda}(k+1)-\bar{\boldsymbol{\lambda}}(k+1) \\
        & = \mathcal{W}\boldsymbol{\lambda}(k)+ \sigma \boldsymbol{\eta}(k+1) - \textbf{1}_N \otimes \bar{\lambda}(k+1) \\
        & = \mathcal{W}\boldsymbol{\lambda}(k)+ \sigma \boldsymbol{\eta}(k+1) - \textbf{1}_N \otimes (\bar{\lambda}(k)+\sigma \bar{\eta}(k+1)) \\
        & = \mathcal{W}\boldsymbol{\lambda}(k) - \bar{\boldsymbol{\lambda}}(k) + \sigma \boldsymbol{\eta}(k+1) - \sigma \bar{\boldsymbol{\eta}}(k+1) \\
        & = \mathcal{W}\boldsymbol{\epsilon}_2(k) + \sigma \boldsymbol{\epsilon}_1(k+1) \\
        & = \mathcal{W}_0\boldsymbol{\epsilon}_2(k) + \sigma \boldsymbol{\epsilon}_1(k+1).
    \end{aligned}
\end{equation}
The third equation relies on the equality,
\begin{equation}\label{eq:lemma2_in_fal}
  \bar{\lambda}(k+1)  = \bar{\lambda}(k)+ \sigma \bar{\eta}(k+1).
\end{equation}
For detailed proofs, readers are referred to Lemma 2 in \cite{falsone2020tracking}.

$(\romannumeral 3)$ Given that $\tilde{\boldsymbol{y}}(k)$ belongs to the set $\tilde{\Omega}$ for all $k$, the sequence $\{\tilde{\boldsymbol{y}}(k)\}_{k \geq 0}$ is confined within bounds. Furthermore, according to the following equation
\begin{equation}\label{eq:lemma1_in_fal}
   N\bar{\eta}(k)= \sum\nolimits_i \hat{A}_i\boldsymbol{y}_i(k+1) - \boldsymbol{d},
\end{equation}
where the proof relies on the double stochasticity property stated in Assumption \ref{ass:com_network}. For detailed proofs, readers are referred to Lemma 1 in \cite{falsone2020tracking}. Then, the variable $\bar{\eta}(k)$ is bounded, and the sequence $\{\bar{\boldsymbol{\eta}}(k)\}_{k \geq 0}$ is similarly bounded by the definition. Consequently, since $\boldsymbol{\xi}(k) = \hat{A}_b\tilde{\boldsymbol{y}}(k) -\tilde{\boldsymbol{\eta}}(k)$, it follows that $\{\boldsymbol{\xi}(k)\}_{k \geq 0}$ is also a bounded sequence.

In light of the boundedness of $\{\boldsymbol{\xi}(k)\}_{k \geq 0}$, it is straightforward to deduce that the sequence $\{\boldsymbol{\xi}(k+1)-\boldsymbol{\xi}(k)\}_{k \geq 0}$ remains bounded as well. Consequently, leveraging the compression matrix $\mathcal{W}_0$ in equation \eqref{eq:epsilon_1}, we can ascertain that $\{\boldsymbol{\epsilon}_1(k)\}_{k \geq 0}$ is a bounded sequence. Ultimately, using the same reason applied to equation \eqref{eq:epsilon_2}, combined with the boundedness of $\boldsymbol{\epsilon}_1(k)$, we conclude that $\{\boldsymbol{\epsilon}_2(k)\}_{k \geq 0}$ is also bounded.

\subsection{Preparation for Proposition \ref{pro:Lyapunov}}
\label{proof:pro_Lya_before}

\begin{lemma}\label{lem:pos_def_mat}
   The following matrix is a symmetric positive-definite matrix,
   \begin{equation*}
       \mathcal{Q} = 
       \begin{pmatrix}
        2 I_a  &  (I_a - \mathcal{W}_0)^{-1}-2I_a \\
        (I_a - \mathcal{W}_0)^{-1}-2I_a  & \quad (I_a - \mathcal{W}_0)^{-2}-2 (I_a - \mathcal{W}_0)^{-1}+2I_a
       \end{pmatrix},
   \end{equation*}
   where $a=Nn_0$.
\end{lemma}

\proof
 First of all, given that the symmetric matrix $\mathcal{W}_0$ satisfies $\|\mathcal{W}_0\| <1$, it implies that all eigenvalues of $\mathcal{W}_0$ lie within the open unit circle. Consequently, the matrix $I_a-\mathcal{W}_0$ is invertible, which ensures that the matrix $\mathcal{Q}$ is well-defined. Additionally, $I_a-\mathcal{W}_0$ is symmetric and positive definite, i.e., $(I_a-\mathcal{W}_0)^{\rm T} = I_a-\mathcal{W}_0$ and $I_a-\mathcal{W}_0 \succ 0$.

Then, for symmetry, it is easy to see that matrices $(I_a - \mathcal{W}_0)^{-1}$ and $(I_a - \mathcal{W}_0)^{-2}$ are all symmetric, so the symmetry of $\mathcal{Q}$ is proved.

Finally, since the matrix $2I_a$ is invertible, we consider the Schur complement with respect to the matrix $2I_a$,
\begin{equation*}
    \begin{aligned}
       & (I_a - \mathcal{W}_0)^{-2}-2 (I_a - \mathcal{W}_0)^{-1}+2I_a - \frac{1}{2} \left((I_a - \mathcal{W}_0)^{-1}-2I_a\right)^{2} \\
       = & (I_a - \mathcal{W}_0)^{-2}-2 (I_a - \mathcal{W}_0)^{-1}+2I_a- \frac{1}{2}(I_a - \mathcal{W}_0)^{-2} + 2(I_a - \mathcal{W}_0)^{-1} -2I_a \\
       = & \frac{1}{2}(I_a - \mathcal{W}_0)^{-2} \succ 0.
    \end{aligned}
\end{equation*}
By the Schur complement lemma, the positive definiteness of $\mathcal{Q}$ is proved.
\endproof

\begin{lemma}\label{lem:inequality_2}
    Under the Assumption \ref{ass:const_set}-\ref{ass:com_network}, 
    we have the following equation,
    \begin{equation}\label{eq:lambda_A_y}
        \begin{aligned}
           & (\boldsymbol{\lambda}(k+1)-\boldsymbol{\lambda}^*)^{\rm T} \hat{A}_b \left(\tilde{\boldsymbol{y}}(k+1)-\tilde{\boldsymbol{y}}^*\right) \\
            = & \frac{1}{2\sigma} \left(\|\bar{\boldsymbol{\lambda}}(k+1)-\boldsymbol{\lambda}^*\|^2  + \|\bar{\boldsymbol{\lambda}}(k+1)-\bar{\boldsymbol{\lambda}}(k)\|^2 - \|\bar{\boldsymbol{\lambda}}(k)-\boldsymbol{\lambda}^*\|^2 \right)+ \boldsymbol{\epsilon}_2(k+1) ^{\rm T}\left( \boldsymbol{\xi}(k+1)-\hat{A}_b \tilde{\boldsymbol{y}}^*\right)
        \end{aligned}
    \end{equation}
\end{lemma}

\proof
First, we add and subtract term $\bar{\boldsymbol{\lambda}}(k+1)^{\rm T}\hat{A}_b(\tilde{\boldsymbol{y}}(k+1)-\tilde{\boldsymbol{y}}^*)$ to the left-hand side in equation \eqref{eq:lambda_A_y} to get
\begin{equation*}
    (\bar{\boldsymbol{\lambda}}(k+1)-\boldsymbol{\lambda}^*)^{\rm T}\hat{A}_b(\tilde{\boldsymbol{y}}(k+1)-\tilde{\boldsymbol{y}}^*) + (\boldsymbol{\lambda}(k+1)-\bar{\boldsymbol{\lambda}}(k+1))^{\rm T}\hat{A}_b(\tilde{\boldsymbol{y}}(k+1)-\tilde{\boldsymbol{y}}^*).
\end{equation*}
For the first term, we have
\begin{equation}\label{eq:lambda_A_y_2_1}
    \begin{aligned}
        (\bar{\boldsymbol{\lambda}}(k+1)-\boldsymbol{\lambda}^*)^{\rm T}\hat{A}_b(\tilde{\boldsymbol{y}}(k+1)-\tilde{\boldsymbol{y}}^*) & =   (\bar{\lambda}(k+1)-\lambda^*)^{\rm T}\hat{A}(\tilde{\boldsymbol{y}}(k+1)-\tilde{\boldsymbol{y}}^*) \\
        & = (\bar{\lambda}(k+1)-\lambda^*)^{\rm T}(\hat{A}\tilde{\boldsymbol{y}}(k+1)-\boldsymbol{d}) \\
        & = (\bar{\lambda}(k+1)-\lambda^*)^{\rm T}N\bar{\eta}(k+1) \\
        & = \frac{N}{\sigma}(\bar{\lambda}(k+1)-\lambda^*)^{\rm T}(\bar{\lambda}(k+1)-\bar{\lambda}(k)) \\
        & = \frac{1}{\sigma}(\bar{\boldsymbol{\lambda}}(k+1)-\boldsymbol{\lambda}^*)^{\rm T}(\bar{\boldsymbol{\lambda}}(k+1)-\bar{\boldsymbol{\lambda}}(k)).
    \end{aligned}
\end{equation}
The third and fourth equalities hold by equations \eqref{eq:lemma2_in_fal} and \eqref{eq:lemma1_in_fal} respectively. For the second term, we have
\begin{equation}\label{eq:lambda_A_y_2_2}
    \begin{aligned}
        & (\boldsymbol{\lambda}(k+1)-\bar{\boldsymbol{\lambda}}(k+1))^{\rm T}\hat{A}_b(\tilde{\boldsymbol{y}}(k+1)-\tilde{\boldsymbol{y}}^*) \\
        = &    (\boldsymbol{\lambda}(k+1)-\bar{\boldsymbol{\lambda}}(k+1))^{\rm T}(\hat{A}_b\tilde{\boldsymbol{y}}(k+1)-\bar{\boldsymbol{\eta}}(k+1)+\bar{\boldsymbol{\eta}}(k+1)-\hat{A}_b\tilde{\boldsymbol{y}}^*) \\
       = &    (\boldsymbol{\lambda}(k+1)-\bar{\boldsymbol{\lambda}}(k+1))^{\rm T}(\boldsymbol{\xi}(k+1)+\bar{\boldsymbol{\eta}}(k+1)-\hat{A}_b\tilde{\boldsymbol{y}}^*) \\
       =  &    (\boldsymbol{\lambda}(k+1)-\bar{\boldsymbol{\lambda}}(k+1))^{\rm T}(\boldsymbol{\xi}(k+1)-\hat{A}_b\tilde{\boldsymbol{y}}^*) + (\boldsymbol{\lambda}(k+1)-\bar{\boldsymbol{\lambda}}(k+1))^{\rm T}\bar{\boldsymbol{\eta}}(k+1) \\
        = &    (\boldsymbol{\lambda}(k+1)-\bar{\boldsymbol{\lambda}}(k+1))^{\rm T}(\boldsymbol{\xi}(k+1)-\hat{A}_b\tilde{\boldsymbol{y}}^*) + \sum\nolimits_i(\lambda_i(k+1)-\bar{\lambda}(k+1))^{\rm T}\bar{\eta}(k+1) \\
       = &   (\boldsymbol{\lambda}(k+1)-\bar{\boldsymbol{\lambda}}(k+1))^{\rm T}(\boldsymbol{\xi}(k+1)-\hat{A}_b\tilde{\boldsymbol{y}}^*). 
    \end{aligned}
\end{equation}
The second equation holds by definition of $\boldsymbol{\xi}(k+1)$. Then, by combining the equations \eqref{eq:lambda_A_y_2_1} and \eqref{eq:lambda_A_y_2_2}, we have
\begin{equation*}
\begin{aligned}
& (\boldsymbol{\lambda}(k+1)-\boldsymbol{\lambda}^*)^{\rm T} \hat{A}_b \left(\tilde{\boldsymbol{y}}(k+1)-\tilde{\boldsymbol{y}}^*\right) \\
  =&   \frac{1}{\sigma}(\bar{\boldsymbol{\lambda}}(k+1)-\boldsymbol{\lambda}^*)^{\rm T}(\bar{\boldsymbol{\lambda}}(k+1)-\bar{\boldsymbol{\lambda}}(k)) + (\boldsymbol{\lambda}(k+1)-\bar{\boldsymbol{\lambda}}(k+1))^{\rm T}(\boldsymbol{\xi}(k+1)-\hat{A}_b\tilde{\boldsymbol{y}}^*).
\end{aligned}
\end{equation*}

Finally, utilizing the following identity, 
\begin{equation*}
    \begin{aligned}
    & 2(\bar{\boldsymbol{\lambda}}(k+1)-\boldsymbol{\lambda}^*)^{\rm T}(\bar{\boldsymbol{\lambda}}(k+1)-\bar{\boldsymbol{\lambda}}(k)) \\
    = &
            \|\bar{\boldsymbol{\lambda}}(k+1)-\boldsymbol{\lambda}^*\|^2  + \|\bar{\boldsymbol{\lambda}}(k+1)-\bar{\boldsymbol{\lambda}}(k)\|^2 - \|\bar{\boldsymbol{\lambda}}(k)-\boldsymbol{\lambda}^*\|^2, 
        \end{aligned}
\end{equation*}
it can be deduced that equation \eqref{eq:lambda_A_y} holds, thereby proving the lemma.
\endproof

\begin{lemma}\label{pro:inequality_1}
    Under the Assumptions \ref{ass:const_set}-\ref{ass:com_network}, we have
    \begin{equation}\label{eq:upper-bound}
    \begin{aligned}
       & \mathcal{L}^1\left(\tilde{\boldsymbol{y}}(k+1),\tilde{\boldsymbol{\theta}}(k+1),\lambda^*,\mu^{*}\right) - \mathcal{L}^1\left(\tilde{\boldsymbol{y}}^*,\tilde{\boldsymbol{\theta}}^*,\lambda^*,\mu^{*}\right) \\
        \leq & -\frac{1}{\rho} \left( \mu(k+1)-\mu^*\right)^{\rm T}\left(  \mu(k+1)-\mu(k)\right) \\
        &- \rho \left( M_2(\tilde{\boldsymbol{\theta}}(k+1)-\tilde{\boldsymbol{\theta}}(k))\right)^{\rm T}\left( M_2(\tilde{\boldsymbol{\theta}}(k+1)-\tilde{\boldsymbol{\theta}}^*)\right)  \\
        & +\frac{1}{2\sigma}\left(\|\bar{\boldsymbol{\lambda}}(k)-\boldsymbol{\lambda}^*\|^2 - \|\bar{\boldsymbol{\lambda}}(k+1)-\bar{\boldsymbol{\lambda}}(k)\|^2 - \|\bar{\boldsymbol{\lambda}}(k+1)-\boldsymbol{\lambda}^*\|^2\right) \\
        & -\boldsymbol{\epsilon}_2(k+1) ^{\rm T}\left( \boldsymbol{\xi}(k+1)-\hat{A}_b \tilde{\boldsymbol{y}}^*\right),
    \end{aligned}
    \end{equation}
where $\boldsymbol{\lambda}^* = \mathbf{1}_N \otimes \lambda^*$.
\end{lemma}

\proof
Firstly, we transform equation \eqref{eq:step_ADMM_y} by letting 
\begin{equation*}
    \begin{aligned}
        & J_1(\tilde{\boldsymbol{y}})=H(\tilde{\boldsymbol{y}})+\mu(k)^{\rm T}\left(M_1 \tilde{\boldsymbol{y}}+ M_2 \tilde{\boldsymbol{\theta}}(k)\right) +\left(\mathcal{W}\boldsymbol{\lambda}(k)\right)^{\rm T}\hat{A}_b\tilde{\boldsymbol{y}},  \\
       &  J_2(\tilde{\boldsymbol{y}})= \frac{\rho}{2}\left\|M_1 \tilde{\boldsymbol{y}}+ M_2 \tilde{\boldsymbol{\theta}}(k)\right\|^2 + \frac{\sigma}{2} \left\| \hat{A}_b\tilde{\boldsymbol{y}}- \hat{A}_b\tilde{\boldsymbol{y}}(k)+\mathcal{W}\boldsymbol{\eta}(k)\right\|^2.
    \end{aligned}
\end{equation*}
It is straightforward to verify that $J_1$ and $J_2$ are convex functions with respect to $\tilde{\boldsymbol{y}}$, and the feasible region is a convex set. According to Lemma 4.1 in \cite{bertsekas2015parallel} and equations \eqref{eq:step_ADMM_eta}-\eqref{eq:step_ADMM_lambda}, it follows that
\begin{equation*}
    \begin{split}
          \tilde{\boldsymbol{y}}(k+1) & =  \argmin\limits_{\tilde{\boldsymbol{y}} \in \tilde{\Omega}} \left\{H(\tilde{\boldsymbol{y}})+\mu(k)^{\rm T}\left(M_1 \tilde{\boldsymbol{y}}+ M_2 \tilde{\boldsymbol{\theta}}(k)\right) + \rho\left(M_1 \tilde{\boldsymbol{y}}(k+1)+ M_2 \tilde{\boldsymbol{\theta}}(k)\right)^{\rm T}M_1\tilde{\boldsymbol{y}} \right. \\
        & \hspace{32mm} + \left(\mathcal{W}\boldsymbol{\lambda}(k)\right)^{\rm T}\hat{A}_b\tilde{\boldsymbol{y}} +\left. \sigma \left(\hat{A}_b\tilde{\boldsymbol{y}}(k+1)- \hat{A}_b\tilde{\boldsymbol{y}}(k)+\mathcal{W}\boldsymbol{\eta}(k) \right)^{\rm T} \hat{A}_b \tilde{\boldsymbol{y}}  \right\} \\
        & =  \argmin\limits_{\tilde{\boldsymbol{y}} \in \tilde{\Omega}} \left\{H(\tilde{\boldsymbol{y}})+\mu(k)^{\rm T}\left(M_1 \tilde{\boldsymbol{y}}+ M_2 \tilde{\boldsymbol{\theta}}(k)\right) + \rho\left(M_1 \tilde{\boldsymbol{y}}(k+1)+ M_2 \tilde{\boldsymbol{\theta}}(k)\right)^{\rm T}M_1\tilde{\boldsymbol{y}} \right. \\
        & \hspace{32mm} + \left. \left[\mathcal{W}\boldsymbol{\lambda}(k)+ \sigma \left(\hat{A}_b\tilde{\boldsymbol{y}}(k+1)- \hat{A}_b\tilde{\boldsymbol{y}}(k)+\mathcal{W}\boldsymbol{\eta}(k) \right)\right]^{\rm T}\hat{A}_b\tilde{\boldsymbol{y}} \right\} \\
        & =  \argmin\limits_{\tilde{\boldsymbol{y}} \in \tilde{\Omega}} \left\{H(\tilde{\boldsymbol{y}})+\boldsymbol{\lambda}(k+1)^{\rm T}\hat{A}_b\tilde{\boldsymbol{y}}+\mu(k)^{\rm T}M_1 \tilde{\boldsymbol{y}} + \rho\left(M_1 \tilde{\boldsymbol{y}}(k+1)+ M_2 \tilde{\boldsymbol{\theta}}(k)\right)^{\rm T}M_1\tilde{\boldsymbol{y}} \right\}  \\
        & =  \argmin\limits_{\tilde{\boldsymbol{y}}} \left\{H(\tilde{\boldsymbol{y}})+ \mathcal{I}_{\tilde{\Omega}}(\tilde{\boldsymbol{y}})+\boldsymbol{\lambda}(k+1)^{\rm T}\hat{A}_b\tilde{\boldsymbol{y}}+\mu(k)^{\rm T}M_1 \tilde{\boldsymbol{y}} + \rho\left(M_1 \tilde{\boldsymbol{y}}(k+1)+ M_2 \tilde{\boldsymbol{\theta}}(k)\right)^{\rm T}M_1\tilde{\boldsymbol{y}} \right\},
    \end{split} 
\end{equation*}
where the symbol $\mathcal{I}_{\tilde{\Omega}}$ represents the indicator function on set $\tilde{\Omega}$, i.e.,
\begin{equation*}
    \begin{aligned}
        \mathcal{I}_{\tilde{\Omega}}(\tilde{\boldsymbol{y}}) = \left\{
        \begin{array}{cl}
          0,   & \quad \text{if\ } \tilde{\boldsymbol{y}} \in \tilde{\Omega}, \\
        +\infty,  & \quad \text{otherwise}.
        \end{array}
        \right.
    \end{aligned}
\end{equation*}
Based on the definition of a convex function, it is straightforward to prove that the indicator function $\mathcal{I}$, when applied to a convex set $\tilde{\Omega}$, results in a convex function $\mathcal{I}_{\tilde{\Omega}}$. Since $\tilde{\boldsymbol{y}}(k+1)$ is the optimal solution to the above optimization problem, it follows that
\begin{equation*}
    0 \in \partial_{\tilde{\boldsymbol{y}}}\left(H(\tilde{\boldsymbol{y}}(k+1))+\mathcal{I}_{\tilde{\Omega}}(\tilde{\boldsymbol{y}}(k+1))\right) + \hat{A}_b^{\rm T}\boldsymbol{\lambda}(k+1)+M_1^{\rm T}\mu(k)+\rho M_1^{\rm T}(M_1 \tilde{\boldsymbol{y}}(k+1)+ M_2 \tilde{\boldsymbol{\theta}}(k)).
\end{equation*}
Since $\mathcal{I}_{\tilde{\Omega}}(\tilde{\boldsymbol{y}})$ is not a continuously differentiable function, the symbol $\partial_{\tilde{\boldsymbol{y}}}\mathcal{I}_{\tilde{\Omega}}(\tilde{\boldsymbol{y}})$ here denotes the sub-gradient of $\mathcal{I}_{\tilde{\Omega}}(\tilde{\boldsymbol{y}})$ at point $\tilde{\boldsymbol{y}}$. If $\tilde{\boldsymbol{y}} \in \tilde{\Omega}$, then the sub-gradient serves as the normal vector to the hyperplane supporting that point inside $\Omega$. If $\tilde{\boldsymbol{y}} \notin \tilde{\Omega}$, then the sub-gradient does not exist. Thus, we can obtain an element in the sub-gradient set of function $ H(\tilde{\boldsymbol{y}})+\mathcal{I}_{\tilde{\Omega}}(\tilde{\boldsymbol{y}})$ at point $\tilde{\boldsymbol{y}}(k+1)$, i.e.,
\begin{equation}\label{eq:sub_grad_y}
    -\hat{A}_b^{\rm T}\boldsymbol{\lambda}(k+1)-M_1^{\rm T}\mu(k)-\rho M_1^{\rm T}(M_1 \tilde{\boldsymbol{y}}(k+1)+ M_2 \tilde{\boldsymbol{\theta}}(k)) \in \partial_{\tilde{\boldsymbol{y}}}\left(H(\tilde{\boldsymbol{y}}(k+1))+\mathcal{I}_{\tilde{\Omega}}(\tilde{\boldsymbol{y}}(k+1))\right)
\end{equation}
Therefore, based on the properties of convex functions, for an arbitrary $\tilde{\boldsymbol{y}} \in \tilde{\Omega}$, we have
\begin{equation}\label{eq:convex_y}
    \begin{aligned}
       H(\tilde{\boldsymbol{y}})+\mathcal{I}_{\tilde{\Omega}}(\tilde{\boldsymbol{y}}) \geq & H(\tilde{\boldsymbol{y}}(k+1))+\mathcal{I}_{\tilde{\Omega}}(\tilde{\boldsymbol{y}}(k+1)) + \boldsymbol{\lambda}(k+1)^{\rm T} \hat{A}_b \left(\tilde{\boldsymbol{y}}(k+1)-\tilde{\boldsymbol{y}}\right) \\
        & + \left[\mu(k)+\rho(M_1 \tilde{\boldsymbol{y}}(k+1)+ M_2 \tilde{\boldsymbol{\theta}}(k))\right]^{\rm T}M_1\left(\tilde{\boldsymbol{y}}(k+1)-\tilde{\boldsymbol{y}}\right) \\
        = & H(\tilde{\boldsymbol{y}}(k+1))+\mathcal{I}_{\tilde{\Omega}}(\tilde{\boldsymbol{y}}(k+1)) + \boldsymbol{\lambda}(k+1)^{\rm T} \hat{A}_b \left(\tilde{\boldsymbol{y}}(k+1)-\tilde{\boldsymbol{y}}\right) \\
        & + \left[\mu(k+1)-\rho M_2( \tilde{\boldsymbol{\theta}}(k+1)-  \tilde{\boldsymbol{\theta}}(k))\right]^{\rm T}M_1\left(\tilde{\boldsymbol{y}}(k+1)-\tilde{\boldsymbol{y}}\right),
    \end{aligned}
\end{equation}
where the equality holds due to the update step $\eqref{eq:step_ADMM_mu}$.  Similarly, $\tilde{\boldsymbol{\theta}}(k+1)$ is the optimal solution for equation \eqref{eq:step_ADMM_theta}, then for any 
$\tilde{\boldsymbol{\theta}} \in \mathbb{R}^{\rm N(n_0+n_1+\dots+n_N)}$ we get
\begin{equation}\label{eq:convex_theta}
\begin{aligned}
    0 & \geq  [\mu(k)+\rho(M_1 \tilde{\boldsymbol{y}}(k+1)+ M_2 \tilde{\boldsymbol{\theta}}(k+1))]^{\rm T}M_2(\tilde{\boldsymbol{\theta}}(k+1)-\tilde{\boldsymbol{\theta}}) \\
    & =  \mu(k+1)^{\rm T}M_2(\tilde{\boldsymbol{\theta}}(k+1)-\tilde{\boldsymbol{\theta}}).
\end{aligned}
\end{equation}
By setting $\tilde{\boldsymbol{y}} = \tilde{\boldsymbol{y}}^*$ in inequality \eqref{eq:convex_y} and $\tilde{\boldsymbol{\theta}} = \tilde{\boldsymbol{\theta}}^*$ in inequality \eqref{eq:convex_theta}, and considering the definition of function $\mathcal{I}_{\tilde{\Omega}}$, we have
\begin{align}
    \begin{split}\label{eq:convex_y2}
       & H(\tilde{\boldsymbol{y}}^*) \geq  H(\tilde{\boldsymbol{y}}(k+1)) + \boldsymbol{\lambda}(k+1)^{\rm T} \hat{A}_b \left(\tilde{\boldsymbol{y}}(k+1)-\tilde{\boldsymbol{y}}^*\right) \\
       & \hspace{20mm} +  \left[\mu(k+1)-\rho M_2( \tilde{\boldsymbol{\theta}}(k+1)-  \tilde{\boldsymbol{\theta}}(k))\right]^{\rm T}M_1\left(\tilde{\boldsymbol{y}}(k+1)-\tilde{\boldsymbol{y}}^*\right),
    \end{split} \\
    & \mu(k+1)^{\rm T}M_2(\tilde{\boldsymbol{\theta}}(k+1)-\tilde{\boldsymbol{\theta}}^*) \leq 0. \label{eq:convex_theta2}
\end{align}

Next, according to the Lagrangian function $\mathcal{L}^1$, we have
\begin{equation}\label{eq:Lag_bound}
    \begin{aligned}
       & \mathcal{L}^1\left(\tilde{\boldsymbol{y}}(k+1),\tilde{\boldsymbol{\theta}}(k+1),\lambda^*,\mu^{*}\right) - \mathcal{L}^1\left(\tilde{\boldsymbol{y}}^*,\tilde{\boldsymbol{\theta}}^*,\lambda^*,\mu^{*}\right) \\
        = & H(\tilde{\boldsymbol{y}}(k+1))-H(\tilde{\boldsymbol{y}}^*)+{\mu^*}^{\rm T}\left(M_1\tilde{\boldsymbol{y}}(k+1)+M_2 \tilde{\boldsymbol{\theta}}(k+1)\right) + {\lambda^*}^{\rm T}\left(\hat{A}\tilde{\boldsymbol{y}}(k+1)-\boldsymbol{d}\right) \\
        \leq &  -\boldsymbol{\lambda}(k+1)^{\rm T} \hat{A}_b \left(\tilde{\boldsymbol{y}}(k+1)-\tilde{\boldsymbol{y}}^*\right)-\left[\mu(k+1)-\rho M_2( \tilde{\boldsymbol{\theta}}(k+1)-  \tilde{\boldsymbol{\theta}}(k))\right]^{\rm T}M_1\left(\tilde{\boldsymbol{y}}(k+1)-\tilde{\boldsymbol{y}}^*\right) \\
        & - \mu(k+1)^{\rm T}M_2(\tilde{\boldsymbol{\theta}}(k+1)-\tilde{\boldsymbol{\theta}}^*) + {\mu^*}^{\rm T}\left(M_1\tilde{\boldsymbol{y}}(k+1)+M_2 \tilde{\boldsymbol{\theta}}(k+1)\right) + {\lambda^*}^{\rm T}\left(\hat{A}\tilde{\boldsymbol{y}}(k+1) -\boldsymbol{d}\right) \\
        = & -\boldsymbol{\lambda}(k+1)^{\rm T} \hat{A}_b \left(\tilde{\boldsymbol{y}}(k+1)-\tilde{\boldsymbol{y}}^*\right) + {\lambda^*}^{\rm T}\left(\hat{A}\tilde{\boldsymbol{y}}(k+1) -\boldsymbol{d}\right) \\
        & + (\mu^*-\mu(k+1))^{\rm T}\left(M_1\tilde{\boldsymbol{y}}(k+1)+M_2 \tilde{\boldsymbol{\theta}}(k+1)\right)+ \rho \left[M_2( \tilde{\boldsymbol{\theta}}(k+1)-  \tilde{\boldsymbol{\theta}}(k))\right]^{\rm T}\! M_1\left(\tilde{\boldsymbol{y}}(k+1)-\tilde{\boldsymbol{y}}^*\right) \\
          = & -\boldsymbol{\lambda}(k+1)^{\rm T} \hat{A}_b \left(\tilde{\boldsymbol{y}}(k+1)-\tilde{\boldsymbol{y}}^*\right) + {\lambda^*}^{\rm T}\left(\hat{A}\tilde{\boldsymbol{y}}(k+1) -\boldsymbol{d}\right)-\frac{1}{\rho} (\mu(k+1)-\mu^*)^{\rm T}(\mu(k+1)-\mu(k)) \\
        & - \rho \left[M_2( \tilde{\boldsymbol{\theta}}(k+1)-  \tilde{\boldsymbol{\theta}}(k))\right]^{\rm T}M_2(\tilde{\boldsymbol{\theta}}(k+1)-\tilde{\boldsymbol{\theta}}^*) + \left[M_2( \tilde{\boldsymbol{\theta}}(k+1)-  \tilde{\boldsymbol{\theta}}(k))\right]^{\rm T}(\mu(k+1)-\mu(k)) \\
        = & -(\boldsymbol{\lambda}(k+1)-\boldsymbol{\lambda}^*)^{\rm T} \hat{A}_b \left(\tilde{\boldsymbol{y}}(k+1)-\tilde{\boldsymbol{y}}^*\right) -\frac{1}{\rho} (\mu(k+1)-\mu^*)^{\rm T}(\mu(k+1)-\mu(k))  \\
        & - \rho \left[M_2( \tilde{\boldsymbol{\theta}}(k+1)-  \tilde{\boldsymbol{\theta}}(k))\right]^{\rm T}M_2(\tilde{\boldsymbol{\theta}}(k+1)-\tilde{\boldsymbol{\theta}}^*) + \left[M_2( \tilde{\boldsymbol{\theta}}(k+1)-  \tilde{\boldsymbol{\theta}}(k))\right]^{\rm T}(\mu(k+1)-\mu(k)).
    \end{aligned}
\end{equation}
The first inequality is derived using the inequalities \eqref{eq:convex_y2}-\eqref{eq:convex_theta2}, the third equality follows from update step \eqref{eq:step_ADMM_mu}, and the final equality is established by writing $\boldsymbol{d}=\hat{A}\tilde{\boldsymbol{y}}^*$ and using the following equation,
\begin{equation*}
    {\boldsymbol{\lambda}^{*}}^{\rm T}\hat{A}_b ={\lambda^*}^{\rm T}\hat{A}.
\end{equation*}
To further bound the inner product term $\left[M_2( \tilde{\boldsymbol{\theta}}(k+1)-  \tilde{\boldsymbol{\theta}}(k))\right]^{\rm T}(\mu(k+1)-\mu(k))$, we set $\tilde{\boldsymbol{\theta}}=\tilde{\boldsymbol{\theta}}(k)$ in inequality \eqref{eq:convex_theta} and get
\begin{equation*}
     0  \geq 
     \mu(k+1)^{\rm T}M_2(\tilde{\boldsymbol{\theta}}(k+1)-\tilde{\boldsymbol{\theta}}(k)).
\end{equation*}
Then, we decrement the index $t+1$ to $t$ in \eqref{eq:convex_theta} and set $\tilde{\boldsymbol{\theta}}=\tilde{\boldsymbol{\theta}}(k+1)$ to get
\begin{equation*}
     0  \geq 
     \mu(k)^{\rm T}M_2(\tilde{\boldsymbol{\theta}}(k)-\tilde{\boldsymbol{\theta}}(k+1)).
\end{equation*}
Summing the above two inequalities leads to
\begin{equation*}
    (\mu(k+1)-\mu(k))^{\rm T}M_2( \tilde{\boldsymbol{\theta}}(k+1)-  \tilde{\boldsymbol{\theta}}(k)) \leq 0.
\end{equation*}
Thus, the inequality \eqref{eq:Lag_bound} becomes
\begin{equation*}
    \begin{aligned}
        & \mathcal{L}^1\left(\tilde{\boldsymbol{y}}(k+1),\tilde{\boldsymbol{\theta}}(k+1),\lambda^*,\mu^{*}\right) - \mathcal{L}^1\left(\tilde{\boldsymbol{y}}^*,\tilde{\boldsymbol{\theta}}^*,\lambda^*,\mu^{*}\right) \\
         \leq & -(\boldsymbol{\lambda}(k+1)-\boldsymbol{\lambda}^*)^{\rm T} \hat{A}_b \left(\tilde{\boldsymbol{y}}(k+1)-\tilde{\boldsymbol{y}}^*\right) -\frac{1}{\rho} (\mu(k+1)-\mu^*)^{\rm T}(\mu(k+1)-\mu(k))  \\
        & - \rho \left[M_2( \tilde{\boldsymbol{\theta}}(k+1)-  \tilde{\boldsymbol{\theta}}(k))\right]^{\rm T}M_2(\tilde{\boldsymbol{\theta}}(k+1)-\tilde{\boldsymbol{\theta}}^*) .
    \end{aligned}
\end{equation*}

Finally, by substituting term $(\boldsymbol{\lambda}(k+1)-\boldsymbol{\lambda}^*)^{\rm T} \hat{A}_b \left(\tilde{\boldsymbol{y}}(k+1)-\tilde{\boldsymbol{y}}^*\right)$ with equation \eqref{eq:lambda_A_y}, we have
\begin{equation}\label{eq:L_L_1}
    \begin{aligned}
        & \mathcal{L}^1\left(\tilde{\boldsymbol{y}}(k+1),\tilde{\boldsymbol{\theta}}(k+1),\lambda^*,\mu^{*}\right) - \mathcal{L}^1\left(\tilde{\boldsymbol{y}}^*,\tilde{\boldsymbol{\theta}}^*,\lambda^*,\mu^{*}\right) \\
         \leq & -\frac{1}{\rho} \left( \mu(k+1)-\mu^*\right)^{\rm T}\left(  \mu(k+1)-\mu(k)\right) - \rho \left( M_2(\tilde{\boldsymbol{\theta}}(k+1)-\tilde{\boldsymbol{\theta}}(k))\right)^{\rm T}\left( M_2(\tilde{\boldsymbol{\theta}}(k+1)-\tilde{\boldsymbol{\theta}}^*)\right)  \\
        & +\frac{1}{2\sigma}\left(\|\bar{\boldsymbol{\lambda}}(k)-\boldsymbol{\lambda}^*\|^2 - \|\bar{\boldsymbol{\lambda}}(k+1)-\bar{\boldsymbol{\lambda}}(k)\|^2 - \|\bar{\boldsymbol{\lambda}}(k+1)-\boldsymbol{\lambda}^*\|^2\right) -\boldsymbol{\epsilon}_2(k+1) ^{\rm T}\!\left( \boldsymbol{\xi}(k+1)-\hat{A}_b \tilde{\boldsymbol{y}}^*\right)\!.
    \end{aligned}
\end{equation}
\endproof

\subsection{Proof of Proposition \ref{pro:Lyapunov}}
\label{proof:pro_Lya}
\proof  
First of all, for convenience in expression, we define $\boldsymbol{\zeta}(k) = \sigma (\boldsymbol{\xi}(k)-\hat{A}_b \tilde{\boldsymbol{y}}^*)$ and
\begin{equation*}
\begin{aligned}
     &\tilde{\boldsymbol{\epsilon}}(k) = 
    \begin{pmatrix}
        \boldsymbol{\epsilon}_1(k) \\
        \sigma \boldsymbol{\epsilon}_2(k)
    \end{pmatrix}, \quad 
    U_0=\begin{pmatrix}
        \mathcal{W}_0 & \mathcal{W}_0  \\
         0_a & \mathcal{W}_0
    \end{pmatrix}, \quad  U_1= \begin{pmatrix}
        I_{a}  \\
         I_{a}
    \end{pmatrix}, \quad 
    U_2= \begin{pmatrix}
        I_{a}  & 0_a
    \end{pmatrix},
\end{aligned}
\end{equation*}
where $0_a$ denotes zero matrices of dimensions $a\times a$ and $a=Nn_0$. Then, the iterative process in equations \eqref{eq:epsilon_1}-\eqref{eq:epsilon_2} can be represented as
\begin{equation}\label{eq:iter_epsi}
    \tilde{\boldsymbol{\epsilon}}(k+1) = U_0  \tilde{\boldsymbol{\epsilon}}(k) + U_1(\boldsymbol{\zeta}(k+1)-\boldsymbol{\zeta}(k)),
\end{equation}
and the equation \eqref{eq:lambda_A_y} can be represented as
\begin{equation}\label{eq:inequality_3}
    \begin{aligned}
     & (\boldsymbol{\lambda}(k+1)-\boldsymbol{\lambda}^*)^{\rm T} \hat{A}_b \left(\tilde{\boldsymbol{y}}(k+1)-\tilde{\boldsymbol{y}}^*\right) \\
     = &      \frac{1}{2\sigma} \left(\|\bar{\boldsymbol{\lambda}}(k+1)-\boldsymbol{\lambda}^*\|^2  + \|\bar{\boldsymbol{\lambda}}(k+1)-\bar{\boldsymbol{\lambda}}(k)\|^2 - \|\bar{\boldsymbol{\lambda}}(k)-\boldsymbol{\lambda}^*\|^2 \right)+ \frac{1}{\sigma}\boldsymbol{\zeta}(k+1) ^{\rm T}U_2\tilde{\boldsymbol{\epsilon}}(k+1).
     \end{aligned}
\end{equation}

Next, recall the positive definite symmetric matrix $\mathcal{Q} \in \mathbb{R}^{\rm 2Nn_0 \times 2Nn_0}$ in Lemma \ref{lem:pos_def_mat} and consider
\begin{equation}\label{eq:V_zeta_epsi}
    \begin{aligned}
       & \|U_1\boldsymbol{\zeta}(k+1)-\tilde{\boldsymbol{\epsilon}}(k+1)\|_{\mathcal{Q}}^2 \\
       =&  \|U_1\boldsymbol{\zeta}(k)-\tilde{\boldsymbol{\epsilon}}(k) + (I_b-U_0)\tilde{\boldsymbol{\epsilon}}(k)\|_{\mathcal{Q}}^2 \\
      = &\|U_1\boldsymbol{\zeta}(k)-\tilde{\boldsymbol{\epsilon}}(k)\|_{\mathcal{Q}}^2+2(U_1\boldsymbol{\zeta}(k)-\tilde{\boldsymbol{\epsilon}}(k))^{\rm T}\mathcal{Q}(I_b-U_0)\tilde{\boldsymbol{\epsilon}}(k)+\tilde{\boldsymbol{\epsilon}}(k)^{\rm T}(I_b-U_0)^{\rm T}\mathcal{Q}(I_b-U_0)\tilde{\boldsymbol{\epsilon}}(k) \\
      = & \|U_1\boldsymbol{\zeta}(k)-\tilde{\boldsymbol{\epsilon}}(k)\|_{\mathcal{Q}}^2 + 2 \boldsymbol{\zeta}(k)^{\rm T}U_1^{\rm T}\mathcal{Q}(I_b-U_0)\tilde{\boldsymbol{\epsilon}}(k)-\|\tilde{\boldsymbol{\epsilon}}(k)\|_{\mathcal{Q}-U_0^{\rm T}\mathcal{Q}U_0}^2 \\
      = & \|U_1\boldsymbol{\zeta}(k)-\tilde{\boldsymbol{\epsilon}}(k)\|_{\mathcal{Q}}^2 + 2 \boldsymbol{\zeta}(k)^{\rm T}U_2\tilde{\boldsymbol{\epsilon}}(k)-\|\tilde{\boldsymbol{\epsilon}}(k)\|_{\mathcal{Q}-U_0^{\rm T}\mathcal{Q}U_0}^2,
    \end{aligned}
\end{equation}
where $b=2a$. The first equality is derived from equation \eqref{eq:iter_epsi} through the addition and subtraction of the term $\tilde{\boldsymbol{\epsilon}}(k)$, and the last equality relies on the following fact, 
\begin{equation*}
        U_1^{\rm T}\mathcal{Q}(I_b-U_0) =  U_1^{\rm T}\mathcal{Q} \begin{pmatrix}
            I_a - \mathcal{W}_0 & -\mathcal{W}_0 \\
            0_a  & I_a - \mathcal{W}_0
        \end{pmatrix} 
        =  U_1^{\rm T}
        \begin{pmatrix}
           2( I_a - \mathcal{W}_0) & -I_a \\
            2 \mathcal{W}_0 - I_a  & I_a
        \end{pmatrix} 
        =  
        \begin{pmatrix}
           I_a  & 0_a
        \end{pmatrix} =U_2.
\end{equation*}
Additionally, 
\begin{equation*}
    \begin{aligned}
        \mathcal{Q}-U_0^{\rm T}\mathcal{Q}U_0 = & \mathcal{Q}-U_0^{\rm T} \begin{pmatrix}
            2\mathcal{W}_0 & (I_a-\mathcal{W}_0)^{-1}\mathcal{W}_0 \\
            (I_a-\mathcal{W}_0)^{-1}\mathcal{W}_0-2\mathcal{W}_0 & \quad \mathcal{W}_0(I_a-\mathcal{W}_0)^{-2}\mathcal{W}_0
        \end{pmatrix} \\
        = & \mathcal{Q} - \begin{pmatrix}
            2\mathcal{W}_0^2 & \mathcal{W}_0(I_a-\mathcal{W}_0)^{-1}\mathcal{W}_0 \\
            \mathcal{W}_0(I_a-\mathcal{W}_0)^{-1}\mathcal{W}_0 & \quad \mathcal{W}_0(I_a-\mathcal{W}_0)^{-2}\mathcal{W}_0
        \end{pmatrix} \\
        = & \begin{pmatrix}
            2(I_a-\mathcal{W}_0)(I_a+\mathcal{W}_0) & \quad -(I_a-\mathcal{W}_0) \\
            -(I_a-\mathcal{W}_0) &  I_a
        \end{pmatrix} \succ 0.
    \end{aligned}
\end{equation*}
The establishment of positive definiteness relies on Schur's complement lemma, wherein the matrix $I_a$ is positive definite, and its corresponding Schur complement is
\begin{equation*}
    2 (I_a-\mathcal{W}_0)(I_a+\mathcal{W}_0) - (I_a-\mathcal{W}_0)^2 = (I_a +3\mathcal{W}_0)(I_a-\mathcal{W}_0) \succ 0.
\end{equation*}
The above equation holds due to the eigenvalues of the matrix $\mathcal{W}_0$ are all lie within the unit circle. This implies that both term $\|\tilde{\boldsymbol{\epsilon}}(k)\|_{\mathcal{Q}-U_0^{\rm T}\mathcal{Q}U_0}^2$ are nonnegative for all $k$.

Finally, by multiplying both sides of inequality \eqref{eq:L_L_1} by $2\sigma$, and then adding the equation \eqref{eq:V_zeta_epsi} to the result, we obtain
\begin{equation}\label{eq:inequality_4}
    \begin{aligned}
          & 2\sigma\mathcal{L}^1\left(\tilde{\boldsymbol{y}}(k+1),\tilde{\boldsymbol{\theta}}(k+1),\lambda^*,\mu^{*}\right) - 2\sigma\mathcal{L}^1\left(\tilde{\boldsymbol{y}}^*,\tilde{\boldsymbol{\theta}}^*,\lambda^*,\mu^{*}\right) \\
         & +\|U_1\boldsymbol{\zeta}(k+1)-\tilde{\boldsymbol{\epsilon}}(k+1)\|_{\mathcal{Q}}^2+\|\bar{\boldsymbol{\lambda}}(k+1)-\boldsymbol{\lambda}^*\|^2  +   2\boldsymbol{\zeta}(k+1) ^{\rm T}U_2\tilde{\boldsymbol{\epsilon}}(k+1)
            \\
           \leq & -\frac{2\sigma}{\rho} \left( \mu(k+1)-\mu^*\right)^{\rm T}\left(  \mu(k+1)-\mu(k)\right) - 2\sigma\rho \left( M_2(\tilde{\boldsymbol{\theta}}(k+1)-\tilde{\boldsymbol{\theta}}(k))\right)^{\rm T}\left( M_2(\tilde{\boldsymbol{\theta}}(k+1)-\tilde{\boldsymbol{\theta}}^*)\right) \\
             & +  \|U_1\boldsymbol{\zeta}(k)-\tilde{\boldsymbol{\epsilon}}(k)\|_{\mathcal{Q}}^2 +\|\bar{\boldsymbol{\lambda}}(k)-\boldsymbol{\lambda}^*\|^2 + 2 \boldsymbol{\zeta}(k)^{\rm T}U_2\tilde{\boldsymbol{\epsilon}}(k)-\|\tilde{\boldsymbol{\epsilon}}(k)\|_{\mathcal{Q}-U_0^{\rm T}\mathcal{Q}U_0}^2-\|\bar{\boldsymbol{\lambda}}(k+1)-\bar{\boldsymbol{\lambda}}(k)\|^2 .
        \end{aligned}
\end{equation}
To have better control over the cross terms in the above equation, we define 
\begin{equation*}
    V_0(k) = \sigma\rho \|M_2(\tilde{\boldsymbol{\theta}}(k)-\tilde{\boldsymbol{\theta}}^*)\|^2 + \frac{\sigma}{\rho} \| \mu(k)-\mu^*\|^2.
\end{equation*}
Then it can be seen that
\begin{equation}\label{eq:V_0_k}
    \begin{aligned}
        V_0(k) = & \sigma\rho \|M_2(\tilde{\boldsymbol{\theta}}(k+1)-\tilde{\boldsymbol{\theta}}(k)+\tilde{\boldsymbol{\theta}}(k+1)-\tilde{\boldsymbol{\theta}}^*)\|^2 + \frac{\sigma}{\rho} \| \mu(k)-\mu(k+1)+\mu(k+1)-\mu^*\|^2 \\
        = & V_0(k+1) -2\sigma\rho \!\left( M_2(\tilde{\boldsymbol{\theta}}(k+1)-\tilde{\boldsymbol{\theta}}(k))\right)^{\!\rm T} \! \left( M_2(\tilde{\boldsymbol{\theta}}(k+1)-\tilde{\boldsymbol{\theta}}^*)\right)\! + \sigma\rho \|M_2(\tilde{\boldsymbol{\theta}}(k+1)-\tilde{\boldsymbol{\theta}}(k))\|^2 \\
        & -\frac{2\sigma}{\rho} \left( \mu(k+1)-\mu^*\right)^{\rm T}\left(  \mu(k+1)-\mu(k)\right) + \frac{\sigma}{\rho}\|\mu(k+1)-\mu(k)\|^2.
    \end{aligned}
\end{equation}
Now, we introduce the Lyapunov function
\begin{equation*}
    V(k)=V_0(k)+\|\bar{\boldsymbol{\lambda}}(k)-\boldsymbol{\lambda}^*\|^2 + \|U_1\boldsymbol{\zeta}(k)-\tilde{\boldsymbol{\epsilon}}(k)\|_{\mathcal{Q}}^2.
\end{equation*}
Immediately, using Lyapunov function $V(k)$ and equation \eqref{eq:V_0_k},  we can rewrite inequality \eqref{eq:inequality_4} as
\begin{equation}\label{eq:V_k_V_k_1}
    \begin{aligned}
     V(k) \geq & V(k+1)  + \sigma\rho \|M_2(\tilde{\boldsymbol{\theta}}(k+1)-\tilde{\boldsymbol{\theta}}(k))\|^2 +\frac{\sigma}{\rho}\|\mu(k+1)-\mu(k)\|^2\\
     &+2\sigma \left(\mathcal{L}^1(\tilde{\boldsymbol{y}}(k+1),\tilde{\boldsymbol{\theta}}(k+1),\lambda^*,\mu^{*}) - \mathcal{L}^1(\tilde{\boldsymbol{y}}^*,\tilde{\boldsymbol{\theta}}^*,\lambda^*,\mu^{*})\right)   \\
     &+   2\boldsymbol{\zeta}(k+1) ^{\rm T}U_2\tilde{\boldsymbol{\epsilon}}(k+1)- 2 \boldsymbol{\zeta}(k)^{\rm T}U_2\tilde{\boldsymbol{\epsilon}}(k) +\|\tilde{\boldsymbol{\epsilon}}(k)\|_{\mathcal{Q}-U_0^{\rm T}\mathcal{Q}U_0}^2+\|\bar{\boldsymbol{\lambda}}(k+1)-\bar{\boldsymbol{\lambda}}(k)\|^2 . 
    \end{aligned}
\end{equation}

At the same time, our design for the optimization problem ensures it inherently possesses the strong duality property, then by the Saddle Point Theorem, we have 
\begin{equation*}
\mathcal{L}^1\left(\tilde{\boldsymbol{y}}^*,\tilde{\boldsymbol{\theta}}^*,\lambda,\mu\right)\leq \mathcal{L}^1\left(\tilde{\boldsymbol{y}}^*,\tilde{\boldsymbol{\theta}}^*,\lambda^{*},\mu^{*}\right) \leq \mathcal{L}^1\left(\tilde{\boldsymbol{y}},\tilde{\boldsymbol{\theta}},\lambda^*,\mu^*\right).
\end{equation*}\endproof

\subsection{Preparation for Theorem \ref{thm:convergence}}

\begin{lemma}\label{lem:thm_cor}
    Under the Assumptions \ref{ass:const_set}-\ref{ass:com_network}, the sequence generated by Consensus-Tracking-ADMM satisfy:
    \begin{enumerate}[label=(\roman*),leftmargin=24pt]
        \item $\lim_{k\rightarrow 0}\boldsymbol{\epsilon}_1(k)=\boldsymbol{0}$,
        \item $\lim_{k\rightarrow 0}\boldsymbol{\epsilon}_2(k)=\boldsymbol{0}$, 
        \item $\{\rho\sigma \|M_2(\tilde{\boldsymbol{\theta}}(k)-\tilde{\boldsymbol{\theta}}^*)\|^2+\frac{\sigma}{\rho}\|\mu(k)-\mu^*\|^2+\|\bar{\boldsymbol{\lambda}}(k)-\boldsymbol{\lambda}^*\|^2+2\sigma^2\|\boldsymbol{\xi}(k)-\hat{A}_b\tilde{\boldsymbol{y}}^*\|^2\}$ is convergent.
    \end{enumerate}
\end{lemma}

\proof
 Based on inequality \eqref{eq:V_0_Sum} in the proof of Theorem \ref{thm:convergence}, it can be deduced that the boundedness of $V(0)$ implies that as $K_0 \rightarrow \infty$, 
 \begin{equation*}
     \|\tilde{\boldsymbol{\epsilon}}(k)\|_{\mathcal{Q}-U_0^{\rm T}\mathcal{Q}U_0} \rightarrow 0.
 \end{equation*}
 Furthermore, given that $\mathcal{Q}-U_0^{\rm T}\mathcal{Q}U_0 \succ 0$ and the principle of norm equivalence, it follows that 
 \begin{equation*}
     \|\tilde{\boldsymbol{\epsilon}}(k)\| \rightarrow 0.
 \end{equation*}
Subsequently, according to the definition of $\tilde{\boldsymbol{\epsilon}}(k)$, we can conclude that 
 \begin{equation*}
     \|\boldsymbol{\epsilon}_1(k)\| \rightarrow 0, \quad \|\boldsymbol{\epsilon}_2(k)\| \rightarrow 0.
 \end{equation*}
Hence, the statements $(\romannumeral 1)$ and $(\romannumeral 4)$ are proven.

Then, from inequality \eqref{eq:V_k_V_k_1}, we have that the sequence $\{V(k)+2 \boldsymbol{\zeta}(k)^{\rm T}U_2\tilde{\boldsymbol{\epsilon}}(k)\}_{k \geq 0}$ is a non-increasing sequence. Additionally, given that the cross-term $\boldsymbol{\zeta}(k)^{\rm T}U_2\tilde{\boldsymbol{\epsilon}}(k)$ is bounded and the matrix $\mathcal{Q} \succ 0$, we can deduce that the sequence possesses a lower bound, and consequently, is convergent. Since $\|\tilde{\boldsymbol{\epsilon}}(k)\|$ converges to 0 as $k$ approaches infinity and $\{\boldsymbol{\zeta}(k)\}_{k\geq 0}$ is bounded by definition, we have that also the sequence $\{V_0(k)+\|\bar{\boldsymbol{\lambda}}(k)-\boldsymbol{\lambda}^*\|^2+\|U_1\boldsymbol{\zeta}(k)\|_{\mathcal{Q}}^2\}_{k\geq 0}$ is convergent. Furthermore, applying the  principle of norm equivalence, we deduce that the sequence $\{V_0(k)+\|\bar{\boldsymbol{\lambda}}(k)-\boldsymbol{\lambda}^*\|^2+\|U_1\boldsymbol{\zeta}(k)\|\}_{k\geq 0}$ is convergent. Finally, through straightforward calculation, we have
\begin{equation*}
    \|U_1\boldsymbol{\zeta}(k)\|^2 = \boldsymbol{\zeta}(k)^{\rm T} U_1^{\rm T}U_1\boldsymbol{\zeta}(k) = 2\boldsymbol{\zeta}(k)^{\rm T} \boldsymbol{\zeta}(k) = 2\sigma^2  \|\boldsymbol{\xi}(k)-\hat{A}_b\tilde{\boldsymbol{y}}^*\|^2.
\end{equation*}
This concludes the proof of the statement $(\romannumeral 3)$.
\endproof

\subsection{Proof of Theorem \ref{thm:convergence}}\label{proof:thm_co}

\proof
By taking the telescoping sum from $k=0$ to $K_0-1$ with $K_0 \in \mathbb{N}_{+}$ in above inequality, we have
\begin{equation}\label{eq:V_0_V_K}
    \begin{aligned}
        V(0) \geq & V(K_0) + \sum\nolimits_{i=0}^{K_0-1} \left(\sigma\rho \|M_2(\tilde{\boldsymbol{\theta}}(k+1)-\tilde{\boldsymbol{\theta}}(k))\|^2 +\frac{\sigma}{\rho}\|\mu(k+1)-\mu(k)\|^2\right) \\
        & + 2\sigma \sum\nolimits_{i=0}^{K_0-1}\left(\mathcal{L}^1(\tilde{\boldsymbol{y}}(k+1),\tilde{\boldsymbol{\theta}}(k+1),\lambda^*,\mu^{*}) - \mathcal{L}^1(\tilde{\boldsymbol{y}}^*,\tilde{\boldsymbol{\theta}}^*,\lambda^*,\mu^{*})\right) \\
        & + \sum\nolimits_{i=0}^{K_0-1} \left(\|\tilde{\boldsymbol{\epsilon}}(k)\|_{\mathcal{Q}-U_0^{\rm T}\mathcal{Q}U_0}^2+\|\bar{\boldsymbol{\lambda}}(k+1)-\bar{\boldsymbol{\lambda}}(k)\|^2\right) + 2\boldsymbol{\zeta}(K_0) ^{\rm T}U_2\tilde{\boldsymbol{\epsilon}}(K_0)- 2 \boldsymbol{\zeta}(0)^{\rm T}U_2\tilde{\boldsymbol{\epsilon}}(0).
    \end{aligned}
\end{equation}
According to Lemma \ref{lem:bounded}, the sequences $\{\boldsymbol{\xi}(k)\}_{k \geq 0}$, $\{\boldsymbol{\epsilon}_1(k)\}_{k \geq 0}$, and $\{\boldsymbol{\epsilon}_2(k)\}_{k \geq 0}$ are all bounded. Consequently, based on the definitions of $\tilde{\boldsymbol{\epsilon}}(k)$ and $\boldsymbol{\zeta}(k)$, it follows that $\{\tilde{\boldsymbol{\epsilon}}(k)\}_{k \geq 0}$ and $\{\boldsymbol{\zeta}(k)\}_{k \geq 0}$ are also bounded sequences. Therefore, we conclude that $2\boldsymbol{\zeta}(K_0) ^{\rm T}U_2\tilde{\boldsymbol{\epsilon}}(K_0)- 2 \boldsymbol{\zeta}(0)^{\rm T}U_2\tilde{\boldsymbol{\epsilon}}(0) \geq -\mathcal{C}$. Additionally, given that the Lyapunov function $V(k)$ is composed of four nonnegative norms, it follows that $V(k) \geq 0$ for all $k \geq 0$. Thus, the inequality \eqref{eq:V_0_V_K} can be further rewritten as 
\begin{equation*}
    \begin{aligned}
        V(0) \geq &  \sum\nolimits_{i=0}^{K_0-1} \left(\sigma\rho \|M_2(\tilde{\boldsymbol{\theta}}(k+1)-\tilde{\boldsymbol{\theta}}(k))\|^2 +\frac{\sigma}{\rho}\|\mu(k+1)-\mu(k)\|^2\right) \\
        & + 2\sigma \sum\nolimits_{i=0}^{K_0-1}\left(\mathcal{L}^1(\tilde{\boldsymbol{y}}(k+1),\tilde{\boldsymbol{\theta}}(k+1),\lambda^*,\mu^{*}) - \mathcal{L}^1(\tilde{\boldsymbol{y}}^*,\tilde{\boldsymbol{\theta}}^*,\lambda^*,\mu^{*})\right) \\
        & + \sum\nolimits_{i=0}^{K_0-1} \left(\|\tilde{\boldsymbol{\epsilon}}(k)\|_{\mathcal{Q}-U_0^{\rm T}\mathcal{Q}U_0}^2+\|\bar{\boldsymbol{\lambda}}(k+1)-\bar{\boldsymbol{\lambda}}(k)\|^2\right) - \mathcal{C}.
    \end{aligned}
\end{equation*}
Taking the limit as $K_0 \rightarrow \infty$ finally yields
\begin{equation}\label{eq:V_0_Sum}
    \begin{aligned}
        V(0) \geq &  \sum\nolimits_{i=0}^{\infty} \left(\sigma\rho \|M_2(\tilde{\boldsymbol{\theta}}(k+1)-\tilde{\boldsymbol{\theta}}(k))\|^2 +\frac{\sigma}{\rho}\|\mu(k+1)-\mu(k)\|^2\right) \\
        & + 2\sigma \sum\nolimits_{i=0}^{\infty}\left(\mathcal{L}^1(\tilde{\boldsymbol{y}}(k+1),\tilde{\boldsymbol{\theta}}(k+1),\lambda^*,\mu^{*}) - \mathcal{L}^1(\tilde{\boldsymbol{y}}^*,\tilde{\boldsymbol{\theta}}^*,\lambda^*,\mu^{*})\right) \\
        & + \sum\nolimits_{i=0}^{\infty} \left(\|\tilde{\boldsymbol{\epsilon}}(k)\|_{\mathcal{Q}-U_0^{\rm T}\mathcal{Q}U_0}^2+\|\bar{\boldsymbol{\lambda}}(k+1)-\bar{\boldsymbol{\lambda}}(k)\|^2\right) - \mathcal{C}.
    \end{aligned}
\end{equation}
Since $V(0)$ is finite, it follows that as $k \rightarrow \infty$,
\begin{equation*}
\begin{aligned}
   & \mathcal{L}^1(\tilde{\boldsymbol{y}}(k+1),\tilde{\boldsymbol{\theta}}(k+1),\lambda^*,\mu^{*}) - \mathcal{L}^1(\tilde{\boldsymbol{y}}^*,\tilde{\boldsymbol{\theta}}^*,\lambda^*,\mu^{*}) \rightarrow 0,\\
    &
   \mu(k+1)-\mu(k) \rightarrow 0, \quad  \bar{\boldsymbol{\lambda}}(k+1)-\bar{\boldsymbol{\lambda}}(k) \rightarrow 0 \ \left(\bar{\lambda}(k+1)-\bar{\lambda}(k) \rightarrow 0 \right).
\end{aligned}
\end{equation*}
Thus, combining equations \eqref{eq:lemma2_in_fal} \eqref{eq:lemma1_in_fal}, we have 
\begin{equation*}
    \bar{\eta}(k) = \frac{1}{\sigma}\left(\bar{\lambda}(k+1)-\bar{\lambda}(k)\right) \rightarrow 0,
\end{equation*}
thereby proving the statement $(\romannumeral 1)$. Immediately after that, we have
\begin{align}
        & \hat{A}\tilde{\boldsymbol{y}}(k+1)-\boldsymbol{d}=\sum\nolimits_i\hat{A}_i\boldsymbol{y}_i(k+1) -\boldsymbol{d} = N\bar{\eta}(k) \rightarrow 0,  \label{eq:lim_con_1}\\
        & M_1\tilde{\boldsymbol{y}}(k)+M_2\tilde{\boldsymbol{\theta}}(k) = \mu(k)-\mu(k-1) \rightarrow 0, \label{eq:lim_con_2}\\
        & H(\tilde{\boldsymbol{y}}(k))-H(\tilde{\boldsymbol{y}}^*)=\mathcal{L}^1(\tilde{\boldsymbol{y}}(k),\tilde{\boldsymbol{\theta}}(k),\lambda^*,\mu^{*}) - \mathcal{L}^1(\tilde{\boldsymbol{y}}^*,\tilde{\boldsymbol{\theta}}^*,\lambda^*,\mu^{*}) \notag \\
        & \hspace{30mm} - {\mu^*}^{\rm T}\left(M_1\tilde{\boldsymbol{y}}(k)+M_2\tilde{\boldsymbol{\theta}}(k)\right)-{\lambda^*}^{\rm T}\left(\hat{A}\tilde{\boldsymbol{y}}(k)-\boldsymbol{d}\right) \rightarrow 0, \notag
\end{align}
This demonstrates that the function values converge, thereby proving the statement $(\romannumeral 2)$. 

Next, in the context of the convergence analysis for the iterative sequence $\{\tilde{\boldsymbol{y}}(k)\}_{k \geq 0}$, $\{\tilde{\boldsymbol{\theta}}(k)\}_{k \geq 0}$, $\{\mu(k)\}_{k \geq 0}$, and $\{\boldsymbol{\lambda}(k)\}_{k \geq 0}$ generated by iterative process \eqref{eq:step_ADMM} in optimization problem \ref{eq:Distri_ADMM}, given that the Slater condition is readily verifiable, we focus on the KKT conditions satisfied by the optimal solution $\tilde{\boldsymbol{y}}^*$, $\tilde{\boldsymbol{\theta}}^*$, $\mu^*$, and $\lambda^*$, which is,
\begin{equation*}
    \begin{aligned}
        & -M_1^{\rm T}\mu^* - \hat{A}^{\rm T}\lambda^* \in \partial_{\tilde{\boldsymbol{y}}}(H(\tilde{\boldsymbol{y}}^*)+\mathcal{I}_\Omega(\tilde{\boldsymbol{y}}^*)), \\
        & -M_2^{\rm T}\mu^* \in \partial_{\tilde{\boldsymbol{\theta}}}G(\tilde{\boldsymbol{\theta}}^*), \\
        & \hat{A} \tilde{\boldsymbol{y}}^* =\boldsymbol{d}, \\
        & M_1\tilde{\boldsymbol{y}}^* + M_2\tilde{\boldsymbol{\theta}}^*=0.
    \end{aligned}
\end{equation*}
Then, from \eqref{eq:sub_grad_y}, we can identify a subgradient of the function $H(\tilde{\boldsymbol{y}})+\mathcal{I}_{\tilde{\Omega}}(\tilde{\boldsymbol{y}})$ at the optimal solution $\tilde{\boldsymbol{y}}(k+1)$ after the $(k+1)$-th iteration. By introducing the term $\hat{A}_b^{\rm T}\bar{\boldsymbol{\lambda}}(k+1)$ into this subgradient(both by adding and subtracting it), and applying the iterative step \eqref{eq:step_ADMM_mu}, we obtain the following transformation of the sub-gradient.
\begin{equation*}
    \begin{aligned}
       & -\hat{A}_b^{\rm T}\boldsymbol{\lambda}(k+1)-M_1^{\rm T}\mu(k)-\rho M_1^{\rm T}(M_1 \tilde{\boldsymbol{y}}(k+1)+ M_2 \tilde{\boldsymbol{\theta}}(k))-\hat{A}_b^{\rm T}\bar{\boldsymbol{\lambda}}(k+1)+\hat{A}_b^{\rm T}\bar{\boldsymbol{\lambda}}(k+1) \\
       =& -M_1^{\rm T}[\mu(k)+\rho(M_1 \tilde{\boldsymbol{y}}(k+1)+ M_2 \tilde{\boldsymbol{\theta}}(k))]-\hat{A}_b^{\rm T}\bar{\boldsymbol{\lambda}}(k+1) - \hat{A}_b^{\rm T}(\boldsymbol{\lambda}(k+1)-\bar{\boldsymbol{\lambda}}(k+1))\\
        = &  -M_1^{\rm T}\mu(k+1)+\rho M_1^{\rm T}M_2(\tilde{\boldsymbol{\theta}}(k+1)-\tilde{\boldsymbol{\theta}}(k)) -\hat{A}_b^{\rm T}\bar{\boldsymbol{\lambda}}(k+1) - \hat{A}_b^{\rm T}(\boldsymbol{\lambda}(k+1)-\bar{\boldsymbol{\lambda}}(k+1)) \\
      = &  -M_1^{\rm T}\mu(k+1) -\hat{A}^{\rm T}\bar{\lambda}(k+1) +\rho M_1^{\rm T}M_2(\tilde{\boldsymbol{\theta}}(k+1)-\tilde{\boldsymbol{\theta}}(k))- \hat{A}_b^{\rm T}\boldsymbol{\epsilon}_2(k+1),
    \end{aligned}
\end{equation*}
where the last equality holds by
\begin{equation*}
    \hat{A}_b^{\rm T}\bar{\boldsymbol{\lambda}}(k+1) =  \hat{A}^{\rm T} \bar{\lambda}(k+1).
\end{equation*}
Similarly, the subgradient of the function $G$ at optimal solution $\tilde{\boldsymbol{\theta}}(k+1)$ after $(k+1)$-th iteration can also be simplified using iterative step \eqref{eq:step_ADMM_mu}. Ultimately, after $k+1$ iterations, we obtain,

\begin{align}
\begin{split}
   & -M_1^{\rm T}\mu(k+1) -\hat{A}^{\rm T}\bar{\lambda}(k+1) +\rho M_1^{\rm T}M_2(\tilde{\boldsymbol{\theta}}(k+1)-\tilde{\boldsymbol{\theta}}(k))- \hat{A}_b^{\rm T}\boldsymbol{\epsilon}_2(k+1)  \\
    & \hspace{50mm} \in \partial_{\tilde{\boldsymbol{y}}}(H(\tilde{\boldsymbol{y}}(k+1))+\mathcal{I}_{\tilde{\Omega}}(\tilde{\boldsymbol{y}}(k+1))),
\end{split} \\
& -M_2^{\rm T}\mu(k+1) = 0
\end{align}

Furthermore, according to Lemma \ref{lem:thm_cor}, we know the sequence $\{V(k)+2 \boldsymbol{\zeta}(k)^{\rm T}U_2\tilde{\boldsymbol{\epsilon}}(k)\}_{k \geq 0}$ is bounded. Consequently, the terms $\|M_2(\tilde{\boldsymbol{\theta}}(k)-\tilde{\boldsymbol{\theta}}^*)\|^2$, $\|\bar{\boldsymbol{\lambda}}(k)-\boldsymbol{\lambda}^*\|^2$, and $\|\mu(k)-\mu^*\|^2$ within the sequence are bounded. Additionally, given that the matrix $M_2$ is of full column rank, it can be inferred that the sequences $\{\tilde{\boldsymbol{y}}(k)\}_{k \geq 0}$, $\{\tilde{\boldsymbol{\theta}}(k)\}_{k \geq 0}$, $\{\bar{\lambda}(k)\}_{k \geq 0}$, and $\{\mu(k)\}_{k \geq 0}$ are also bounded (the boundedness of $\{\tilde{\boldsymbol{y}}(k)\}_{k \geq 0}$ is given in Lemma \ref{lem:bounded}). Since every bounded sequence necessarily contains a convergent subsequence, there exists a set of indices $k_s$ such that the subsequences of the above sequence converge to the limit points $\tilde{\boldsymbol{y}}^{\infty}$, $\tilde{\boldsymbol{\theta}}^{\infty}$, $\bar{\lambda}^{\infty}$, and $\mu^{\infty}$, respectively. Moreover, based on inequality \eqref{eq:V_0_Sum}, it can be deduced that as $k$ approaches infinity, $\|M_2(\tilde{\boldsymbol{\theta}}(k+1)-\tilde{\boldsymbol{\theta}}(k))\|$ tends towards 0. Combining this with statement $(\romannumeral 2)$ of Lemma \ref{lem:thm_cor}, we have
\begin{equation*}
\begin{aligned}
   & \lim\nolimits_{s \rightarrow\infty} -M_1^{\rm T}\mu(k_s) -\hat{A}^{\rm T}\bar{\lambda}(k_s) +\rho M_1^{\rm T}M_2(\tilde{\boldsymbol{\theta}}(k_s)-\tilde{\boldsymbol{\theta}}(k_s-1))- \hat{A}_b^{\rm T}\boldsymbol{\epsilon}_2(k_s) 
    = -M_1^{\rm T}\mu^{\infty} - \hat{A}^{\rm T}\bar{\lambda}^{\infty}, \\
    & \lim\nolimits_{s \rightarrow\infty} -M_2^{\rm T}\mu(k+1) = -M_2^{\rm T}\mu^{\infty}.
\end{aligned}
\end{equation*}
From the fact that the graph of the subgradient mapping is a closed set, it can be inferred that
\begin{equation*}
    -M_1^{\rm T}\mu^{\infty} - \hat{A}^{\rm T}\bar{\lambda}^{\infty} \in \partial_{\tilde{\boldsymbol{y}}}(H(\tilde{\boldsymbol{y}}^{\infty})+\mathcal{I}_{\tilde{\Omega}}(\tilde{\boldsymbol{y}}^{\infty})), \quad -M_2^{\rm T}\mu^{\infty} =0.
\end{equation*}
Additionally, based on \eqref{eq:lim_con_1} and \eqref{eq:lim_con_2}, we have
\begin{equation*}
   \lim\nolimits_{s \rightarrow\infty} \hat{A} \tilde{\boldsymbol{y}}(k_s) -\boldsymbol{d}
    = \hat{A} \tilde{\boldsymbol{y}}^{\infty} -\boldsymbol{d}=0, \quad
    \lim\nolimits_{s \rightarrow\infty} M_1\tilde{\boldsymbol{y}}(k_s)+M_2\tilde{\boldsymbol{\theta}}(k_s)=M_1\tilde{\boldsymbol{y}}^{\infty}+M_2\tilde{\boldsymbol{\theta}}^{\infty} = 0.
\end{equation*}
Therefore, $\tilde{\boldsymbol{y}}^{\infty}$, $\tilde{\boldsymbol{\theta}}^{\infty}$, $\bar{\lambda}^{\infty}$, and $\mu^{\infty}$ satisfy the KKT conditions for optimization problem \ref{eq:Distri_ADMM}, and hence $(\tilde{\boldsymbol{y}}^{\infty},\tilde{\boldsymbol{\theta}}^{\infty},\bar{\lambda}^{\infty},\mu^{\infty})=(\tilde{\boldsymbol{y}}^*,\tilde{\boldsymbol{\theta}}^*,\bar{\lambda}^*,\mu^*)$. Thus, by taking the limit of the sequence specified in statement $(\romannumeral 3)$ of Lemma \ref{lem:thm_cor} in the sense of subsequences, we obtain that
\begin{equation*}
   \lim\nolimits_{s \rightarrow \infty} \rho\sigma \|M_2(\tilde{\boldsymbol{\theta}}(k_s)-\tilde{\boldsymbol{\theta}}^*)\|^2+\frac{\sigma}{\rho}\|\mu(k_s)-\mu^*\|^2+\|\bar{\boldsymbol{\lambda}}(k_s)-\boldsymbol{\lambda}^*\|^2+2\sigma^2\|\boldsymbol{\xi}(k_s)-\hat{A}_b\tilde{\boldsymbol{y}}^*\|^2 = 0.
\end{equation*}
Meanwhile, Lemma \ref{lem:thm_cor} informs us that the above sequence is convergent, hence all limit points of its subsequences are identical, that is,
\begin{equation*}
   \lim\nolimits_{k \rightarrow \infty} \rho\sigma \|M_2(\tilde{\boldsymbol{\theta}}(k)-\tilde{\boldsymbol{\theta}}^*)\|^2+\frac{\sigma}{\rho}\|\mu(k)-\mu^*\|^2+\|\bar{\boldsymbol{\lambda}}(k)-\boldsymbol{\lambda}^*\|^2+2\sigma^2\|\boldsymbol{\xi}(k)-\hat{A}_b\tilde{\boldsymbol{y}}^*\|^2 = 0.
\end{equation*}
Furthermore, we have
\begin{align*}
        &\lim\nolimits_{k\rightarrow \infty } \|\mu(k)-\mu^*\|=0,  \\
       & \lim\nolimits_{k\rightarrow \infty }\|M_2(\tilde{\boldsymbol{\theta}}(k)-\tilde{\boldsymbol{\theta}}^*)\| =0,\\ 
       & \lim\nolimits_{k\rightarrow \infty } \|\boldsymbol{\xi}(k)-\hat{A}_b\tilde{\boldsymbol{y}}^*\|=\lim\nolimits_{k\rightarrow \infty } \|\hat{A}_b\tilde{\boldsymbol{y}}(k)-\hat{A}_b\tilde{\boldsymbol{y}}^*-\bar{\boldsymbol{\eta}}(k)\|=0, \\
       & \lim\nolimits_{k\rightarrow \infty } \|\bar{\boldsymbol{\lambda}}(k)-\boldsymbol{\lambda}^* \|=0.
\end{align*}
From the above equation, it is straightforward to deduce that $\mu(k)\rightarrow \mu^*$. Given that the matrix $M_2$ is column full-rank (i.e., $M_2^{\rm T}M_2 \succ 0$), it follows that $\tilde{\boldsymbol{\theta}}(k) \rightarrow \tilde{\boldsymbol{\theta}}^*$. With $\bar{\eta} \rightarrow 0$  and matrix $\hat{A}_b$ being column full-rank, we can conclude that $\tilde{\boldsymbol{y}}(k) \rightarrow \tilde{\boldsymbol{y}}^*$. Finally, for the dual variable $\lambda$, we have
\begin{equation*}
    0 \leq \lim_{k \rightarrow \infty} \|\boldsymbol{\lambda}(k)-\boldsymbol{\lambda}^*\| \leq \lim_{k \rightarrow \infty} \|\boldsymbol{\lambda}(k)-\bar{\boldsymbol{\lambda}}(k)\|+\lim_{k \rightarrow \infty} \|\bar{\boldsymbol{\lambda}}(k)-\boldsymbol{\lambda}^*\|=\lim_{k \rightarrow \infty} \|\boldsymbol{\epsilon}_2(k)\|+\lim_{k \rightarrow \infty} \|\bar{\boldsymbol{\lambda}}(k)-\boldsymbol{\lambda}^*\|=0.
\end{equation*}
By the definitions of $\boldsymbol{\lambda}(k)$ and $\boldsymbol{\lambda}^*$, this conclusion is equivalent to
\begin{equation*}
    \lim_{k \rightarrow \infty} \|\lambda_i(k)-\lambda^*\| =0, \quad \forall i \in \{1,\dots,N\}.
\end{equation*}
Thus far, statements $(\romannumeral 3)$ and $(\romannumeral 4)$ in Theorem \ref{thm:convergence} have been proven.
\endproof

\subsection{Proof of Proposition \ref{pro:SDP_IR}}\label{proof:pro_sd}

\proof

We show that participants are individually rational under the optimal price signal $\pi_i^*$. From Assumption \ref{ass:str_conv}, $f_i$ is convex in $\boldsymbol{x}_i$. Thus, we have
\begin{equation}\label{eq:convex_f}
    f_i(\boldsymbol{0},\boldsymbol{x}_{-i}^*) \geq f_i(\boldsymbol{x}_i^*,\boldsymbol{x}_{-i}^*) -\left(\nabla_{ \boldsymbol{x}_i}f_i(\boldsymbol{x}_i^*,\boldsymbol{x}_{-i}^*)\right)^{\rm T}\boldsymbol{x}_i^*, \quad \forall i.
\end{equation}
Similarly, Assumption \ref{ass:str_conv} implies that each component constraint function $g_i^j (j=1,\dots,q_i)$ is convex in $\boldsymbol{x}_i$. Therefore, by this convexity, we derive that
\begin{equation}\label{eq:convex_g}
    g_i^j(\boldsymbol{0}) \geq g_i^j(\boldsymbol{x}_i^*) -\left(\nabla g_i^j(\boldsymbol{x}_i^*)\right)^{\rm T}\boldsymbol{x}_i^*.
\end{equation}
Combined with the setting $f_i(\boldsymbol{0},\boldsymbol{x}_{-i}^*)=0$ in Assumption \ref{ass:str_conv}, inequalities \eqref{eq:convex_f} becomes
\begin{equation}\label{eq:convex_f_i}
        f_i(\boldsymbol{x}_i^*,\boldsymbol{x}_{-i}^*) -\left(\nabla_{ \boldsymbol{x}_i}f_i(\boldsymbol{x}_i^*,\boldsymbol{x}_{-i}^*)\right)^{\rm T}\boldsymbol{x}_i^* \leq 0, \quad \forall i.   
\end{equation}

Next, we calculate the benefit that each participant receives for participating in the task
\begin{equation*}
    \begin{aligned}
        u_i(\boldsymbol{x}_i^*,\boldsymbol{x}_{-i}^*) & = f_i(\boldsymbol{x}_i^*,\boldsymbol{x}_{-i}^*)- {\boldsymbol{\pi}_i^*}^{\rm T} \boldsymbol{x}_i^* \\
        & = f_i(\boldsymbol{x}_i^*,\boldsymbol{x}_{-i}^*)- \left(A_i^{\rm T}\lambda^* - \nabla_{ \boldsymbol{x}_i}\sum\nolimits_{s\neq i} f_s(\boldsymbol{x}^*)\right)^{\rm T} \boldsymbol{x}_i^*  \\
       & = f_i(\boldsymbol{x}_i^*,\boldsymbol{x}_{-i}^*)- \left(\nabla_{ \boldsymbol{x}_i}f_i(\boldsymbol{x}_i^*,\boldsymbol{x}_{-i}^*)+ (J_{g_i}(\boldsymbol{x}_i^*))^{\rm T} \alpha_{i}^*\right)^{\rm T} \boldsymbol{x}_i^* \\
       & = f_i(\boldsymbol{x}_i^*,\boldsymbol{x}_{-i}^*)- \left(\nabla_{ \boldsymbol{x}_i}f_i(\boldsymbol{x}_i^*,\boldsymbol{x}_{-i}^*)\right)^{\rm T}\boldsymbol{x}_i^*- {\alpha_{i}^*}^{\rm T} J_{g_i}(\boldsymbol{x}_i^*)\boldsymbol{x}_i^* \\
       & \leq   -\sum\nolimits_j{\alpha_{i,j}^*} \left(\nabla g_i^j(\boldsymbol{x}_i^*)\right)^{\rm T}\boldsymbol{x}_i^*\\
       & \leq  \sum\nolimits_j{\alpha_{i,j}^*}g_i^j(\boldsymbol{0})- \sum\nolimits_j\alpha_{i,j}^*g_i^j(\boldsymbol{x}_i^*)\\
       & \leq 0,
    \end{aligned}
\end{equation*}
where $J_{g_i}(\boldsymbol{x}_i)$ represents the Jacobi matrix of the function $g_i$. The first equality follows from the definition of $\pi_i^*$, the second equality employs substitution using equation \eqref{eq:sta_con_c}, the first inequality applies inequality \eqref{eq:convex_f_i}, the second inequality applies inequality \eqref{eq:convex_g}, and the final inequality leverages equation $\alpha_i^* \odot g_i(\boldsymbol{x}_i^*)$  in conjunction with the condition $\alpha_i^*\geq 0$ and $0 \in \Omega_i$.
\endproof

\subsection{Proof of Proposition \ref{pro:NE_eq_OP}}\label{proof:pro_NE}

\proof
We start with the existence of the equilibrium solution. The Lagrangian function associated with the optimization problem \ref{eq:Distributed} is 
\begin{equation*}
    \begin{aligned}
        \mathcal{L} & =\sum\nolimits_{i}f_i(\boldsymbol{x}) - \lambda^{\rm T}\left(\sum\nolimits_{i} A_i\boldsymbol{x}_{i} -\boldsymbol{d}\right) +\sum\nolimits_i\alpha_i^{\rm T}g_i(\boldsymbol{x}_i).
    \end{aligned}
\end{equation*}
By assumption \ref{ass:saddle}, the optimal solution $\boldsymbol{x}^*={\rm col}(\boldsymbol{x}_1^*,\dots,\boldsymbol{x}_N^*)$ and dual optimal solutions $\lambda^*$, $\alpha_1^*, \dots, \alpha_N^*$ satisfy the following Karush-Kuhn-Tucker(KKT) condition.
\begin{itemize}
    \item Stability conditions: 
    \begin{equation}\label{eq:sta_con_c}
    \begin{aligned}
         \nabla_{\boldsymbol{x}_i}\sum\nolimits_s f_s(\boldsymbol{x}^*) - A_i^{\rm T} \lambda^*  + (J_{g_i}(\boldsymbol{x}_i^*))^{\rm T} \alpha_{i}^*=0 \quad \forall i.
    \end{aligned}      
    \end{equation}
    \item Initial feasibility conditions: 
    \begin{equation*}
        \begin{aligned}
            & \sum\nolimits_i A_i \boldsymbol{x}_i^*  = \boldsymbol{d}, \\
            & g_i(\boldsymbol{x}_i^*) \leq \boldsymbol{0}, \quad \forall i.
        \end{aligned}
    \end{equation*}
    \item Dual feasibility conditions:
    \begin{equation*}
        \begin{aligned}
            & \alpha_i^* \geq 0, \quad \forall i.
        \end{aligned}
    \end{equation*}
    \item Complementary slackness conditions:
    \begin{equation}
        \begin{aligned}
            & {\lambda^*}^{\rm T}\left(\sum\nolimits_{i} A_i\boldsymbol{x}_{i}^* -\boldsymbol{d}\right)=0, \notag\\
            & \alpha_i^* \odot g_i(\boldsymbol{x}_i^*)=0, \quad \forall i. \label{eq:con_sla_c}
        \end{aligned}
    \end{equation}
\end{itemize}

Similarly, we consider the Lagrangian function and the KKT condition for the optimization problem \ref{eq:ind_supply} under the optimal price signal $\pi_i^*$ in the shadow pricing mechanism.
\begin{equation*}
    \begin{aligned}
        \mathcal{L}_i(\boldsymbol{x}_i,\alpha_i) & = f_i(\boldsymbol{x}_i,\boldsymbol{x}_{-i}) -{\boldsymbol{\pi}_i^*}^{\rm T} \boldsymbol{x}_i +\alpha_i^{\rm T}g_i(\boldsymbol{x}_i) \\
        &=f_i(\boldsymbol{x}_i,\boldsymbol{x}_{-i}) -\left(A_i^{\rm T}\lambda^* - \nabla_{ \boldsymbol{x}_i}\sum\nolimits_{s\neq i} f_s(\boldsymbol{x}^*)\right)^{\rm T} \boldsymbol{x}_i +\alpha_i^{\rm T}g_i(\boldsymbol{x}_i).
    \end{aligned}
\end{equation*}
If optimization problems $\mathcal{P}_1^s, \dots, \mathcal{P}_N^s$ have an equilibrium solution $\boldsymbol{x}^{E},\boldsymbol{\alpha}^{E}$, then this equilibrium solution should satisfy the following KKT conditions for any $i \in \{1,\dots, N\}$.
\begin{itemize}
    \item Stability conditions: 
    \begin{equation}\label{eq:sta_con}
    \begin{aligned}
        \nabla_{ \boldsymbol{x}_i}f_i(\boldsymbol{x}_i^{E},\boldsymbol{x}_{-i}^{E}) - A_i^{\rm T}\lambda^* + \nabla_{ \boldsymbol{x}_i}\sum\nolimits_{s\neq i} f_s(\boldsymbol{x}^*)  + (J_{g_i}(\boldsymbol{x}_i^{E}))^{\rm T} \alpha_{i}^{E}=0.
    \end{aligned}      
    \end{equation}
    \item Initial feasibility conditions: 
    \begin{equation}\label{eq:ini_con}
        \begin{aligned}
            & g_i(\boldsymbol{x}_i^{E}) \leq \boldsymbol{0}.
        \end{aligned}
    \end{equation}
    \item Dual feasibility conditions:
    \begin{equation}\label{eq:dua_con}
        \begin{aligned}
            \alpha_i^{E} \geq 0.
        \end{aligned}
    \end{equation}
    \item Complementary slackness conditions:
    \begin{equation}\label{eq:com_con}
        \begin{aligned}
            & \alpha_i^{E} \odot g_i(\boldsymbol{x}_i^{E}) = 0.
        \end{aligned}
    \end{equation}
\end{itemize}
It is easy to verify that the optimal solution $\boldsymbol{x}_1^*,\dots,\boldsymbol{x}_N^*$ and the optimal dual variables $\alpha_1^*,\dots,\alpha_N^*$ satisfy equation \eqref{eq:sta_con}-\eqref{eq:com_con}. So far, we have demonstrated that the equilibrium solution of the game exists.

Next, Assumption \ref{ass:NE_uni} implies that function $F$ is strictly monotonic. Furthermore, based on Theorem 2.2.3 in \cite{facchinei2003finite}, the following variational inequality problem has at most one solution.
\begin{equation*}
    \text{Find\ } \boldsymbol{x}^E \in \mathcal{X}, \text{\ such that\ } \langle F(\boldsymbol{x}), \boldsymbol{x}-\boldsymbol{x}^E \rangle \geq 0, \ \forall \boldsymbol{x} \in   \mathcal{X}.
\end{equation*}
Based on the preceding analysis, solution $\boldsymbol{x}^*$ satisfies the aforementioned variational inequality. Therefore, a solution exists, and according to the theorem, this solution is unique. 

Therefore, under the Assumption \ref{ass:const_set}, \ref{ass:saddle}, \ref{ass:str_conv} and \ref{ass:LICQ}, the optimal solution $\boldsymbol{x}^*$ and the optimal dual variables $\lambda^*,\alpha_1^*,\dots,\alpha_N^*$ of problem \ref{eq:Distributed} are unique. Furthermore, under the unique optimal price signal $\boldsymbol{\pi}_i^*$, the optimal solution $\boldsymbol{x}^*$ coincides with the equilibrium of problems $\mathcal{P}_1^s, \dots, \mathcal{P}_N^s$. Since the equilibrium solution $\boldsymbol{x}^E$ is unique, it must coincide with the optimal solution.
\endproof

\section{Notes on Examples}\label{sup_mat:Notes}
\subsection{Notes on Example \ref{example:allo}}
In this section, we provide supplementary analysis regarding the convexity of the function $f_i$ defined in Example \ref{example:allo}. For simplicity, we combine variables into vector form according to specific rules, as follows:
\begin{equation*}
    \begin{aligned}
    & \boldsymbol{x}_{ijk} = {\rm col}(x_{ijk1}, \dots, x_{ijkr}, \dots, x_{ijkR_{ij}}),\quad  \boldsymbol{x}_{ij} = {\rm col}(\boldsymbol{x}_{ij1},\dots , \boldsymbol{x}_{ijk},\dots,\boldsymbol{x}_{ijK}),\\
   & \boldsymbol{x}_i = {\rm col}(\boldsymbol{x}_{i1}, \dots, \boldsymbol{x}_{ij}, \dots,\boldsymbol{x}_{iM}), \quad  \boldsymbol{c}_i = {\rm col}(\boldsymbol{c}_{i1}, \dots, \boldsymbol{c}_{ie}, \dots,\boldsymbol{c}_{iE}),
    \end{aligned}
\end{equation*}
where $R_{ij}$ denotes the total number of paths in set $\mathcal{P}_{ij}$, and $E$ represents the number of edges in graph $\mathcal{G}_{\mathrm{tra}}$. Meanwhile, we design the indicator function $\mathbf{1}_{ijkr}(e)$ as
\begin{equation*}
    \mathbf{1}_{ijkr}(e)=\left\{
    \begin{array}{cl}
      1,   & \text{if}\  r\in P_{ij}, \  e\in r,  \\
      0,   & \text{if}\  r\in P_{ij}, \ e\notin r,
    \end{array}
    \right.
\end{equation*}
and define
\begin{equation*}
\begin{aligned}
   & \mathbf{1}_{ijk}(e) = {\rm col}(\mathbf{1}_{ijk1}(e), \dots, \mathbf{1}_{ijkr}(e),\dots, \mathbf{1}_{ijkR_{ij}}(e)), \quad  \mathbf{1}_{ij}(e) = {\rm col}(\mathbf{1}_{ij1}(e),\dots , \mathbf{1}_{ijk}(e),\dots,\mathbf{1}_{ijK}(e)),\\
   & \mathbf{1}_{i}(e) = {\rm col}(\mathbf{1}_{i1}(e), \dots, \mathbf{1}_{ij}(e), \dots,\mathbf{1}_{iM}(e)).
\end{aligned}
\end{equation*}
Then, we can rewrite the total traffic on edge $e$ in the above notation as
\begin{equation*}
    q_e(\boldsymbol{x})= \sum\nolimits_{i=1}^N \mathbf{1}_i(e)^{\rm T} \boldsymbol{x}_i.
\end{equation*}
The vector consisting of the total traffic on each edge can be expressed as
\begin{equation*}
    {\rm col}\left(q_{e_1}(\boldsymbol{x}),q_{e_2}(\boldsymbol{x}), \dots, q_{e_E}(\boldsymbol{x})\right) = \sum\nolimits_{i=1}^N Q_i \boldsymbol{x}_i,
\end{equation*}
where the matrix $Q_i$ is defined by 
$\begin{pmatrix}
        \mathbf{1}_{i}(e_1) & \mathbf{1}_{i}(e_2) & \dots & \mathbf{1}_{i}(e_E) 
    \end{pmatrix}^{\rm T}$.
The matrix $Q_i$ indicates whether participant $i$ has an available transportation path on each edge $e$ of the network $\mathcal{G}$. We can further define the vector-valued cost function $C(\boldsymbol{x})$ as:
\begin{equation*}
    C(\boldsymbol{x}):= {\rm col}\left(c_{e_1}(q_{e_1}(\boldsymbol{x})), c_{e_2}(q_{e_2}(\boldsymbol{x})), \dots,  c_{e_E}(q_{e_E}(\boldsymbol{x}))\right).
\end{equation*}
Thus, we can rewrite the objective function $f(\boldsymbol{x})$ in the optimization problem \eqref{eq:example-tran} as:
\begin{equation*}
    f(\boldsymbol{x})=C(\boldsymbol{x})^{\rm T} \left(\sum\nolimits_{i=1}^N  Q_i \boldsymbol{x}_i \right) + \sum\nolimits_{i=1}^N \boldsymbol{c}_i^{\rm T}Q_i\boldsymbol{x}_i.
\end{equation*}
When the function $C(\boldsymbol{x})$ has the linear form
\begin{equation}\label{eq:function_c}
    C(\boldsymbol{x})={\rm col}(c_{e_1}\cdot q_{e_1}(\boldsymbol{x}), c_{e_2}\cdot q_{e_2}(\boldsymbol{x}), \dots, c_{e_E}\cdot q_{e_E}(\boldsymbol{x})),
\end{equation}
then the objective function $f(\boldsymbol{x})$ has the form
\begin{equation*}
    f(\boldsymbol{x})=\left(\boldsymbol{c}_e \odot \sum\nolimits_{i=1}^N  Q_i \boldsymbol{x}_i \right)^{\rm T} \left(\sum\nolimits_{i=1}^N  Q_i \boldsymbol{x}_i \right) + \sum\nolimits_{i=1}^N \boldsymbol{c}_i^{\rm T}Q_i\boldsymbol{x}_i,
\end{equation*}
where $\boldsymbol{c}_e={\rm col}(c_{e_1},c_{e_2},\dots,c_{e_E})$ and the operator $\odot$ denotes element-wise multiplication.

Then, for each participant, the decomposition according to individual actual cost is
\begin{equation}\label{eq:actual_f}
    f_i^{\rm actual}(\boldsymbol{x})=\left(\boldsymbol{c}_e \odot \sum\nolimits_{s=1}^N  Q_s \boldsymbol{x}_s \right)^{\rm T}  Q_i \boldsymbol{x}_i  +  \boldsymbol{c}_i^{\rm T}Q_i\boldsymbol{x}_i.
\end{equation}
However, with such a decomposition, $f_i^{\rm actual}(\boldsymbol{x})$ is not guaranteed to be a convex function for $\boldsymbol{x}$. Thus, we need to perform a different decomposition on the objective function $f(\boldsymbol{x})$ such that the resulting function $f_i(\boldsymbol{x})$ is convex with respect to $\boldsymbol{x}$. For this purpose, we define the following two matrices,
\begin{equation*}
 Q_0 = {\rm diag}\left(\sum\nolimits_{i=1}^N Q_i \mathbf{1} \right), \quad \bar{Q}_i = {\rm diag} \left( Q_0^{-1}Q_i \mathbf{1} \right), \forall i.
\end{equation*}
The value of the $s$-th diagonal element in the diagonal matrix $\bar{Q}_i$ is $\frac{\sum_{j,k}\sum_{r \in \mathcal{P}_{i,j}} \boldsymbol{1}_{\{e_s \in r\} }}{\sum_{i',j,k}\sum_{r \in \mathcal{P}_{i,j}} \boldsymbol{1}_{\{e_s \in r\} }}$, which is denoted by $\kappa_{i,e_s}$ as defined in Example \ref{example:allo}.
Then, the function $f_i(\boldsymbol{x})$ obtained after readjusting the decomposition scheme is
\begin{equation}\label{eq:convexity}
    f_i(\boldsymbol{x})=\left(\boldsymbol{c}_e \odot \sum\nolimits_{s=1}^N  Q_s \boldsymbol{x}_s \right)^{\rm T} \bar{Q}_i\left(\sum\nolimits_{s=1}^NQ_s\boldsymbol{x}_s\right)+\boldsymbol{c}_i^{\rm T}Q_i\boldsymbol{x}_i,
\end{equation}
which is consistent with the expression in Example \ref{example:allo}.

\begin{remark}
    The diagonal elements of the diagonal matrix $Q_0$ are all positive numbers, because if one of the elements were zero, it would indicate that no routes from any nodes traverse the corresponding edge $e$, making this edge removable from the transportation network $\mathcal{G}$ without impacting the outcome. Consequently, the diagonal matrix $Q_0^{-1}$ must exist. 
\end{remark}

Next, we verify convexity for $f_i(\boldsymbol{x})$ in the equation \ref{eq:convexity}. From its functional expression, it is clear that $f_i$ is twice continuously differentiable. The Hessian matrix of $\boldsymbol{x}$ is given by
\begin{equation*}
\begin{aligned}
    \nabla^2f_i(\boldsymbol{x}) & = 
        \begin{pmatrix}
            \frac{\partial^2}{\partial \boldsymbol{x}_1^2}f_i &  \frac{\partial^2}{\partial \boldsymbol{x}_1 \partial \boldsymbol{x}_2}f_i & \cdots & \frac{\partial^2}{\partial \boldsymbol{x}_1 \partial \boldsymbol{x}_N}f_i \\
            \frac{\partial^2}{\partial \boldsymbol{x}_2 \partial \boldsymbol{x}_1}f_i &  \frac{\partial^2}{\partial \boldsymbol{x}_2^2}f_i & \cdots & \frac{\partial^2}{\partial \boldsymbol{x}_2 \partial \boldsymbol{x}_N}f_i \\
            \vdots & \vdots & \ddots & \vdots  \\
            \frac{\partial^2}{\partial \boldsymbol{x}_N \partial \boldsymbol{x}_1}f_i &  \frac{\partial^2}{\partial \boldsymbol{x}_N \partial \boldsymbol{x}_2}f_i & \cdots & \frac{\partial^2}{ \partial \boldsymbol{x}_N^2}f_i  
        \end{pmatrix} \\
        &=  \ 2
        \begin{pmatrix}
            Q_1^{\rm T} \\
            Q_2^{\rm T} \\
            \vdots \\
            Q_N^{\rm T}
        \end{pmatrix}  \bar{Q}_i 
        \begin{pmatrix}
           \boldsymbol{c}_e \odot Q_1 & \boldsymbol{c}_e \odot Q_2 & \cdots & \boldsymbol{c}_e \odot Q_N
        \end{pmatrix} \\
        & \succeq  \  0.
\end{aligned}
\end{equation*}
Thus, according to the second-order condition for convexity, the function $f_i$ is convex with respect to $\boldsymbol{x}$. Simultaneously, all constraints in Example \ref{example:allo} are linear. Therefore, $g_i$ is necessarily convex with respect to $\boldsymbol{x}$, thereby satisfying statement $(\romannumeral 1)$ of Assumption \ref{ass:const_set}. Regarding statement $(\romannumeral 2)$ of Assumption \ref{ass:const_set}, it holds as long as the boundary values in Example \ref{example:allo} are well-defined.

\subsection{Notes on Example \ref{ex:simple_case}}

We provide some explanations for the calculations in Example \ref{ex:simple_case}. First, it can be readily verified that when the objective function in the optimization problem \eqref{eq:example-tran} adopts the form given by equation \eqref{eq:actual_f}, Assumptions \ref{ass:str_conv}-\ref{ass:NE_uni} clearly hold.

\begin{remark}
    We emphasize that equation \eqref{eq:actual_f} is employed here instead of equation \eqref{eq:convexity} because the construction of \eqref{eq:convexity} does not represent the participant's actual cost. It is introduced solely to ensure the convergence of the distributed algorithm and the computability of the optimal solution. Throughout the mechanism analysis, we exclusively utilize the participant's actual cost function $f_i^{\rm actual}$.
\end{remark}

According to problem \eqref{eq:example-tran}, and with the parameters set to $N = 3$, $M = 1$, $K = 1$, $R = 1$, the overall objective function is given by
\begin{equation}\label{eq:note_tran1}
\begin{aligned}
         \min_{x_1,x_2,x_3} \quad & \sum_{i=1}^3 \left[c_0 \cdot x_{i}\cdot \left(x_i+\sum_{s=1}^3 x_s\right) +
          (c_{ii}+c_{i4})  \cdot x_{i} \right] \\
        {\rm s.t.} \quad \  & x_1+x_2+x_3 =d, \\
        & x_{i}\geq 0,
\end{aligned}
\end{equation}
where $c_{ij}$ denotes the element at the $j$-th position of the vector $\boldsymbol{c}_i$ in Example \ref{ex:simple_case}. Substituting the parameter values specified in Example \ref{ex:simple_case}, the solution to optimization problem \ref{eq:note_tran1} can be obtained.

Subsequently, to calculate the payoff for each individual, the value of the price signal $\pi_i$ should first be computed according to equation \eqref{eq:LMP_price}, followed by calculating the payoff using the equation below
\begin{equation*}
    |u_i^*| = \pi_i^* - c_0 \cdot x_{i}^*\cdot \left(x_i^*+\sum_{s=1}^3 x_s^*\right) - (c_{ii}+c_{i4})  \cdot x_{i}^*.
\end{equation*}

\section{Algorithm improvement} \label{sup_mat:Alg_impro}

\subsection{Accelerated-Consensus-Tracking ADMM}

In many real-world optimization scenarios, the dimensionality of variable $\boldsymbol{y}_i(k)$ in each optimization problem $\mathcal{S}_{i,k}$ is extremely high. For instance, in Example \ref{example:allo}, the corresponding variable $\boldsymbol{y}_i$ exhibits an $NMKR$-dimensional configuration. The consequence of excessive dimensionality is that the algorithm cannot produce satisfactory results within a short time, which is unacceptable for many scenarios in reality. Therefore, we need to make some improvements to Algorithm \ref{alg:TCADMM}. 

For each individual convex constraint set $\Omega_i$ in optimization problem \ref{eq:Distributed}, we perform a linear approximation by replacing it with a nonempty, compact convex polygon $\Xi_i = \{ \boldsymbol{x}_i \ |\ B_i\boldsymbol{x}_i \leq \boldsymbol{m}_i\}$. In optimization problem \ref{eq:Distri_Conv_copy}, the set $\Xi_i$ can be written as $\tilde{\Xi}_i = \{\boldsymbol{y}_i \ |\ \tilde{B}_i\boldsymbol{y}_i \leq \boldsymbol{m}_i\}$, where $\tilde{B}_i = \boldsymbol{e}_{N,i}^{\rm T} \otimes B_i$.  Consider the highly sparse matrix $\tilde{B}_i$, it only imposes constraints on the variables $\boldsymbol{x}_i$, while the remaining parts consist entirely of zero elements to match the dimension of the decision variables $\boldsymbol{y}_i$. Since only a subset of the variables are actually constrained in the optimization problem $\mathcal{S}_{i,k}$, we can optimize the problem by focusing solely on these constrained variables. By transforming the remaining unconstrained variables into functions of the constrained ones, the dimensionality of the problem $\mathcal{S}_{i,k}$ can be significantly reduced, thereby achieving the goal of acceleration.

For brevity in expressions, we denote the objective function in the optimization problem $\mathcal{S}_{i,k}$ as $F_{i,k}$, i.e.,
\begin{equation}\label{eq:F_ik}
    F_{i,k}(\boldsymbol{y}_i) := h_i(\boldsymbol{y}_i)+\dfrac{\rho}{2} \text{deg}(i) \| \boldsymbol{y}_i-v_i(k)\|^2 +l_i(k)^{\rm T}\tilde{A}_i\boldsymbol{y}_i+ \dfrac{\sigma}{2} \| \tilde{A}_i\boldsymbol{y}_i-\tilde{A}_i\boldsymbol{y}_i(k)+\gamma_i(k)\|^2.
\end{equation}
Meanwhile, recalling the definition of $\boldsymbol{y}_i$, we use $\tilde{\boldsymbol{x}}_{-i}$ to denote the part of $\boldsymbol{y}_i$ excluding $\boldsymbol{x}_i$. Next, we denote by $\mathcal{F}_{i,k}$ the new function formed from $F_{i,k}$ by treating the part of the variable $\boldsymbol{y}_i$ involving $\boldsymbol{x}_{i}$ as known, i.e.,
    $\mathcal{F}_{i,k}(\tilde{\boldsymbol{x}}_{-i}) := F_{i,k}(\boldsymbol{y}_i | \boldsymbol{x}_{i}).$
Therefore, the following problem is an unconstrained optimization problem,
\begin{equation}\label{eq:unconstriant}
\tag{$\mathcal{U}_{i,k}$}
    \tilde{\boldsymbol{x}}_{-i}^{k} =  \argmin_{\tilde{\boldsymbol{x}}_{-i}}  \mathcal{F}_{i,k}(\tilde{x}_{-i}).
\end{equation}
\begin{assumption}\label{ass:explicit}
    The optimization problem \ref{eq:unconstriant} has an explicit solution.
\end{assumption}
Under Assumption \ref{ass:explicit}, the optimal solution to optimization problem \ref{eq:unconstriant} can be expressed as $\tilde{\boldsymbol{x}}_{-i}^k(\boldsymbol{x}_i)$, allowing us to treat $\tilde{\boldsymbol{x}}_{-i}^k$ as a function of $\boldsymbol{x}_i$. In optimization problem $\mathcal{S}_{i,k}$, we substitute the component $\tilde{\boldsymbol{x}}_{-i}$ in the variable $\boldsymbol{y}_i$ with $\tilde{\boldsymbol{x}}_{-i}^k(\boldsymbol{x}_i)$, thereby transforming the decision variables into $\boldsymbol{x}_i$ with low dimension. This yields the $k$+$1$-th updated value of $\boldsymbol{x}_i$, given by
\begin{equation*}
    \boldsymbol{x}_i(k+1) = \mathop{\mathrm{argmin}}\limits_{\boldsymbol{x}_i\in \Omega_i} F_{i,k}(\boldsymbol{x}_i,\tilde{\boldsymbol{x}}_{-i}^k(\boldsymbol{x}_i)).
\end{equation*}
Thus, we have revised Algorithm \ref{alg:TCADMM} to create Improved-Consensus-Tracking-ADMM (Algorithm \ref{alg:ATCADMM}). In comparison to Algorithm \ref{alg:TCADMM}, Algorithm \ref{alg:ATCADMM} demonstrates improved efficiency in dealing with sub-optimization problems. We use the symbols $\tilde{\boldsymbol{x}}_{-i,l}$ and $\tilde{\boldsymbol{x}}_{-i,r}$ to denote the first and second halves of the variable $\tilde{\boldsymbol{x}}_{-i}$ respectively, i.e., $\tilde{\boldsymbol{x}}_{-i,l}={\rm col}(\boldsymbol{x}_{i,1},\boldsymbol{x}_{i,2},\dots,\boldsymbol{x}_{i,i-1})$ and $\tilde{\boldsymbol{x}}_{-i,r} = {\rm col}(\boldsymbol{x}_{i,i+1},\boldsymbol{x}_{i,i+2},\dots,\boldsymbol{x}_{i,N})$.
\begin{algorithm}[thb]
    \caption{Improved-Consensus-Tracking-ADMM} \label{alg:ATCADMM}
    \begin{algorithmic}[1]
        \Statex \textbf{Initialization:} 
        \State Choose parameter $\sigma$, $\rho$, and set $k=0$. 
        \For {$i=1,\dots,N$(in parallel)}
        \State Choose the initial value $\boldsymbol{y}_{i}(0)  \in \tilde{\Omega}_i$    \label{line2_3}
        \State Set $\lambda_i(0)=0$, \ $\eta_i(0)=\tilde{A}_i\boldsymbol{y}_i(0)-\boldsymbol{d}/N$ \label{line2_4}
         \State Set $v_i(0)=\text{deg}(i)^{-1} \sum\nolimits_{\{i,j\} \in \mathcal{E}} \frac{1}{2} (\boldsymbol{y}_{i}(0)+\boldsymbol{y}_{j}(0))$ \label{line2_5}
        \EndFor
        \Statex \textbf{Iterations:}
        \While {Convergence is not reached} 
        \For {$i=1,\dots,N$(in parallel)}
        \State $\gamma_i(k)=\sum_{s\in \mathcal{N}_i} w_{is}\eta_s(k)$   \label{line2_9}
        \State $l_i(k)=\sum_{s\in \mathcal{N}_i} w_{is}\lambda_s(k)$  \label{line2_10}
        \State  $\boldsymbol{x}_i(k+1) = \mathop{\mathrm{argmin}}\limits_{\boldsymbol{x}_i\in \Omega_i} F_{i,k}(\boldsymbol{x}_i,\tilde{\boldsymbol{x}}_{-i}^k(\boldsymbol{x}_i))$
        \State $\tilde{\boldsymbol{x}}_{-i}(k+1)  = \tilde{\boldsymbol{x}}_{-i}^k(\boldsymbol{x}_i(k+1))$
        \State $\boldsymbol{y}_i(k+1)  = {\rm col}\left( \tilde{\boldsymbol{x}}_{-i,l}(k+1), \boldsymbol{x}_i(k+1), \tilde{\boldsymbol{x}}_{-i,r}(k+1)\right)$
        \State $\eta_i(k+1)=\gamma_i(k)+\tilde{A}_i\boldsymbol{y}_i(k+1)-\tilde{A}_i\boldsymbol{y}_i(k)$ 
        \State $\lambda_i(k+1)=l_i(k)+\sigma \eta_i(k+1)$
        \State $\delta_i(k+1)= \boldsymbol{y}_i(k+1)-\frac{1}{2}\boldsymbol{y}_i(k)$
        \State $v_i(k+1) = v_i(k) + \text{deg}(i)^{-1}\sum_{s\in \mathcal{N}_i}\delta_s(k+1)- \frac{1}{2}\boldsymbol{y}_i(k)$
        \EndFor
        \State  Set $k= k+1$
        \EndWhile
    \end{algorithmic}
\end{algorithm}

\subsection{Example of improved Algorithm}

We present an illustrative example where the objective function $h_i$ exhibits a quadratic structure and the constraint functions $g_i$ possess a linear structure.

\begin{example}\label{example:alg2}
    We demonstrate the accelerated approach for subproblem optimization using an example where the objective function is defined as $h_i(\boldsymbol{y}_i)=\frac{1}{2}\boldsymbol{y}_i^{\rm T}\Sigma_i \boldsymbol{y}_i+\psi _i^{\rm T}\boldsymbol{y}_i$ and the individual constraint set takes the form $\tilde{\Xi}_i=\{\boldsymbol{y}_i \ |\ B_i\boldsymbol{y}_i-\boldsymbol{m}_i \leq 0\}$.
\end{example}

Under the above setting, the equation \eqref{eq:F_ik} becomes
\begin{equation*}
\begin{aligned}
    F_{i,k}(\boldsymbol{y}_i) &= \frac{1}{2}\boldsymbol{y}_i^{\rm T}\Sigma_i \boldsymbol{y}_i+\dfrac{\rho}{2} \text{deg}(i) \| \boldsymbol{y}_i-v_i(k)\|^2 +l_i(k)^{\rm T}\tilde{A}_i\boldsymbol{y}_i+ \dfrac{\sigma}{2} \| \tilde{A}_i\boldsymbol{y}_i-\tilde{A}_i\boldsymbol{y}_i(k)+\gamma_i(k)\|^2  \\
    & = \frac{1}{2}\boldsymbol{y}_i^{\rm T}\Phi_i\boldsymbol{y}_i+\Psi_i(k)^{\rm T}\boldsymbol{y}_i,
\end{aligned}    
\end{equation*}
where
\begin{equation*}
    \begin{aligned}
         & \Phi_i= \Sigma_i+\rho {\rm deg}(i)I_{n_1+\dots+n_N}+\sigma \tilde{A}_i^{\rm T}\tilde{A}_i,\\
        & \Psi_i(k) = \psi_i-\rho {\rm deg}(i)v_i(k)+\tilde{A}_i^{\rm T}l_i(k)-\sigma \tilde{A}_i^{\rm T}\tilde{A}_i\boldsymbol{y}_i(k)+\sigma \tilde{A}_i^{\rm T} \gamma_i(k).
    \end{aligned}
\end{equation*}
Thus, the sub-optimization problem $\mathcal{S}_{i,k}$ is a quadratic programming with a compact feasible region, which can be formulated as follows,
\begin{equation}\label{eq:sub_opt_stand}
\tag{$\mathcal{S}_{i,k}'$}
    \boldsymbol{y}_i(k+1)=\argmin\limits_{\boldsymbol{y}_i \in \tilde{\Xi}_i} \left\{\dfrac{1}{2}\boldsymbol{y}_i^{\rm T}\Phi_i\boldsymbol{y}_i+\Psi_i(k)^{\rm T}\boldsymbol{y}_i\right\}.
\end{equation}

Then, define $n^0=n_1+\dots+n_N$. For the $n^0$-dimensional vector $v_i(k)$, we define
\begin{equation*}
    \begin{aligned}
       &v_{i,i}(k) = \left(\boldsymbol{e}_{N,i}^{\rm T} \otimes I_{n_i}\right) v_{i}(k), \quad  v_{i,-i}(k) = \left(\sum\nolimits_{s=1,s\neq i}^{N}\boldsymbol{e}_{N,s}^{\rm T} \otimes I_{n_s}\right)v_{i}(k),\\
       &\psi_{i,i} = \left(\boldsymbol{e}_{N,i}^{\rm T} \otimes I_{n_i}\right) \psi_i, \quad  \psi_{i,-i} = \left(\sum\nolimits_{s=1,s\neq i}^{N}\boldsymbol{e}_{N,s}^{\rm T} \otimes I_{n_s}\right)\psi_{i}.
    \end{aligned}
\end{equation*}
And for the $n^0$-dimensional symmetric matrix $\Sigma_i$, we perform some splitting and define the following matrices,
\begin{equation*}
    \begin{aligned}
       & \Sigma_{i,i}= \left(\boldsymbol{e}_{N,i}^{\rm T} \otimes I_{n_i}\right) \Sigma \left(\boldsymbol{e}_{N,i} \otimes I_{n_i}\right) \in \mathbb{R}^{\rm n_i\times n_i}, \\
       & \Sigma_{i,L} = \left(\boldsymbol{e}_{N,i}^{\rm T} \otimes I_{n_i}\right) \Sigma \left(\sum\nolimits_{s=1,s\neq i}^{N}\boldsymbol{e}_{N,s} \otimes I_{n_s}\right) \in \mathbb{R}^{\rm n_i\times (n^0- n_i)}, \\
       & \Sigma_{i,R} = \left(\sum\nolimits_{s=1,s\neq i}^{N}\boldsymbol{e}_{N,s}^{\rm T} \otimes I_{n_s}\right) \Sigma \left(\sum\nolimits_{s=1,s\neq i}^{N}\boldsymbol{e}_{N,s} \otimes I_{n_s}\right) \in \mathbb{R}^{\rm (n^0-n_i)\times (n^0-n_i)}.
    \end{aligned}
\end{equation*}
Additionally, define two more matrices,
\begin{equation*}
    \begin{aligned}
      &  \bar{\Sigma}_{i,i} = \Sigma_{i,i} + \rho {\rm deg}(i) I_{n_i} \in \mathbb{R}^{\rm n_i\times n_i}, \\
      &  \bar{\Sigma}_{i,R} = \Sigma_{i,R} + \rho {\rm deg}(i) I_{n^0-n_i} \in \mathbb{R}^{\rm (n^0-n_i) \times (n^0-n_i)},
    \end{aligned}
\end{equation*}

Next, we express the portion $\tilde{\boldsymbol{x}}_{-i}$ in $\boldsymbol{y}_i$ as an optimal reaction function of the constrained variable $\boldsymbol{x}_i$, and then reduce the dimension of the optimization problem $\mathcal{S}_{i,k}'$ to a constrained optimization problem with only respect to $\boldsymbol{x}_i$. To be more specific, we split the variable $\boldsymbol{y}_i$ into $\boldsymbol{x}_i$ and $\tilde{\boldsymbol{x}}_{-i}$, and transform the objective function in the optimization problem according to these two parts. For the quadratic term,
\begin{equation*}
    \begin{aligned}
        \frac{1}{2}\boldsymbol{y}_i^{\rm T}\Phi_i\boldsymbol{y}_i =& \frac{1}{2}\boldsymbol{y}_i^{\rm T} \left( \Sigma_i+\rho {\rm deg}(i)I_{n^0}+\sigma \tilde{A}_i^{\rm T}\tilde{A}_i\right)\boldsymbol{y}_i \\
        = & \frac{1}{2} \boldsymbol{x}_i^{\rm T} \bar{\Sigma}_{i,i} \boldsymbol{x}_i + \frac{1}{2}\tilde{\boldsymbol{x}}_{-i}^{\rm T} \bar{\Sigma}_{i,R} \tilde{\boldsymbol{x}}_{-i} + \left(\Sigma_{i,L}^{\rm T}\boldsymbol{x}_i\right)^{\rm T}\tilde{\boldsymbol{x}}_{-i} +  \frac{\sigma}{2}\boldsymbol{x}_i^{\rm T}A^{\rm T}A \boldsymbol{x}_i\\
        = & \frac{1}{2} \boldsymbol{x}_i^{\rm T} \left(\bar{\Sigma}_i + \sigma A^{\rm T}A \right)\boldsymbol{x}_i   + \frac{1}{2}\tilde{\boldsymbol{x}}_{-i}^{\rm T} \bar{\Sigma}_{i,R} \tilde{\boldsymbol{x}}_{-i} + \left(\Sigma_{i,L}^{\rm T}\boldsymbol{x}_i\right)^{\rm T}\tilde{\boldsymbol{x}}_{-i}. 
    \end{aligned}
\end{equation*}
For the linear term within $k$-th iteration,
\begin{equation*}
    \begin{aligned}
        \Psi_i(k)^{\rm T}\boldsymbol{y}_i =  & \left(\psi_i-\rho {\rm deg}(i)v_i(k)+\tilde{A}_i^{\rm T}l_i(k)-\sigma \tilde{A}_i^{\rm T}\tilde{A}_i\boldsymbol{y}_i(k)+\sigma \tilde{A}_i^{\rm T} \gamma_i(k)\right)^{\rm T} \boldsymbol{y}_i \\
        = & \psi_{i,i}^{\rm T} \boldsymbol{x}_i + \psi_{i,-i}^{\rm T}\tilde{\boldsymbol{x}}_{-i}- \rho {\rm deg}(i) v_{i,i}(k)^{\rm T}\boldsymbol{x}_i- \rho {\rm deg}(i) v_{i,-i}(k)^{\rm T}\tilde{\boldsymbol{x}}_{-i} + (A_i^{\rm T}l_i(k))^{\rm T}\boldsymbol{x}_i \\
        & - \sigma \boldsymbol{x}_i(k)^{\rm T}A_i^{\rm T}A_i \boldsymbol{x}_i + \sigma (A_i^{\rm T}\gamma_i(k))^{\rm T}\boldsymbol{x}_i   \\
        = & \left( \psi_{i,i} + A_i^{\rm T}l_i(k)+ \sigma A_i^{\rm T}\gamma_i(k) - \sigma A_i^{\rm T}A_i \boldsymbol{x}_{i}(k)- \rho {\rm deg}(i) v_{i,i}(k)\right)^{\rm T}\boldsymbol{x}_i \\
        & + \left(\psi_{i,-i}-\rho {\rm deg}(i) v_{i,-i}(k)\right)^{\rm T}\tilde{\boldsymbol{x}}_{-i}.
    \end{aligned}
\end{equation*}
From this, we can obtain the terms related to $\tilde{\boldsymbol{x}}_{-i}$. By treating the parts involving $\boldsymbol{x}_i$ as parameters, we can derive the best reaction function of $\tilde{\boldsymbol{x}}_{-i}$ with respect to $\boldsymbol{x}_i$. The function consisting of $\tilde{\boldsymbol{x}}_{-i}$ within $k$-th iteration is
\begin{equation*}
    \frac{1}{2}\tilde{\boldsymbol{x}}_{-i}^{\rm T} \bar{\Sigma}_{i,R} \tilde{\boldsymbol{x}}_{-i} + \left(\Sigma_{i,L}^{\rm T}\boldsymbol{x}_i +\psi_{i,-i}- \rho {\rm deg}(i) v_{i,-i}(k)\right)^{\rm T}\tilde{\boldsymbol{x}}_{-i}.
\end{equation*}
Thus, the best reaction function within $k$-th iteration is 
\begin{equation*}
    \tilde{\boldsymbol{x}}_{-i}^k(\boldsymbol{x}_i)= - \bar{\Sigma}_{i,R}^{-1} \Sigma_{i,L}^{\rm T} \boldsymbol{x}_i - \bar{\Sigma}_{i,R}^{-1}\left(\psi_{i,-i}-\rho {\rm deg}(i)  v_{i,-i}(k)\right).
\end{equation*}
Finally, by substituting the $\tilde{\boldsymbol{x}}_{-i}$ in optimization problem  with its best response function $\tilde{\boldsymbol{x}}_{-i}^k(\boldsymbol{x}_i)$, we have
\begin{equation*}
\begin{aligned}
    & \frac{1}{2} \boldsymbol{x}_i^{\rm T} \left(\bar{\Sigma}_i + \sigma A^{\rm T}A \right)\boldsymbol{x}_i  \\
     & +  \left( \psi_{i,i} + A_i^{\rm T}l_i(k)+ \sigma A_i^{\rm T}\gamma_i(k) - \sigma A_i^{\rm T}A_i \boldsymbol{x}_{i}(k)- \rho {\rm deg}(i) v_{i,i}(k)\right)^{\rm T}\boldsymbol{x}_i \\
     & + \frac{1}{2}\tilde{\boldsymbol{x}}_{-i}^k(\boldsymbol{x}_i)^{\rm T} \bar{\Sigma}_{i,R} \tilde{\boldsymbol{x}}_{-i}^k(\boldsymbol{x}_i) + \left(\Sigma_{i,L}^{\rm T}\boldsymbol{x}_i\right)^{\rm T}\tilde{\boldsymbol{x}}_{-i}^k(\boldsymbol{x}_i)
     + (\psi_{i,-i}-\rho {\rm deg}(i) v_{i,-i}(k))^{\rm T}\tilde{\boldsymbol{x}}_{-i}^k(\boldsymbol{x}_i) \\
     = & \frac{1}{2} \boldsymbol{x}_i^{\rm T} \left(\bar{\Sigma}_i + \sigma A^{\rm T}A \right)\boldsymbol{x}_i  \\
     & +  \left( \psi_{i,i} + A_i^{\rm T}l_i(k)+ \sigma A_i^{\rm T}\gamma_i(k) - \sigma A_i^{\rm T}A_i \boldsymbol{x}_{i}(k)- \rho {\rm deg}(i) v_{i,i}(k)\right)^{\rm T}\boldsymbol{x}_i \\
     & + \frac{1}{2} \left(- \bar{\Sigma}_{i,R}^{-1} \Sigma_{i,L}^{\rm T} \boldsymbol{x}_i - \bar{\Sigma}_{i,R}^{-1}\left(\psi_{i,-i}-\rho {\rm deg}(i)  v_{i,-i}(k)\right)\right)^{\rm T} \bar{\Sigma}_{i,R}\left(- \bar{\Sigma}_{i,R}^{-1} \Sigma_{i,L}^{\rm T} \boldsymbol{x}_i - \bar{\Sigma}_{i,R}^{-1}\left(\psi_{i,-i}-\rho {\rm deg}(i)  v_{i,-i}(k)\right)\right) \\
     & + \left(\Sigma_{i,L}^{\rm T}\boldsymbol{x}_i +\psi_{i,-i} -\rho {\rm deg}(i) v_{i,-i}(k) \right)^{\rm T}\left(- \bar{\Sigma}_{i,R}^{-1} \Sigma_{i,L}^{\rm T} \boldsymbol{x}_i - \bar{\Sigma}_{i,R}^{-1}\left(\psi_{i,-i}-\rho {\rm deg}(i)  v_{i,-i}(k)\right)\right) \\
     = & \frac{1}{2}
        \boldsymbol{x}_i^{\rm T}\check{\Phi}_i\boldsymbol{x}_i 
        +\check{\Psi}_{i}(k)^{\rm T}\boldsymbol{x}_i, 
\end{aligned}
\end{equation*}
where 
\begin{equation*}
    \begin{aligned}
        & \check{\Phi}_i  =  \bar{\Sigma}_{i,i}-\Sigma_{i,L}\bar{\Sigma}_{i,R}^{-1}\Sigma_{i,L}^{\rm T} + \sigma A^{\rm T}A , \\
        & \check{\Psi}_{i}(k) = \psi_{i,i} + A_i^{\rm T}l_i(k)+ \sigma A_i^{\rm T}\gamma_i(k) - \sigma A_i^{\rm T}A_i \boldsymbol{x}_{i}(k)- \rho {\rm deg}(i) v_{i,i}(k)  +\Sigma_{i,L}\bar{\Sigma}_{i,R}^{-1}\left( \psi_{i,-i} -\rho {\rm deg}(i)  v_{i,-i}(k) \right).
    \end{aligned}
\end{equation*}
Thus, we can obtain $\boldsymbol{x}_i(k+1)$ by solving the following problem,
\begin{equation*}
    \boldsymbol{x}_i(k+1)=\argmin\limits_{\boldsymbol{x}_i \in \Omega_i} \left\{\dfrac{1}{2}\boldsymbol{x}_i^{\rm T}\check{\Phi}_i\boldsymbol{x}_i+\check{\Psi}_i(k)^{\rm T}\boldsymbol{x}_i\right\}.
\end{equation*}
Given that the optimal solution, from the optimal response function of $\tilde{\boldsymbol{x}}_{-i}^k(\boldsymbol{x}_i)$, we can derive the following equation,
\begin{equation*}
    \tilde{\boldsymbol{x}}_{-i}^k(\boldsymbol{x}_i)= - \bar{\Sigma}_{i,R}^{-1} \Sigma_{i,L}^{\rm T} \boldsymbol{x}_i(k+1) - \bar{\Sigma}_{i,R}^{-1}\left(\psi_{i,-i}-\rho {\rm deg}(i)  v_{i,-i}(k)\right).
\end{equation*}
By the combination of $\boldsymbol{x}_i(k+1)$ and $\tilde{\boldsymbol{x}}_{-i}(k+1)$ to obtain $\boldsymbol{y}(k+1)$.

Therefore, under the setting of Example \ref{example:alg2}, line 11-13 in Algorithm \ref{alg:ATCADMM} can be written more specifically as,
    \begin{equation*}
    \begin{aligned}
     & {\rm step\ } 1. \quad 
            \boldsymbol{x}_i(k+1) 
        =\argmin\limits_{B_i\boldsymbol{x}_i \leq \boldsymbol{m}_i} \left\{\dfrac{1}{2}
       \boldsymbol{x}_i^{\rm T}\check{\Phi}_i\boldsymbol{x}_i 
        +\check{\Psi}_{i}(k)^{\rm T}\boldsymbol{x}_i \right\}. \\
      & {\rm step\ } 2. \quad  \tilde{\boldsymbol{x}}_{-i}(k+1)  = - \bar{\Sigma}_{i,R}^{-1} \Sigma_{i,L}^{\rm T} \boldsymbol{x}_i(k+1) - \bar{\Sigma}_{i,R}^{-1}\left(\psi_{i,-i}-\rho {\rm deg}(i)  v_{i,-i}(k)\right). \\
      & {\rm step\ } 3. \quad  \boldsymbol{y}_i(k+1)  = \text{col}\left( \tilde{\boldsymbol{x}}_{-i,l}(k+1), \boldsymbol{x}_i(k+1), \tilde{\boldsymbol{x}}_{-i,r}(k+1)\right).
    \end{aligned}
\end{equation*}

Next, Figures \ref{fig:Imp_Opt_gap}-\ref{fig:Imp_Con_vio} illustrate the different performance between Algorithm \ref{alg:TCADMM} and Algorithm \ref{alg:ATCADMM} under identical computation times(neglecting the communication time in a distributed framework), utilizing optimization problem \eqref{eq:example-tran} in Example \ref{example:allo} and the network structure of different scales in Figure \ref{fig:trans_graph}.

\begin{figure}[!ht] 
    \centering  
    \begin{minipage}{.33\textwidth}  
        \centering  
        \includegraphics[width=\linewidth]{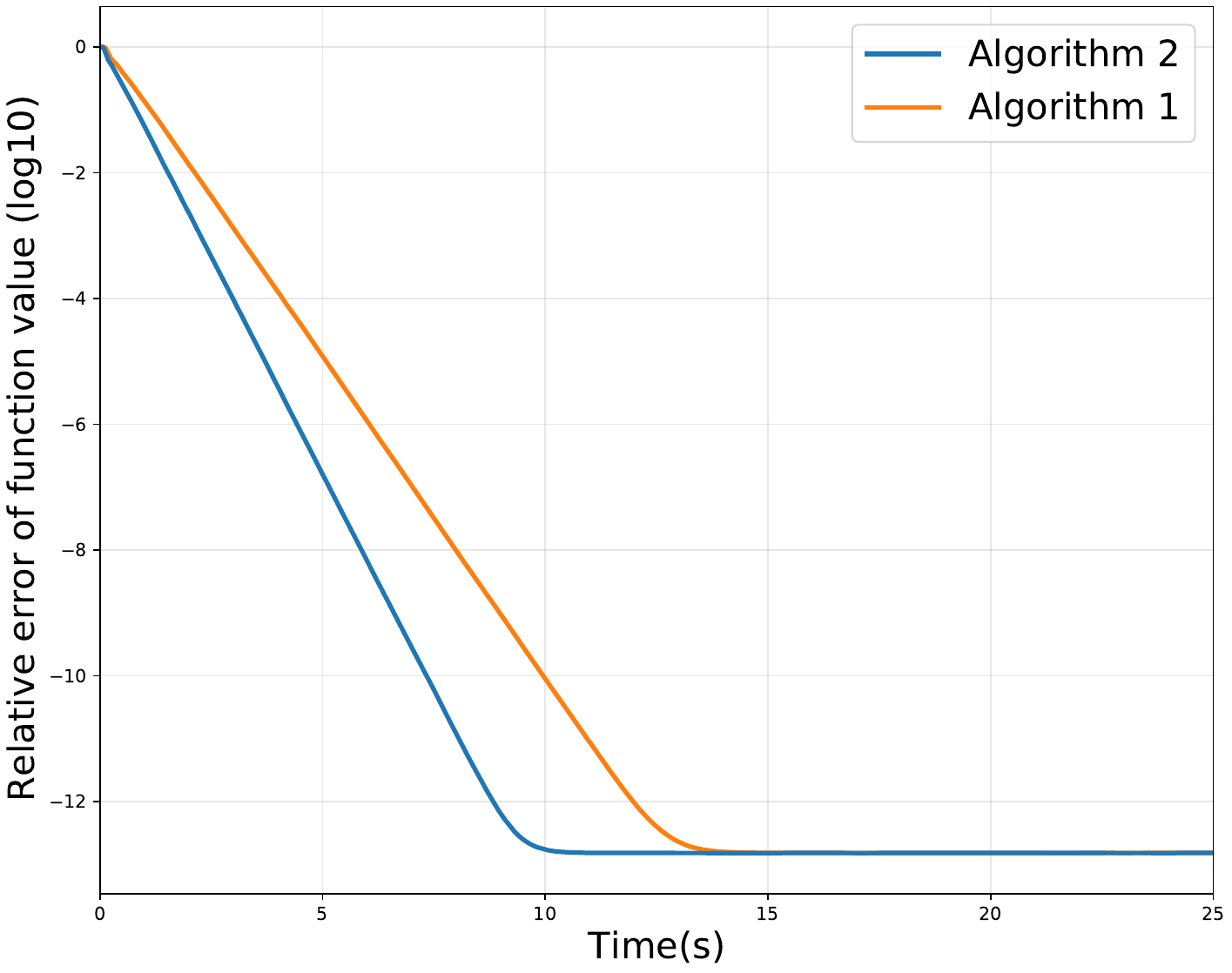}  
        \subcaption{Small-scale}  
    \end{minipage}%
    \hfill 
    \begin{minipage}{.33\textwidth}  
        \centering  
        \includegraphics[width=\linewidth]{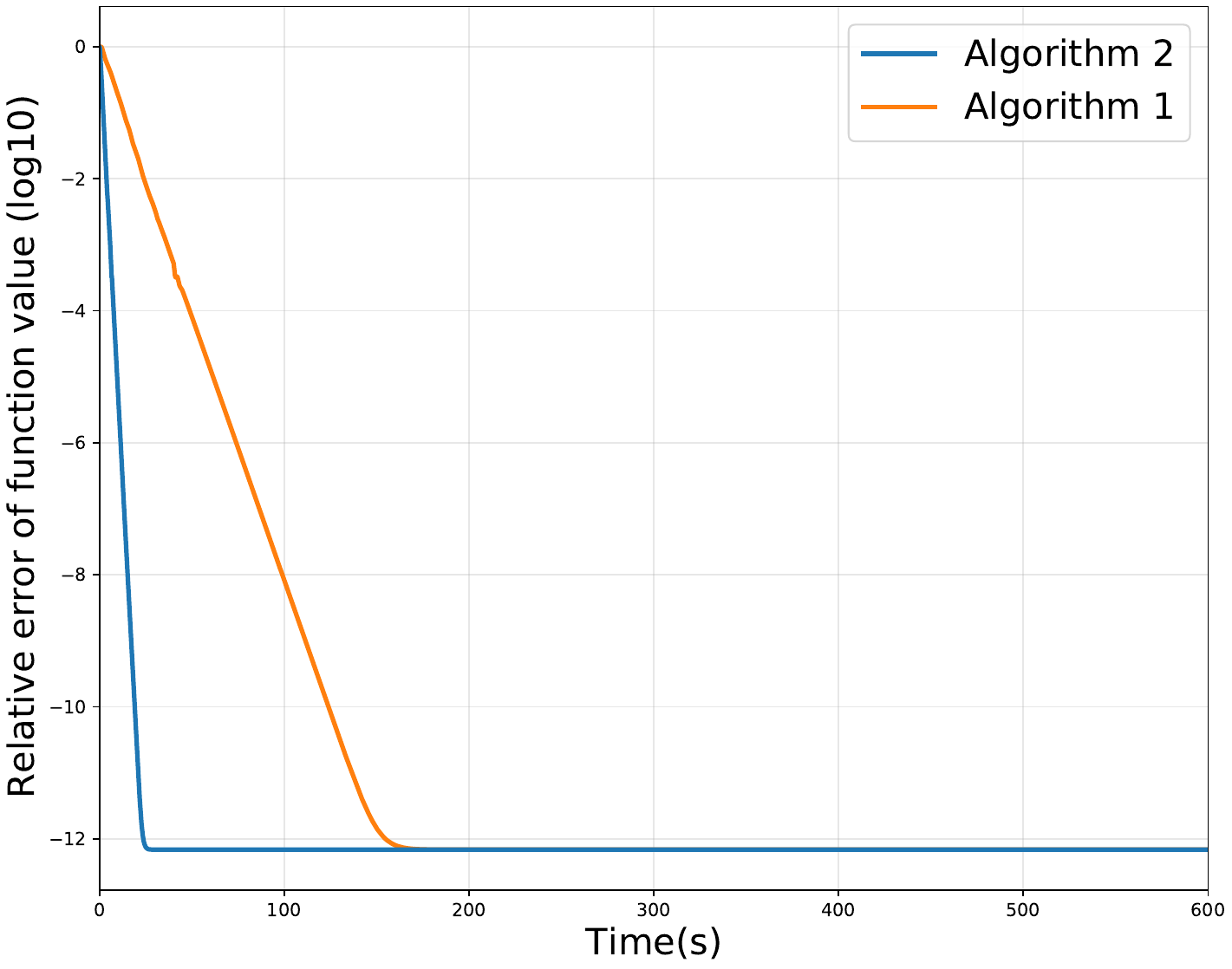}  
        \subcaption{Medium-scale}  
    \end{minipage}%
    \hfill 
    \begin{minipage}{.33\textwidth}  
        \centering  
        \includegraphics[width=\linewidth]{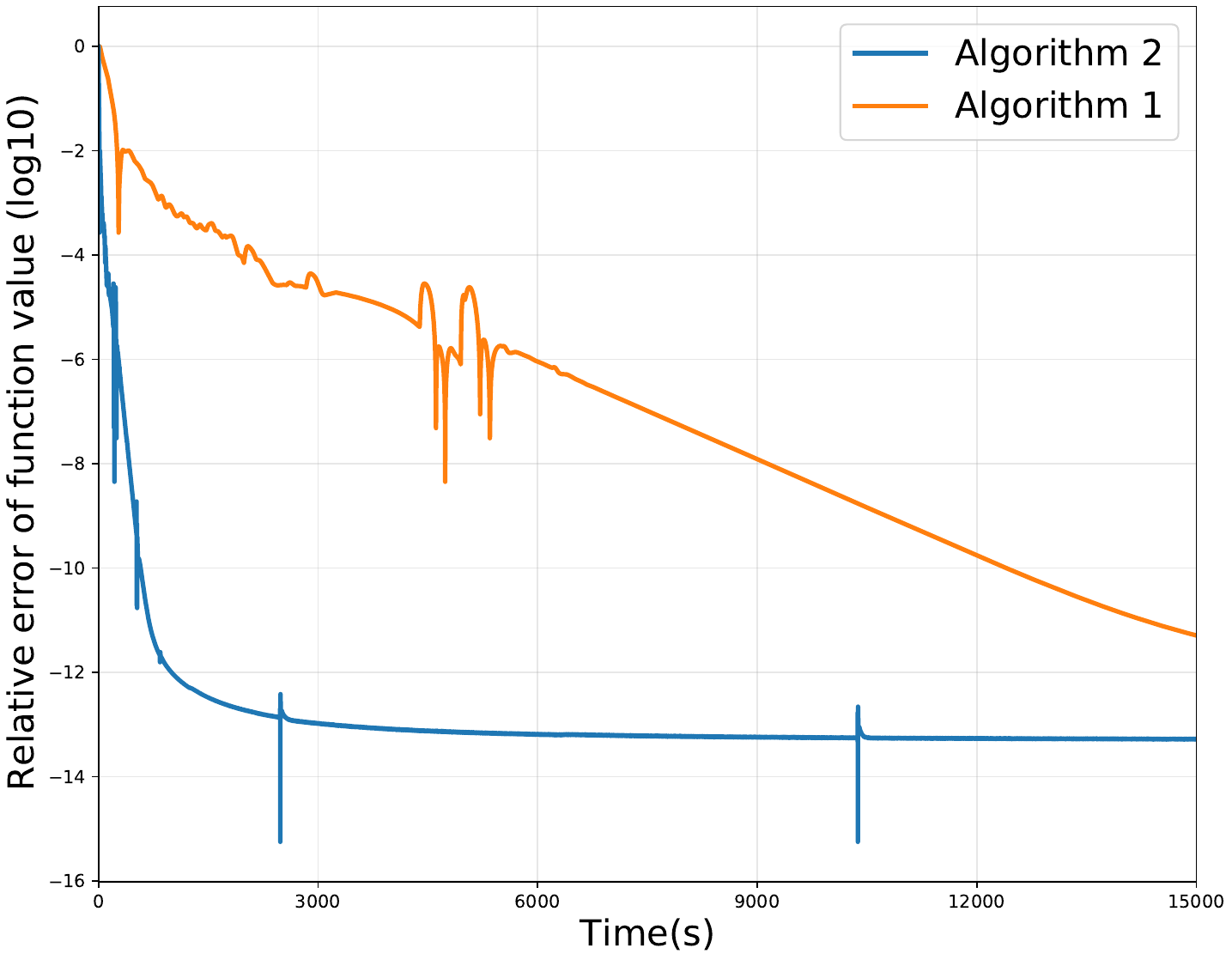}  
        \subcaption{Large-scale}  
    \end{minipage}%
    \caption{Improvement in optimality gap performance}  
    \label{fig:Imp_Opt_gap}  
\end{figure} 

\begin{figure}[!ht]  
    \centering  
    \begin{minipage}{.33\textwidth}  
        \centering  
        \includegraphics[width=\linewidth]{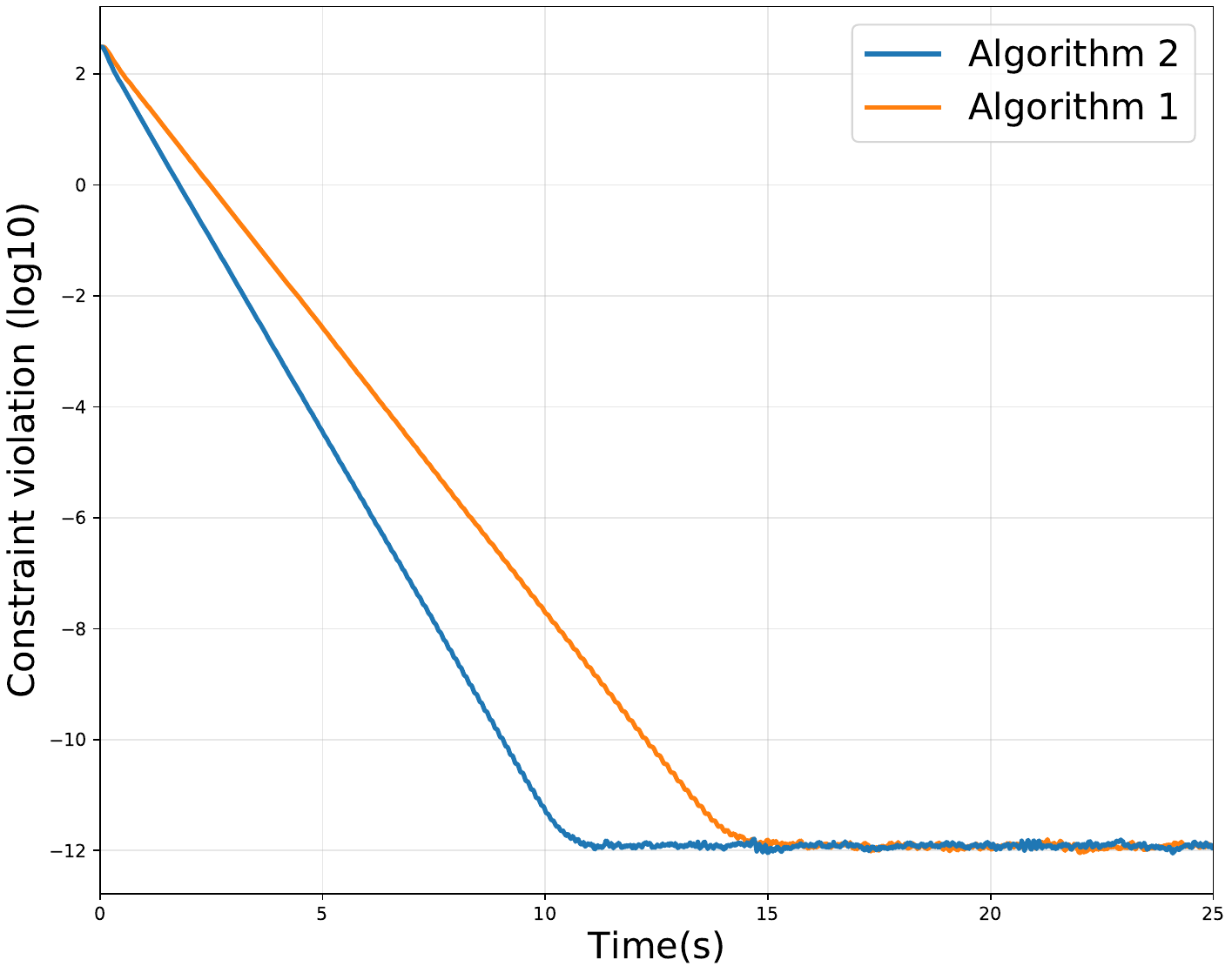}  
        \subcaption{Small-scale}  
    \end{minipage}%
    \hfill 
    \begin{minipage}{.33\textwidth}  
        \centering  
        \includegraphics[width=\linewidth]{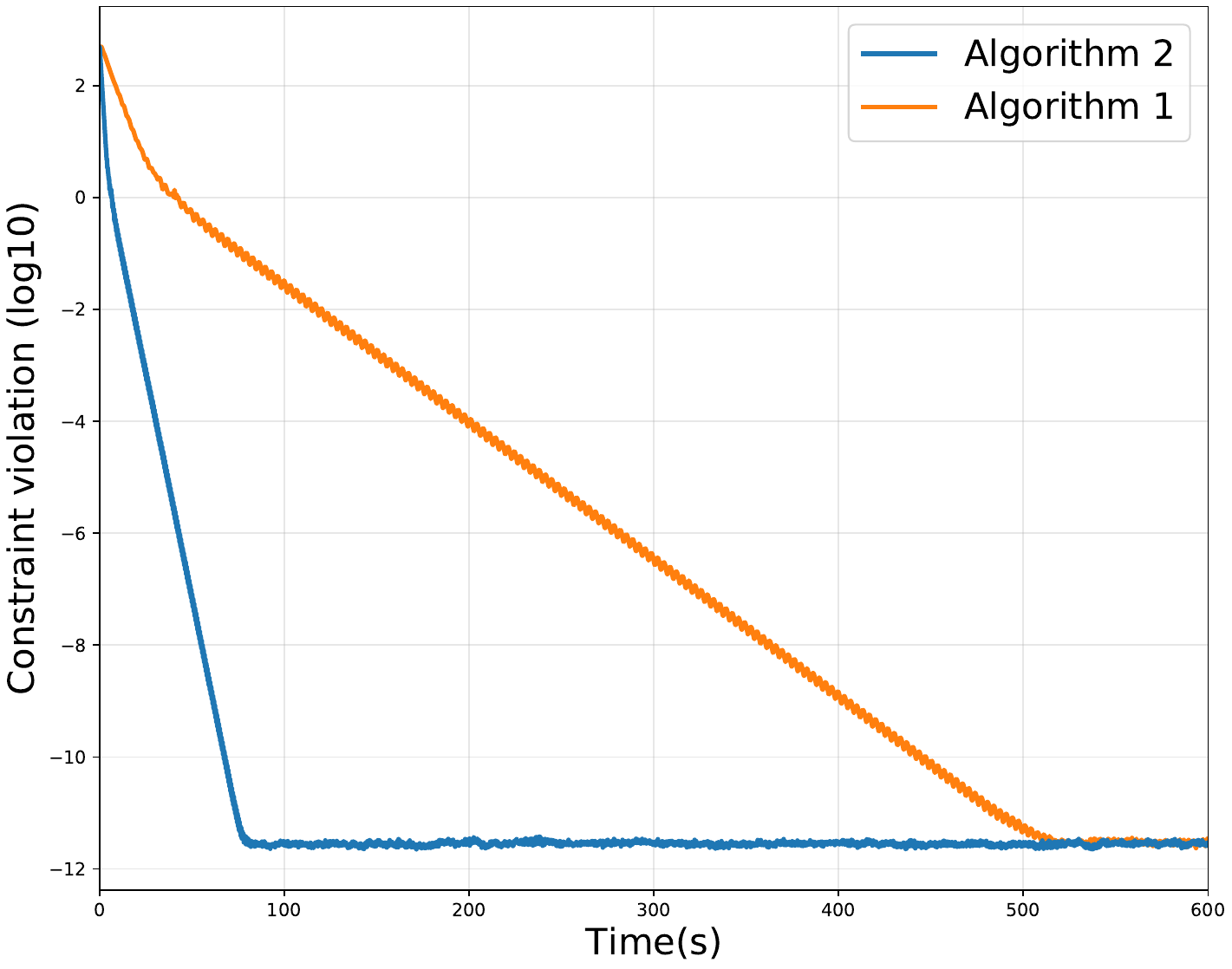}  
        \subcaption{Medium-scale}  
    \end{minipage}%
    \hfill 
    \begin{minipage}{.33\textwidth}  
        \centering  
        \includegraphics[width=\linewidth]{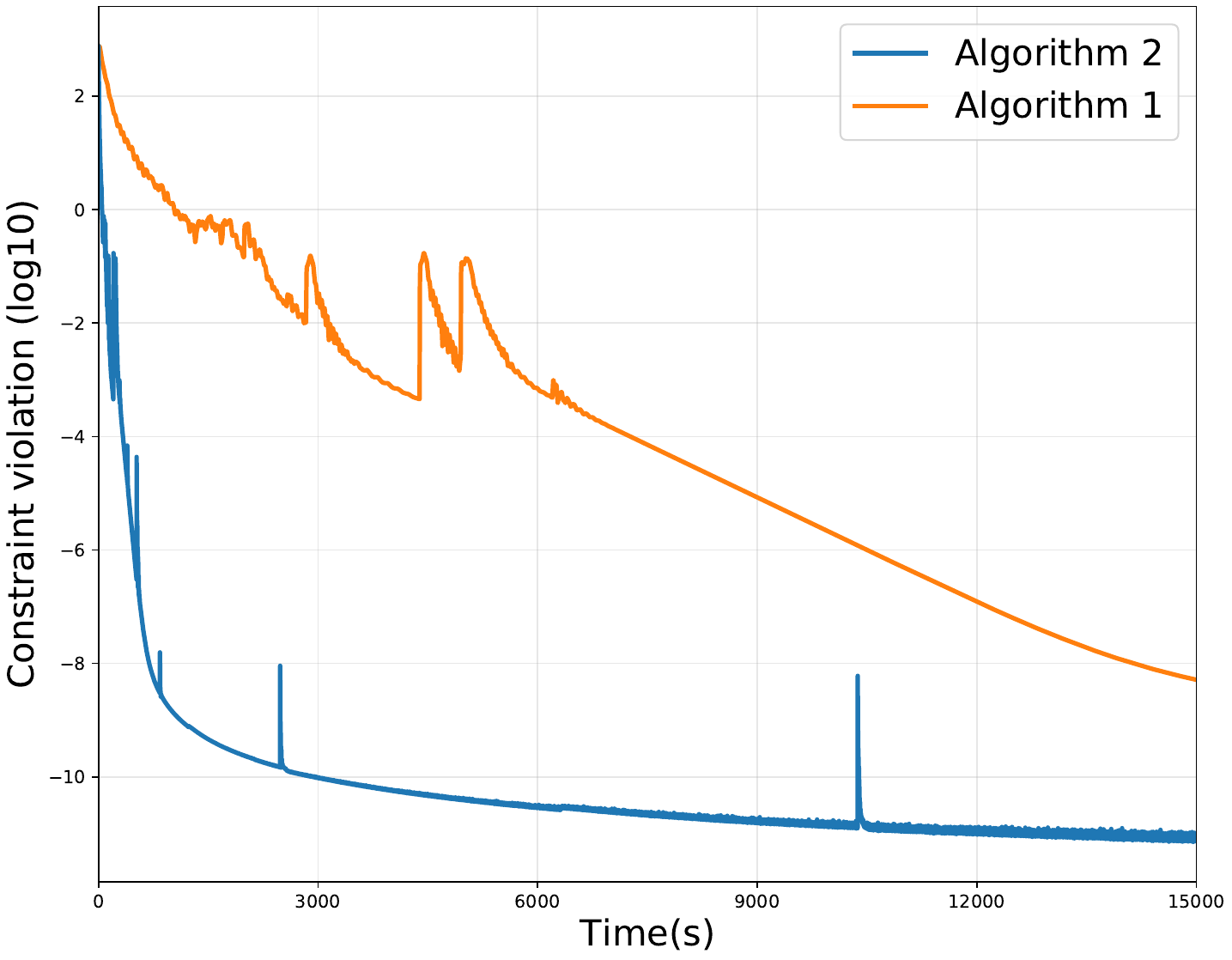}  
        \subcaption{Large-scale}  
    \end{minipage}%
    \caption{Improvement in constraint violation performance}  
    \label{fig:Imp_Con_vio}  
\end{figure}

The figures demonstrate that dimensionality reduction for the sub-optimization problems $\mathcal{S}_{i,k}$ enhances the overall computational efficiency of the algorithm. Moreover, the improvement effect becomes more pronounced for problems with larger parameter scales.

\section{Typical example}\label{app:typ_exa}

\begin{example}\label{ex:general_case_1}
 Based on example \ref{ex:simple_case}, we consider a more general case as shown in Figure \ref{fig:gen_cas_1}, consisting of $N$ supply nodes and one demand node $(M=1)$. Each supply node has only one path $(R=1)$ to the demand node, and only one type of material $(K=1)$ is transported. The supply nodes have no constraints on inventory capacity and transportation capacity. 
    \begin{figure}[!ht]
        \centering
        \includegraphics[scale=0.75]{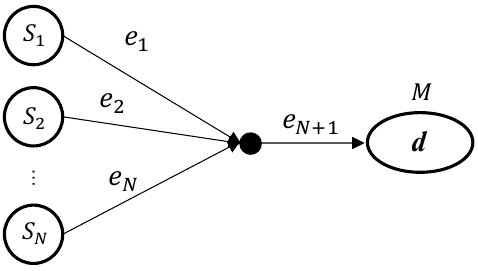}
        \caption{Transportation network in Example \ref{ex:general_case_1}}\label{fig:gen_cas_1}
    \end{figure}
\end{example}

In the setting of the above example, the matrix $Q_i$ degenerates to a vector, i.e., $Q_i=\boldsymbol{e}_{N+1,i}+\boldsymbol{e}_{N+1,N+1}$, and the optimization problem \eqref{eq:example-tran} can be written as
   \begin{equation}
     \begin{aligned}
         \min_{\boldsymbol{x}} \quad & \sum\nolimits_{i=1}^N\left(c_0\left(\sum\nolimits_{s=1}^{N}Q_sx_s\right)^{\rm T}Q_ix_i+\boldsymbol{c}_i^{\rm T}Q_ix_i\right)=\sum\nolimits_{i=1}^N \left(c_0 \left(x_i^2+x_i\sum\nolimits_sx_s\right)+\|\boldsymbol{c}_i\|_1x_i\right) \\
        {\rm s.t.} \quad \  &\sum\nolimits_{i} x_{i} =d, \quad \leftrightarrow \quad \lambda\\
        & \ x_{i} \geq 0, \quad \leftrightarrow \quad \alpha_i, \quad  \forall i.
    \end{aligned}
 \end{equation}
The Lagrangian function associated with the above optimization problem is 
\begin{equation*}
    \mathcal{L} = \sum\nolimits_{i} \left(c_0 \left(x_i^2+x_i\sum\nolimits_sx_s\right)+\|\boldsymbol{c}_i\|_1x_i\right) - \lambda^{\rm T}\left(\sum\nolimits_i x_i - {d}\right)- \sum\nolimits_i\alpha_i{x}_i,
\end{equation*}
and the stability conditions are
\begin{equation*}
   \nabla_{ x_i}\mathcal{L}= 2c_0x_i^*+2c_0d + \|\boldsymbol{c}_i\|_1-\lambda^* - \alpha_i^* = 0, \quad \forall i.
\end{equation*}
If we assume that there are $T$ participants whose optimal solutions are non-zero $(x_i \neq 0)$, denote these participants as forming the set $\mathcal{T}$, and the remaining $N-T$ participants as forming the set $\mathcal{S}$. Combined with the rest of the KKT conditions, we can arrive at the following results. 
\begin{equation*}
        \lambda^*=2\frac{T+1}{T}c_0\boldsymbol{d} + \frac{1}{T}\sum\nolimits_{t=1}^T \|\boldsymbol{c}_{j_t}\|, \quad j_t \in \mathcal{T}.
\end{equation*}
If $i = j_t \in \mathcal{T}$, then
\begin{equation*}
    \left\{
    \begin{array}{l}
     x_i^*= \frac{1}{T}\boldsymbol{d} + \frac{1}{2c_0} \left(\frac{1}{T}\sum\nolimits_{k=1}^T \|\boldsymbol{c}_{j_k}\|-\|\boldsymbol{c}_i\|\right),   \\
      \alpha_i^* = 0.  
    \end{array}
    \right.
\end{equation*}
If $i = i_s \in \mathcal{S}$, then
\begin{equation*}
    \left\{
    \begin{array}{l}
     x_i^*= 0,   \\
      \alpha_i^* = \|\boldsymbol{c}_{i}\|-\frac{1}{T}\left(2c_0 \boldsymbol{d}+\sum\nolimits_{t=1}^T \|\boldsymbol{c}_{j_t}\|\right).  
    \end{array}
    \right.
\end{equation*}
The determination of sets $\mathcal{S}$ and $\mathcal{S}$ can be achieved through comparative analysis of parameter $\|\boldsymbol{c}_i\|_1$, with boundary conditions specified as
\begin{equation*}
    \|\boldsymbol{c}_{i_s}\| \geq \frac{1}{T}\left(2c_0 \boldsymbol{d}+\sum\nolimits_{t=1}^T \|\boldsymbol{c}_{j_t}\|\right), \quad j_t \in \mathcal{T},\ \forall i_s \in \mathcal{S}.
\end{equation*}

Then, we calculate the price guidance signal by equation \eqref{eq:LMP_price} and the profit $|u_i^*|$ for each participant under the SP mechanism.
\begin{equation*}
    \pi_i^*=\lambda^*-c_0Q_i^{\rm T}\left(\sum\nolimits_{s \neq i}Q_sx_s^*\right)=\left\{
    \begin{array}{ll}
      \frac{T+3}{T}c_0\boldsymbol{d} + \frac{2}{3T}\sum\nolimits_{t=1}^T\|\boldsymbol{c}_{j_t}\|_1-\frac{1}{2}\|\boldsymbol{c}_i\|_1,   & i \in \mathcal{T}, \\
      \frac{T+2}{T}c_0\boldsymbol{d} + \frac{1}{T}\sum\nolimits_{t=1}^T\|\boldsymbol{c}_{j_t}\|_1,   & i \in \mathcal{S}.
    \end{array}
    \right.
\end{equation*}
\begin{equation*}
    |u_i^*| = \pi_i^*x_i^* - c_0\left(\sum\nolimits_{s=1}^{N}Q_sx_s\right)^{\rm T}Q_ix_i - \boldsymbol{c}_i^{\rm T}Q_ix_i = \left\{\!\!
    \begin{array}{ll}
      \frac{1}{2c_0T^2}\left(2c_0\boldsymbol{d}+\sum\nolimits_{t=1}^T\|\boldsymbol{c}_{j_t}\|-T\|\boldsymbol{c}_i\|\right)^2,   & i \in \mathcal{T}, \\
      0,   & i \in \mathcal{S}.
    \end{array}
    \right.
\end{equation*}

Next, we investigate the changes of price signals and profits for all participants when a single member of the set $\mathcal{T}$ engages in misreporting, while the remaining participants maintain truthful reporting. Let participant $j_{t_0}$ introduce a perturbation of magnitude $\Delta$ to the norm of the parameter vector, resulting in the new parameter norm $\|\boldsymbol{c}_{j_{t_0}}\|+\Delta$. The motivation for participant $j_{t_0}$ to misreport is to increase its own profit, as $|u_i^*|=\frac{1}{2c_0T^2}\left(2c_0\boldsymbol{d}+\sum\nolimits_{t=1}^T\|\boldsymbol{c}_{j_t}\|-T\|\boldsymbol{c}_i\|\right)^2$ indicates this is only possible when $\Delta$ is negative. Consequently, changes in $\Delta$ will not cause changes to the elements in sets $\mathcal{T}$ and $\mathcal{S}$. Let $\pi_i^{\circ}$ and $|u_i^{\circ}|$ represent the optimal price signal and profit after misreporting, respectively. Using the previous method, we can get
\begin{equation*}
    \pi_i^{\circ}(\Delta)=\left\{
    \begin{array}{ll}
      \frac{T+3}{T}c_0\boldsymbol{d} + \frac{2}{3T}\sum\nolimits_{t=1}^T\left(\|\boldsymbol{c}_{j_t}\|_1+\Delta\cdot \boldsymbol{1}_{\{i=j_{t_0}\}}\right)-\frac{1}{2}\left(\|\boldsymbol{c}_i\|_1+\Delta\cdot \boldsymbol{1}_{\{i=j_{t_0}\}}\right),   & i \in \mathcal{T}, \\
      \frac{T+2}{T}c_0\boldsymbol{d} + \frac{1}{T}\sum\nolimits_{t=1}^T\left(\|\boldsymbol{c}_{j_t}\|_1+\Delta\cdot \boldsymbol{1}_{\{i=j_{t_0}\}}\right),   & i \in \mathcal{S}.
    \end{array}
    \right.
\end{equation*}
\begin{equation*}
    |u_i^{\circ}(\Delta)| = \left\{
    \begin{array}{ll}
      |u_i^*|-\frac{T-2}{2c_0T^2}\Delta\left(2c_0\boldsymbol{d}+\sum\nolimits_{t=1}^T\|\boldsymbol{c}_{j_t}\|-T\|\boldsymbol{c}_i\|\right)-\frac{T-1}{2c_0T^2}\Delta^2,   & i \in \mathcal{T}, \ i = j_{t_0}, \\
      |u_i^*|+\frac{1}{c_0T^2}\Delta\left(2c_0\boldsymbol{d}+\sum\nolimits_{t=1}^T\|\boldsymbol{c}_{j_t}\|-T\|\boldsymbol{c}_i\|\right)+\frac{1}{2c_0T^2}\Delta^2,   & i \in \mathcal{T},\ i \neq j_{t_0}, \\
      0,   & i \in \mathcal{S}.
    \end{array}
    \right.
\end{equation*}
As can be observed, for participant $j_{t_0}$, the change in profit follows a concave quadratic function with respect to its misreporting level. This implies that reducing the cost parameter within a certain range can increase its profit. For other participants, however, the payoff variation with respect to participant $j_{t_0}$'s misreporting level follows a convex quadratic function, indicating that their profits are subject to decreases due to misreporting.

If a participant can benefit from reducing the cost parameter, then the remaining participants can do the same. Consequently, we ultimately analyze the optimal magnitude of parameter reduction when every member of set $\mathcal{T}$ engages in such coordinated parameter minimization. To simplify the discussion, we assume that participants in set $\mathcal{S}$ report truthfully, whereas each participant in set $\mathcal{T}$ has a misreporting magnitude $\Delta_t$. Let $\pi_i^{\diamond}$ and $|u_i^{\diamond}|$ represent the optimal price signal and profit after misreporting, respectively. Using the previous method, we can get
\begin{equation*}
    \pi_i^{\diamond} = \left\{
    \begin{array}{ll}
      \frac{T+3}{T}c_0\boldsymbol{d} + \frac{2}{3T}\sum\nolimits_{t=1}^T\left(\|\boldsymbol{c}_{j_t}\|_1+\Delta_{j_t}\right)-\frac{1}{2}\left(\|\boldsymbol{c}_i\|_1+\Delta_i\right),   & i \in \mathcal{T}, \\
      \frac{T+2}{T}c_0\boldsymbol{d} + \frac{1}{T}\sum\nolimits_{t=1}^T\left(\|\boldsymbol{c}_{j_t}\|_1+\Delta_{j_t}\right),   & i \in \mathcal{S}.
    \end{array}
    \right.
\end{equation*}
{\small \begin{equation*}
    |u_i^{\diamond}| = \left\{\!\!
    \begin{array}{ll}
      |u_i^*|+\frac{1}{2c_0T^2}\left(\sum\nolimits_t \Delta_{j_t} - T \Delta_i\right)\left(\sum\nolimits_t \Delta_{j_t}\right)+\frac{1}{c_0T^2}\left(\sum\nolimits_t \Delta_{j_t} - \frac{1}{2} T \Delta_i\right)\left(2c_0\boldsymbol{d}+\sum\nolimits_{t=1}^T\|\boldsymbol{c}_{j_t}\|-T\|\boldsymbol{c}_i\|\right),   & i \in \mathcal{T}, \\
      0,   & i \in \mathcal{S}.
    \end{array}
    \right.
\end{equation*}}
As can be seen from the above equation, the profit of every participant is affected by the collective misreporting magnitudes of all participants. Let
{\small \begin{equation*}
    G_i = \frac{1}{2c_0T^2}\left(\sum\nolimits_t \Delta_{j_t} - T \Delta_i\right)\left(\sum\nolimits_t \Delta_{j_t}\right)+\frac{1}{c_0T^2}\left(\sum\nolimits_t \Delta_{j_t} - \frac{1}{2} T \Delta_i\right)\left(2c_0\boldsymbol{d}+\sum\nolimits_{t}\|\boldsymbol{c}_{j_t}\|-T\|\boldsymbol{c}_i\|\right), i \in \mathcal{T}.
\end{equation*}}
Given the  misreporting levels of other participants, the best response of participant $i$ must satisfy the following equation,
\begin{equation*}
    \frac{\partial G_i}{\partial \Delta_i}  =0, \quad \forall i \in \mathcal{T},
\end{equation*}
that is 
\begin{equation*}
    \Delta_i^{\diamond}=-\frac{T-2}{2(T-1)}\left(2c_0\boldsymbol{d}+\sum\nolimits_{t}\|\boldsymbol{c}_{j_t}\|-T\|\boldsymbol{c}_i\|+\sum\nolimits_{j_t \neq i} \Delta_{j_t}\right), \quad \forall i \in \mathcal{T}.
\end{equation*}
By simultaneously solving the above $T$ equations, we obtain
\begin{subequations}
\begin{align}
   & \sum\nolimits_t \Delta_{j_t}^{\diamond}= -\frac{T-2}{T-1}2c_0\boldsymbol{d}, \\
   & \Delta_i^{\diamond}=-\frac{T-2}{T(T-1)} \left[2c_0\boldsymbol{d}+(T-1)\left(\sum\nolimits_t \|\boldsymbol{c}_{j_t}\|-T\|\boldsymbol{c}_i\|\right)\right]. \label{eq:best_lie_vol}
\end{align}
\end{subequations}
The optimal misreporting quantity $\Delta_i^{\diamond}$ means that when all participants use this variation magnitude, they can reach an equilibrium where no further misreporting occurs. At this point, the profit of participant $i$ is
\begin{equation*}
    |u_i^{\diamond}|=|u_i^*|-\frac{T-2}{2c_0T^2(T-1)} \left[4c_0^2\boldsymbol{d}^2 - (T-1)\left(\sum\nolimits_t \|\boldsymbol{c}_{j_t}\|-T\|\boldsymbol{c}_i\|\right)^2\right].
\end{equation*}
From the expression, it can be observed that participants' profits decrease. In practice, participants cannot computationally determine the optimal $\Delta_i$ because the expression for $\Delta_i^*$ contains parameters $T$ and $\|\boldsymbol{c}_{j_t}\|$—information that participants lack access to. Consequently, it is difficult for participants to gain additional profits through misreporting. The most reliable approach is to use truthful information.
\section{Supplementary numerical experiments} \label{app:supp_exp}

In the main text, we presented the well-balanced network structure (no supply node has a path length advantage to any demand node) in Figure \ref{subfig:small_graph}. For reader convenience in the supplementary material, this figure is reproduced here as Figure \ref{fig:Add_net_str}(\subref{subfig:well-balanced}). In this section, we introduce two additional structures: semi-balanced and unbalanced networks. Figure \ref{fig:Add_net_str}(\subref{subfig:semi-balanced}) illustrates the semi-balanced network, characterized by a specific supply node (e.g., N1) having the shortest path to a particular demand node (e.g., M1). Figure \ref{fig:Add_net_str}(\subref{subfig:unbalanced}) depicts the unbalanced network, distinguished by a supply node (e.g., N1) possessing the shortest paths to all demand nodes.

\begin{figure}[!ht]
 \centering  
    \begin{minipage}{.33\textwidth}  
        \centering  
        \includegraphics[height=40mm,width=45mm]{Graph_S.pdf}  
        \subcaption{well-balanced} \label{subfig:well-balanced} 
    \end{minipage}%
    \hfill
    \centering  
    \begin{minipage}{.33\textwidth}  
        \centering  
        \includegraphics[height=40mm,width=45mm]{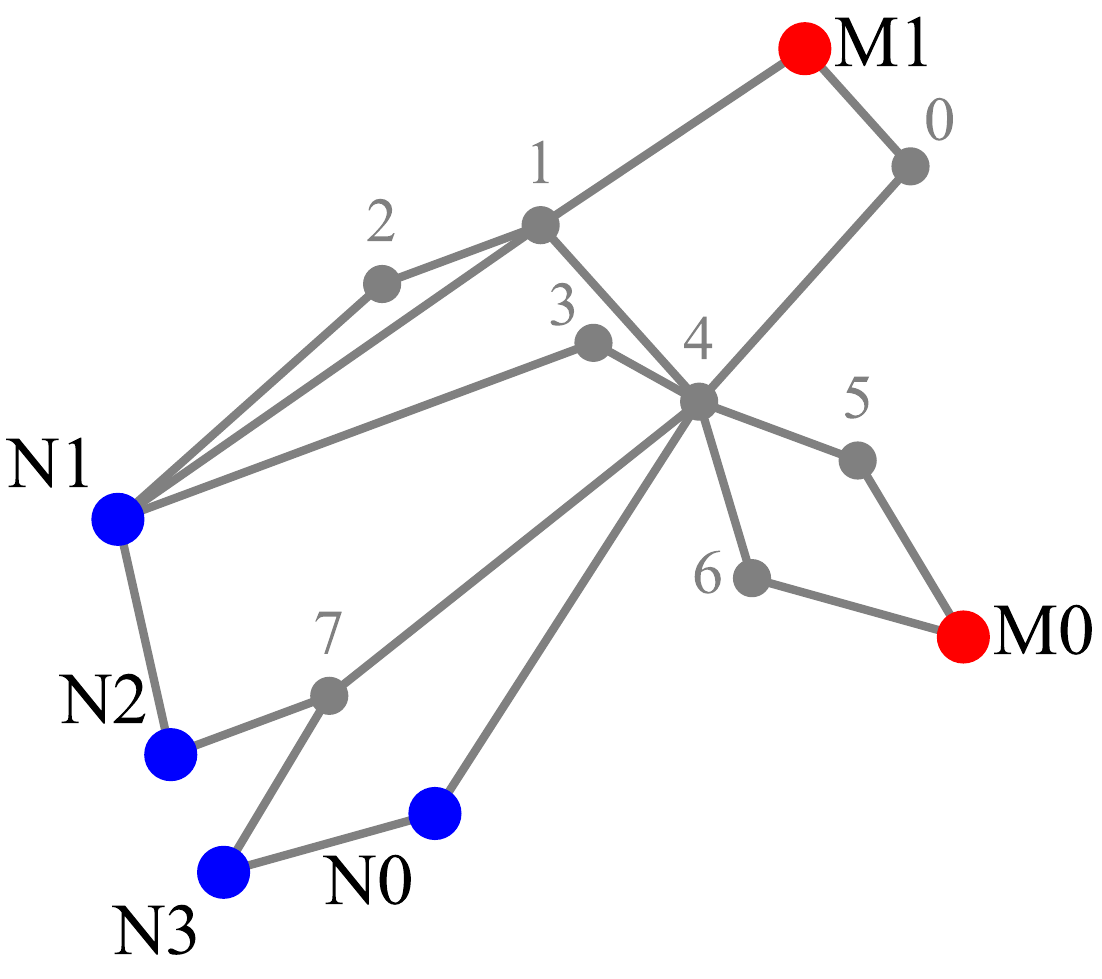}  
        \subcaption{semi-balanced} \label{subfig:semi-balanced} 
    \end{minipage}%
    \hfill
    \begin{minipage}{.33\textwidth}  
        \centering  
        \includegraphics[height=40mm,width=45mm]{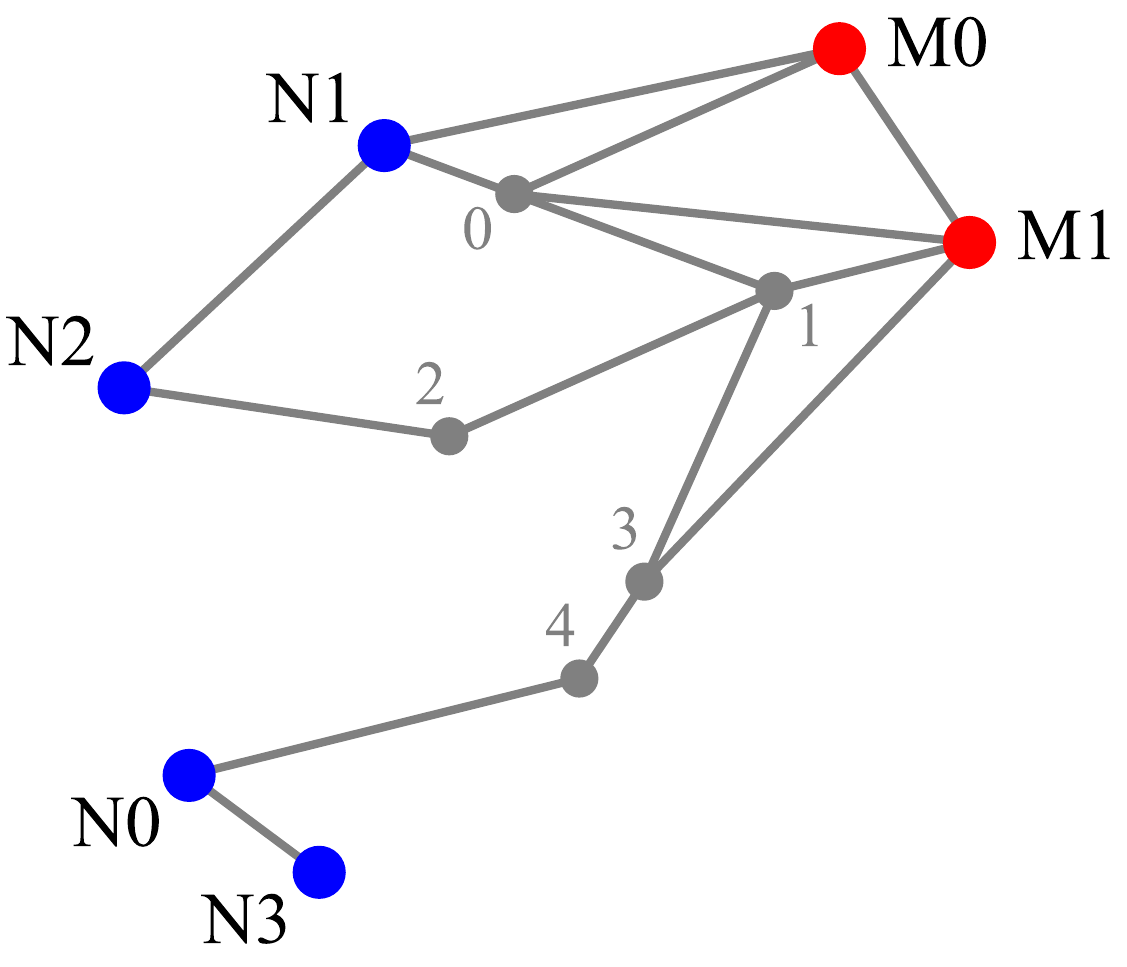}  
        \subcaption{unbalanced}\label{subfig:unbalanced}  
    \end{minipage}%
    \caption{Additional network structures}  
    \label{fig:Add_net_str}  
\end{figure} 

Figure \ref{fig:supp0} illustrates the impact of different parameter variations on the incentives received by participants in a well-balanced network. The parameters include the task target $\boldsymbol{d}$, the individual constraint limit $\boldsymbol{m}$ (where the $\boldsymbol{m}_i$ for each individual are uniformly scaled), and the congestion cost coefficient $c_0$. As presented in Table \ref{tab:incentive} of the main text, the SP mechanism and the VCG mechanism exhibit a modest disparity in the incentives they provide to the overall system under the well-balanced network structure. The three subfigures in Figure \ref{fig:supp0} offer a more detailed examination at the individual level, illustrating how variations in different parameters affect the incentives provided by the two mechanisms. Figure \ref{fig:supp0}(\subref{subfig:supp_d0}) demonstrates that when the task volume is small, node N1 exhibits a higher level of importance. As the task volume increases, the importance of node N0 gradually rises, eventually exceeding that of N1. This significance stems from the topological position of N0: it possesses the same number of paths to node M0 as N1 does, and the same number of paths to node M1 as N2 does. Consequently, N0 gains the capability to offload tasks from other nodes as the task volume grows. Figure \ref{fig:supp0}(\subref{subfig:supp_m0}) indicates that as the inventory capacity increases, nodes previously constrained by inventory limitations -- which hindered their significant contribution to tasks -- gradually begin to function effectively. Consequently, when the individual constraint limit is increased to 1.2 times its original value, the incentives received by all nodes become more evenly distributed compared to the low constraint limit case. Figure \ref{fig:supp0}(\subref{subfig:supp_c0}) demonstrates that the importance of node N0 declines as congestion costs increase. That occurs because N0 handles a larger task volume under low congestion costs, but when congestion costs rise, its total cost surges, diminishing its cost advantage. Consequently, as other nodes capture a greater share of tasks by leveraging their cost advantages, their importance becomes increasingly evident.

\begin{figure}[!ht]
    \centering
    \begin{minipage}[t]{0.98\textwidth} 
        \centering
        \includegraphics[height=50mm, width=150mm]{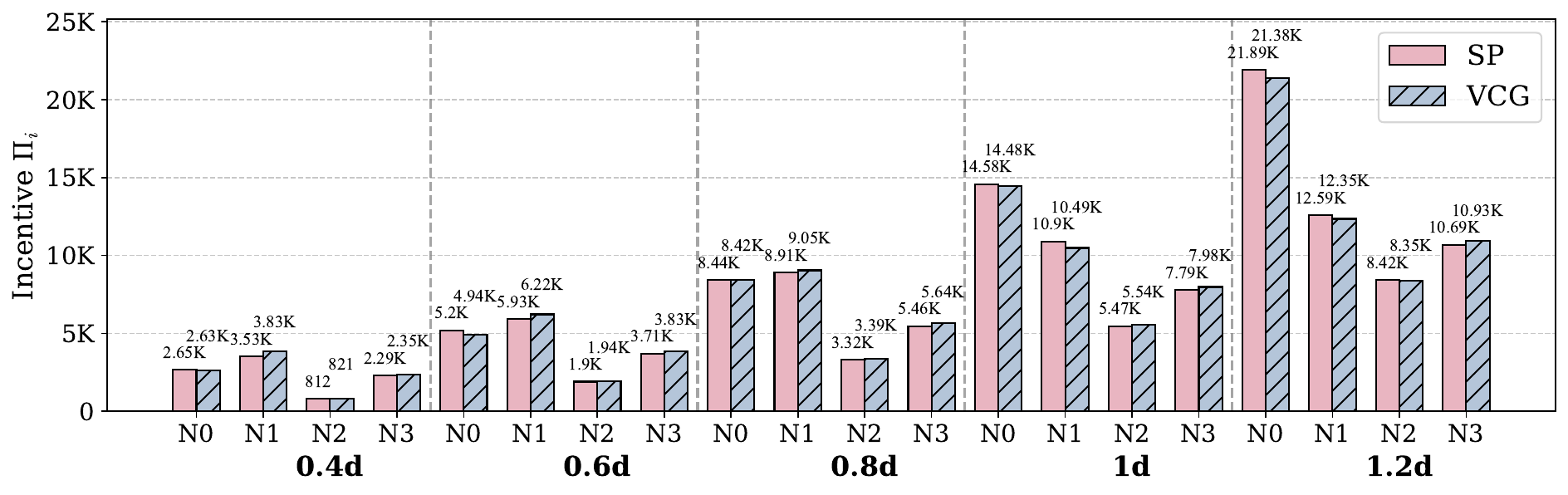}
        \subcaption{Demand parameter $\boldsymbol{d}$}
        \label{subfig:supp_d0}
    \end{minipage}
    
    \vspace{3mm}
    
    \begin{minipage}[t]{0.98\textwidth} 
        \centering
        \includegraphics[height=50mm, width=150mm]{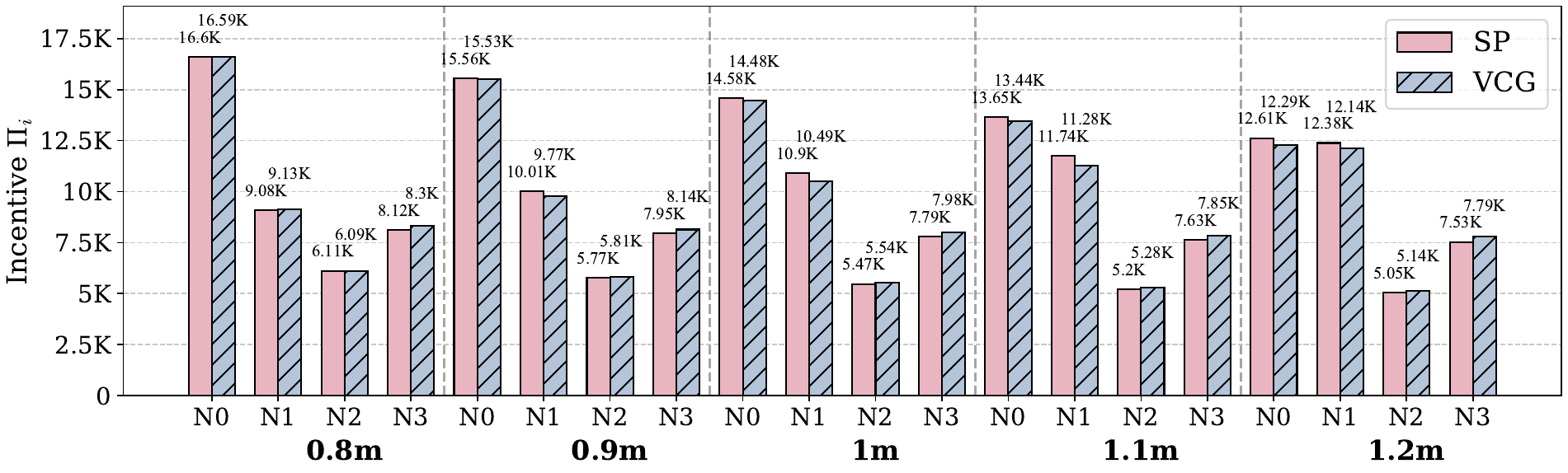}
        \subcaption{Constraint parameter $\boldsymbol{m}$}
        \label{subfig:supp_m0}
    \end{minipage}
    
    \vspace{3mm}
    
    \begin{minipage}[t]{0.98\textwidth} 
        \centering
        \includegraphics[height=50mm, width=150mm]{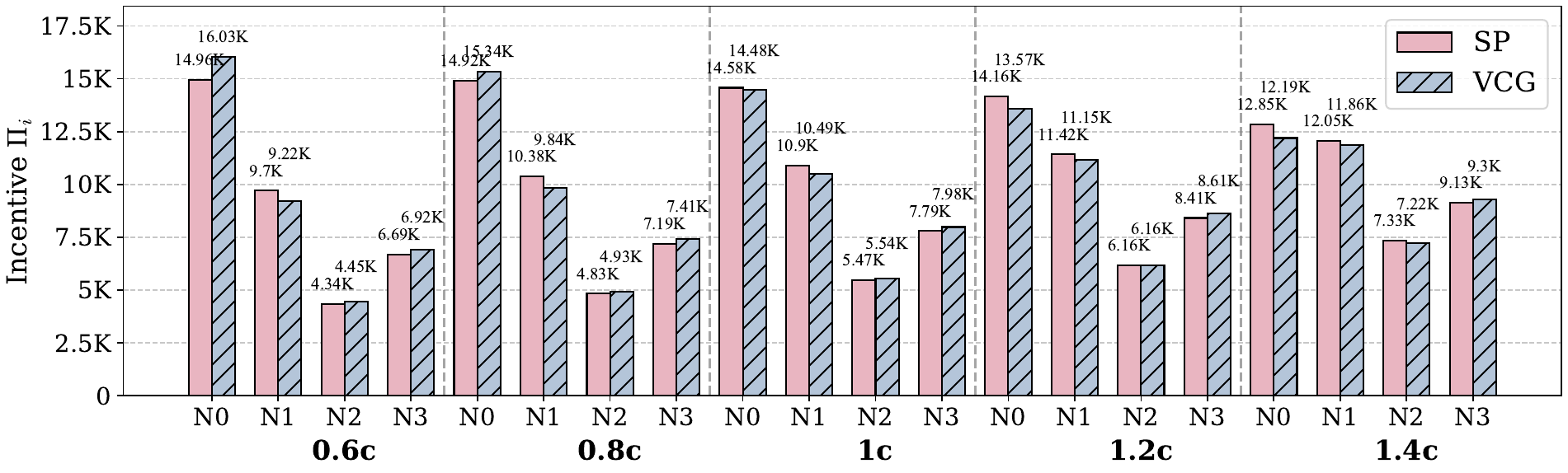}
        \subcaption{Congestion parameter $c_0$}
        \label{subfig:supp_c0}
    \end{minipage}
     \caption{Variation of incentives across mechanisms with different parameters in well-balanced network}
    \label{fig:supp0}
\end{figure}

Figure \ref{fig:supp1} illustrates the impact of different parameters variations on the incentives received by participants in a semi-balanced network. The primary conclusion is that the incentive payments provided by the SP mechanism exceed those under the VCG mechanism. Taking Figure \ref{fig:supp1}(\subref{subfig:supp_d1}) as an example, the incentive payments received by each participant increase with rising demand parameters, which is an intuitive result of demand expansion. It can also be observed that participant N1 receives higher payments compared to other participants. This is attributed to N1’s positional advantage in the path to M1, resulting in relatively lower individual transportation costs. Consequently, N1 can undertake a larger proportion of tasks, leading to greater incentives. Regarding the higher payments under the SP mechanism compared to the VCG mechanism, consider participant N1, whose advantage over M1 lies in having only one fewer edge in its transportation path than participants N0 and N2. This means that even after removing participant N1, the total cost required for the remaining participants to complete the task does not differ substantially from the total cost when N1 is included. Instead, the SP mechanism requires greater incentives to motivate participant N1 to engage in transportation task. Similarly, for participant N0, its advantage over M0 lies in its transportation path containing one fewer edge compared to other participants. This enables N0 to complete a higher proportion of tasks at M0. However, the impact on the system following the removal of node N0 is not sufficiently significant to make the incentives under the VCG mechanism exceed those under the SP mechanism.

\begin{figure}[!ht]
    \centering
    \begin{minipage}[t]{0.98\textwidth} 
        \centering
        \includegraphics[height=50mm, width=150mm]{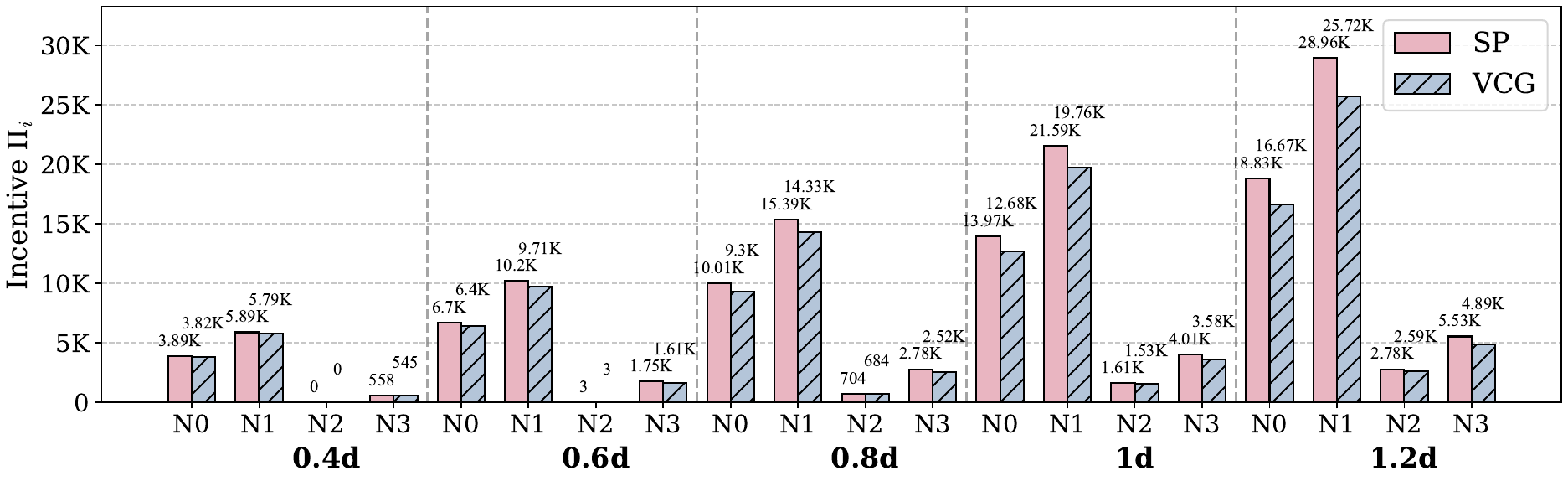}
        \subcaption{Demand parameter $\boldsymbol{d}$}
        \label{subfig:supp_d1}
    \end{minipage}
    
    \vspace{3mm}
    
    \begin{minipage}[t]{0.98\textwidth} 
        \centering
        \includegraphics[height=50mm, width=150mm]{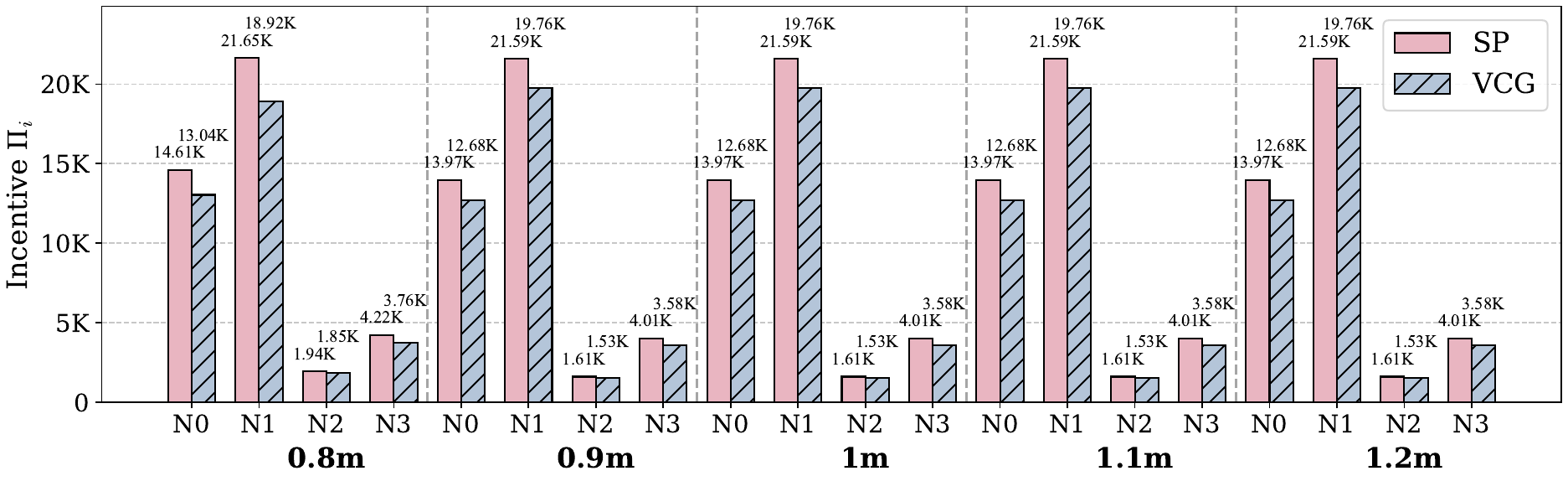}
        \subcaption{Constraint parameter $\boldsymbol{m}$}
        \label{subfig:supp_m1}
    \end{minipage}
    
    \vspace{3mm}
    
    \begin{minipage}[t]{0.98\textwidth} 
        \centering
        \includegraphics[height=50mm, width=150mm]{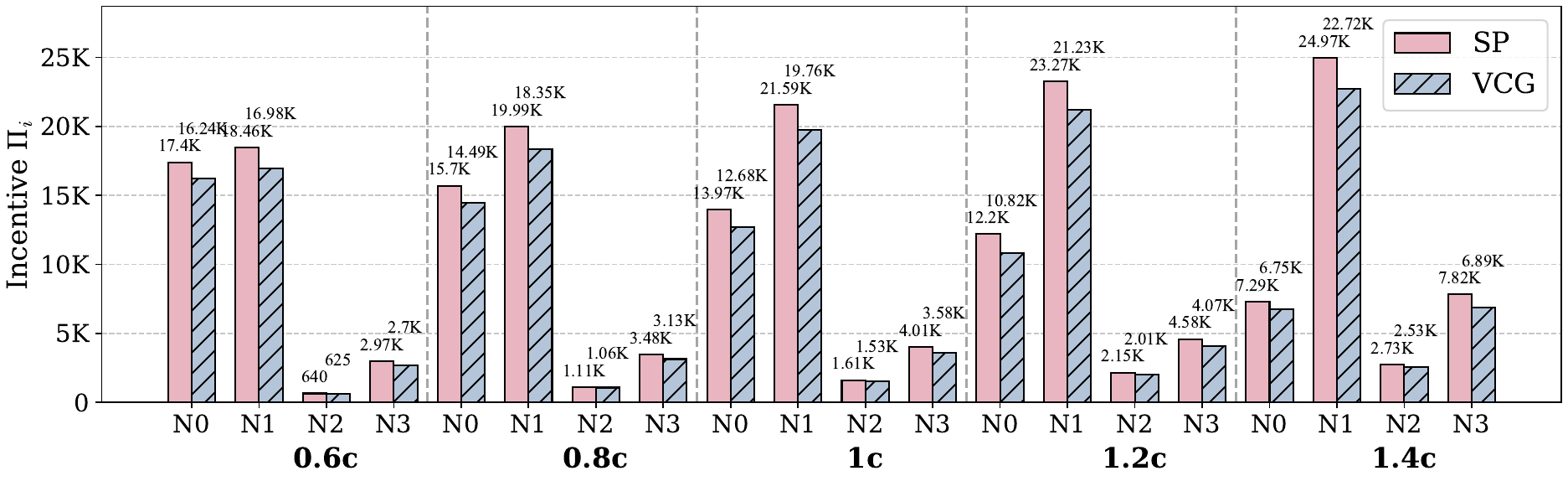}
        \subcaption{Congestion parameter $c_0$}
        \label{subfig:supp_c1}
    \end{minipage}
     \caption{Variation of incentives across mechanisms with different parameters in semi-balanced network}
    \label{fig:supp1}
\end{figure}

Figure \ref{fig:supp2} illustrates the impact of parameter variations on the incentives received by participants in a unbalanced network. In contrast to the semi-balanced network, the main finding under the unbalanced network structure is that the incentive payments under the VCG mechanism exceed those under the SP mechanism. The key reason for this is that participant N1's absolute path advantage towards both M1 and M2 within the network results in a substantial increase in the total cost required for the remaining participants to complete tasks when N1 is removed, compared to scenarios where N1 is included. Consequently, the critical importance of participant N1 leads to significantly higher incentives allocated to it under the VCG mechanism than under the SP mechanism.

Meanwhile, from Figure \ref{fig:supp2}(\subref{subfig:supp_d2})-\ref{fig:supp2}(\subref{subfig:supp_m2}), we observe that under certain conditions, the incentives allocated to node N2 increase significantly—for instance, when task demand $\boldsymbol{d}$ rises to 1.2 times or capacity constraints $\boldsymbol{m}$ tighten to 0.8 times. This occurs because although node N1 possesses an absolute distance advantage, it becomes constrained by its own capacity limitations as parameters shift. To minimize system costs, the mechanism begins incentivizing the suboptimal node N2 (whose criticality increases), thereby encouraging its participation in task execution. Figure \ref{fig:supp2}(\subref{subfig:supp_c2}) reveals that as the congestion coefficient $c_0$ increases, the incentives allocated to node N2 exhibit a modest rise. This phenomenon arises because a higher congestion coefficient force N1 to incur higher congestion costs per unit shipped for the same task volume. When these costs surpass a certain threshold, the system strategically shifts some tasks to N2 to alleviate congestion costs on N0—since distributing tasks reduces localized congestion costs. Consequently, incentives for N1 increase accordingly.

\begin{figure}[!ht]
    \centering
    \begin{minipage}[t]{0.98\textwidth} 
        \centering
        \includegraphics[height=50mm, width=150mm]{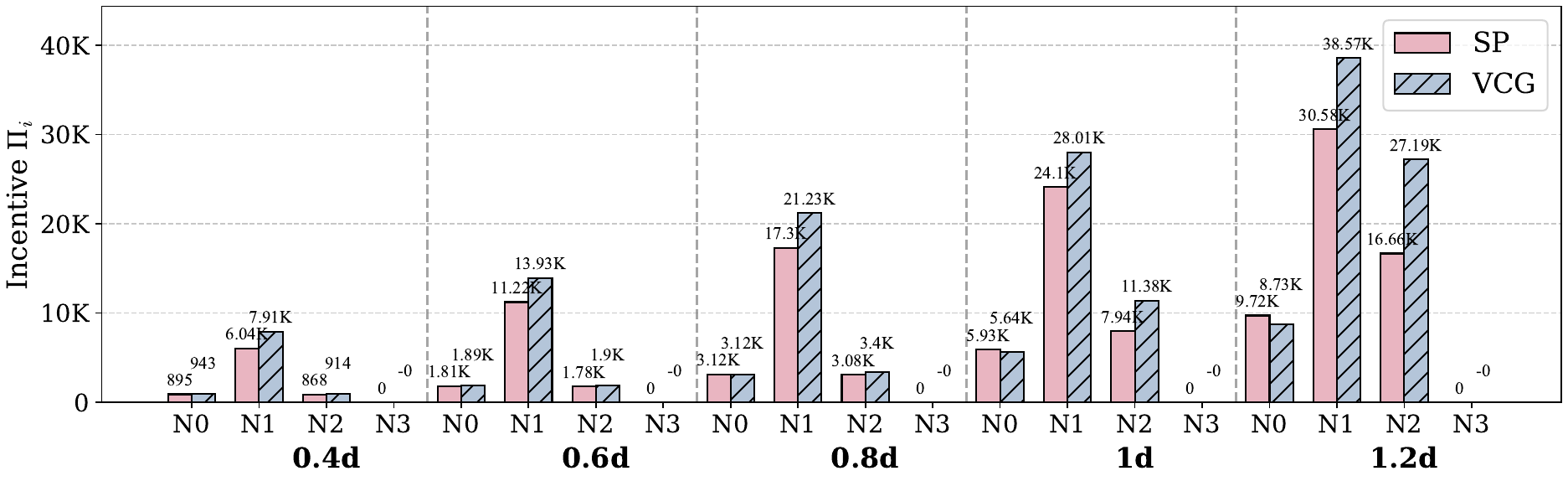}
        \subcaption{Demand parameter $\boldsymbol{d}$}
        \label{subfig:supp_d2}
    \end{minipage}
    
    \vspace{3mm}
    
    \begin{minipage}[t]{0.98\textwidth} 
        \centering
        \includegraphics[height=50mm, width=150mm]{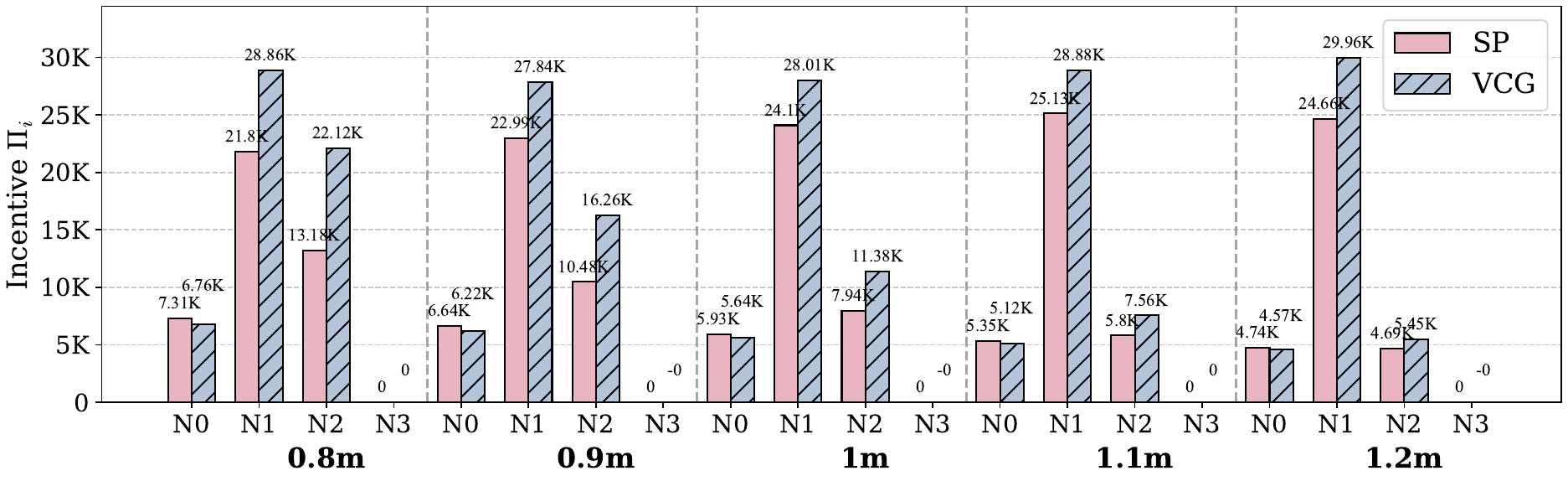}
        \subcaption{Constraint parameter $\boldsymbol{m}$}
        \label{subfig:supp_m2}
    \end{minipage}
    
    \vspace{3mm}
    
    \begin{minipage}[t]{0.98\textwidth} 
        \centering
        \includegraphics[height=50mm, width=150mm]{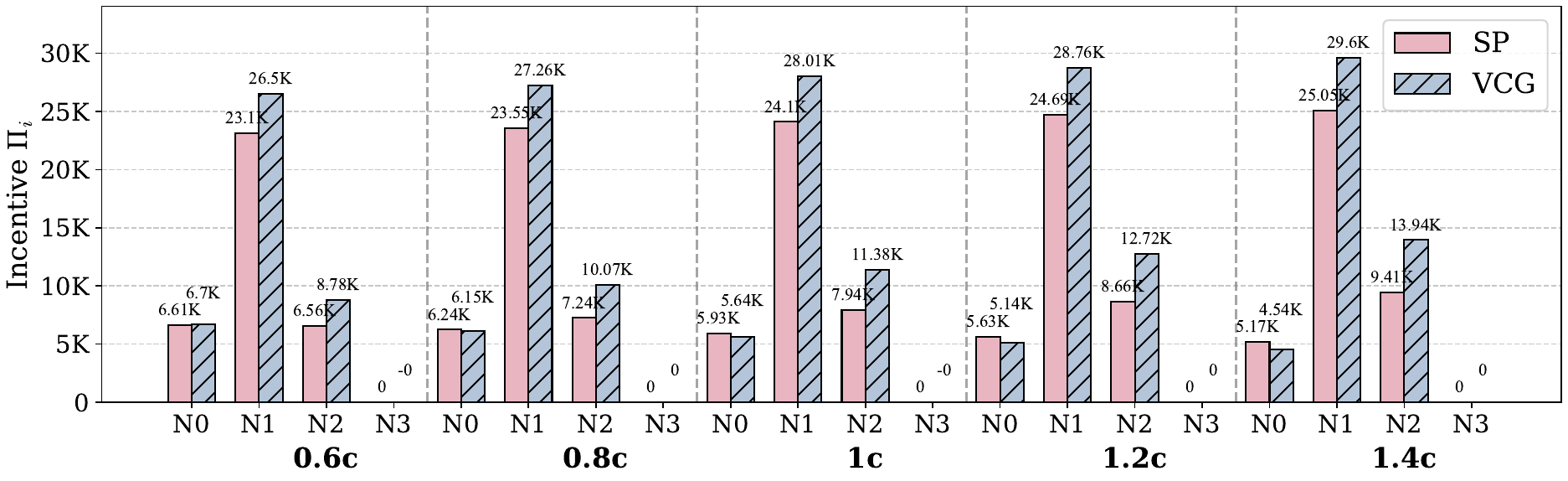}
        \subcaption{Congestion parameter $c_0$}
        \label{subfig:supp_c2}
    \end{minipage}
     \caption{Variation of incentives across mechanisms with different parameters in unbalanced network}
    \label{fig:supp2}
\end{figure}

Table \ref{tab:supp_incentive1}-\ref{tab:supp_incentive2} presents the differences and the percentage of the differences between the two mechanisms under various parameters in the semi-balanced and unbalanced network, respectively. The results demonstrate that in semi-balanced networks, the total incentive payments under the SP mechanism consistently exceeds that under the VCG mechanism, with the percentage difference hovering around 10\%. Conversely, in unbalanced networks, the total incentive payments under the SP mechanism is consistently lower than under the VCG mechanism, exhibiting a percentage difference typically around 15\% and reaching up to 27\% in certain cases. Furthermore, the disparity between the total incentive payments under the two mechanisms does not necessarily change linearly with parameter variations due to the convex structure of the problem. This significant disparity starkly contrasts with the conclusion presented in the main text that the incentive effects of the two mechanisms are essentially indistinguishable in balanced networks. It indicates a strong correlation between the incentive properties of the mechanisms and the underlying network structure.

\begin{table}[!ht]
  \centering
  \small
  \caption{Difference in incentive payments across parameters in semi-balanced network}
  \setstretch{1.1}
    \begin{tabular}{ccccccccc}
    \hline
    \makecell[c]{Task\\Demand} & SP-VCG & Pct. & \makecell[c]{Individual \\Constraints} & SP-VCG & Pct. & \makecell[c]{Congestion \\ Parameter} & SP-VCG & Pct. \\
    \hline
    $0.4\boldsymbol{d}$  & 174.77   & 2\%  & 0.8$\boldsymbol{m}$ & 4848.68   & 11\%  & 0.6$c_0$ & 2931.25  & 7\%\\
    $0.6\boldsymbol{d}$  & 927.92   & 5\% & 0.9$\boldsymbol{m}$ & 3630.03   &   9\% & 0.8$c_0$ & 3243.53 & 8\% \\
    0.8$\boldsymbol{d}$    &  2049.56 & 7\%  & $\boldsymbol{m}$     & 3630.03   & 9\% & $c_0$  & 3630.03 & 9\% \\
    $\boldsymbol{d}$  & 3630.03  & 9\%  & 1.1$\boldsymbol{m}$  & 3630.03     & 9\% & 1.2$c_0$ & 4063.15 &  10\% \\
    1.2$\boldsymbol{d}$  & 6223.75  & 11\%  & 1.2$\boldsymbol{m}$  & 3630.03  &  9\% & 1.4$c_0$ & 3916.04  & 9\% \\
    \hline
    \end{tabular}%
  \label{tab:supp_incentive1}%
\end{table}%

\begin{table}[!ht]
  \centering
  \small
  \caption{Difference in incentive payments across parameters in unbalanced network}
  \setstretch{1.1}
    \begin{tabular}{ccccccccc}
    \hline
    \makecell[c]{Task\\Demand} & SP-VCG & Pct. & \makecell[c]{Individual \\Constraints} & SP-VCG & Pct. & \makecell[c]{Congestion \\ Parameter} & SP-VCG & Pct. \\
    \hline
    $0.4\boldsymbol{d}$  & -1955.13   & 20\%  & 0.8$\boldsymbol{m}$ & -15468.80   & 27\%  & 0.6$c_0$ & -5704.80  & 14\%\\
    $0.6\boldsymbol{d}$  & -2916.15   & 16\% & 0.9$\boldsymbol{m}$ & -10208.90  &   20\% & 0.8$c_0$ & -6442.50  & 15\% \\
    0.8$\boldsymbol{d}$    &  -4249.72  & 15\%  & $\boldsymbol{m}$     & -7052.69   & 16\% & $c_0$  & -7052.69 & 16\% \\
    $\boldsymbol{d}$  & -7052.69  & 16\%  & 1.1$\boldsymbol{m}$  & -5279.35    & 13\% & 1.2$c_0$ & -7647.44 &  16\% \\
    1.2$\boldsymbol{d}$  & -17530.90  & 24\%  & 1.2$\boldsymbol{m}$  & -5894.77  &  15\% & 1.4$c_0$ & -8445.68  & 18\% \\
    \hline
    \end{tabular}%
  \label{tab:supp_incentive2}%
\end{table}%

\end{document}